\documentclass{amsart}
\usepackage{amssymb}
\usepackage[utf8]{inputenc}
\usepackage[english]{babel}

\theoremstyle{plain}
\newtheorem{theorem}{Theorem}[section]
\newtheorem{corollary}[theorem]{Corollary}
\newtheorem{lemma}[theorem]{Lemma}
\newtheorem{proposition}[theorem]{Proposition}
\theoremstyle{property}

\theoremstyle{definition}
\newtheorem{definition}[theorem]{Definition}
\newtheorem{remark}[theorem]{Remark}
\newtheorem{acknowledgement}[theorem]{Acknowledgement}

\numberwithin{equation}{section}

\input{tcilatex}

\begin{document}
\title[The CR Yamabe equation]{Positive mass theorem and the CR Yamabe
equation on 5-dimensional contact spin manifolds}
\author{Jih-Hsin Cheng}
\address{Institute of Mathematics, Academia Sinica and National Center for
Theoretical Sciences, Taipei, Taiwan, R.O.C.}
\email{cheng@math.sinica.edu.tw}
\author{Hung-Lin Chiu}
\address{Department of Mathematics, National Tsing-Hua University, Hsinchu,
Taiwan, R.O.C.}
\email{hlchiu@math.nthu.edu.tw}
\thanks{}

\begin{abstract}
We consider the CR Yamabe equation with critical Sobolev exponent on a
closed contact manifold $M$ of dimension $2n+1.$ The problem of finding
solutions with minimum energy has been resolved for all dimensions except
dimension 5 ($n=2$). In this paper we prove the existence of minimum energy
solutions in the 5-dimensional case when $M$ is spin. The proof is based on
a positive mass theorem built up through a spinorial approach.
\end{abstract}

\maketitle


\section{Introduction and statement of the results}

On an odd dimensional manifold $M$, contact structure $\xi $ is a natural
geometric structure to consider. Moreover, a contact manifold $(M,\xi )$
arising as the boundary of a domain in $\mathbb{C}^{n+1}$ (or a complex
manifold) inherits a complex structure, called $CR$ (Cauchy-Riemann)
structure. The CR structure essentially reflects or controls the complex
structure of the inside domain. We can talk about abstract $CR$ structures
on a contact manifold $M$ (see the Appendix). We consider the following CR
Yamabe equation with critical Sobolev exponent (see (\ref{YE}) and notations
in the Appendix):%
\begin{equation}
(2+\frac{2}{n})\Delta _{b}u+Wu=u^{1+\frac{2}{n}}\text{ on }M.  \label{1-1}
\end{equation}%
\noindent Here $\Delta _{b}$ and $W$ denote the (negative) sublaplacian and
the Tanaka-Webster scalar curvature, respectively (see (\ref{SubL}) and (\ref%
{Wsc}) in the Appendix). There is a variational formulation for equation (%
\ref{1-1}). Namely the energy is provided by the following CR-Sobolev
quotient (see (\ref{YMJ}) in the Appendix): 
\begin{equation}
Q(v):=\frac{\int_{M}((2+\frac{2}{n})|\nabla _{b}v|^{2}+Wv^{2})dV_{\theta }}{%
(\int_{M}v^{2+2/n}dV_{\theta })^{n/(n+1)}}\text{ for }v>0\text{ smooth.}
\label{1-2}
\end{equation}%
\noindent The main goal of this paper is to find a solution $u$ to (\ref{1-1}%
) with minimum energy $Q(u)$ equal to%
\begin{equation}
\inf_{v>0,\text{ smooth}}Q(v)=:\mathcal{Y}(M,J)  \label{1-3}
\end{equation}%
\noindent on a closed (compact with no boundary) 5-dimensional contact and
spin manifold provided the underlying CR structure $J$ is spherical with the
CR Yamabe constant $\mathcal{Y}(M,J)$ $>$ $0$ (see Theorem \ref{YMP}). For
this CR Yamabe minimizer problem, i.e. finding a solution to (\ref{1-1})
(assuming $\mathcal{Y}(M,J)$ $>$ $0;$ the cases $\mathcal{Y}(M,J)$ $\leq $ $%
0 $ are easier$)$ with minimum energy, we have been able to resolve for all
dimensions except dimension 5. Let us give a brief history about this
problem below.

There has been a far-reaching analogy between conformal and CR geometries.
Along the approaches used in conforml geometry by H. Yamabe, N. Trudinger
and T. Aubin, in 1987 D. Jerison and J. Lee \cite{JL} showed the analogous
results in CR geometry. That is, the CR Yamabe constant $\mathcal{Y}(M,J)$
depends only on the CR structure $J$ of $M$ and $\mathcal{Y}(M,J)$ $\leq $ $%
\mathcal{Y}(S^{2n+1},\hat{J})$, where $(S^{2n+1},$ $\hat{J})$ is the
standard CR sphere with the induced CR structure $\hat{J}$ from $\mathbb{C}%
^{n+1}$. In addition, if $\mathcal{Y}(M,J)<\mathcal{Y}(S^{2n+1},\hat{J})$,
then $\mathcal{Y}(M,J)$ is attained for some positive $C^{\infty }$ function 
$u$ (by the compactness of solutions to a family of approximate equations),
hence the CR Yamabe minimizer problem for $(M,J)$ is solvable.

Recall that a CR manifold is called spherical if it is locally CR equivalent
to the standard CR sphere $(S^{2n+1},$ $\hat{J})$. In the case that
dimension $2n+1\geq 5$ and ($M,J)$ is not spherical, Jerison and Lee (\cite%
{JL1}) showed that $\mathcal{Y}(M,J)<\mathcal{Y}(S^{2n+1},\hat{J})$ by a
test function estimate. For the remaining cases: either (i) $\dim M=3$ or
(ii) $\dim M$ $\geq $ $5$ and $M$ is spherical, we need a positive mass
theorem to show that $\mathcal{Y}(M,J)<\mathcal{Y}(S^{2n+1},\hat{J})$ unless 
$M$ is CR equivalent to the standard CR sphere. When $\dim M=3$, this was
shown by Malchiodi, Yang and one of the authors (\cite{CMY}) (the condition
that the CR Paneitz operator of $M$ is nonnegative turns out to be
equivalent to the embeddability of $J$ \cite{Tak}$)$ \ Finally when $\dim M$ 
$\geq $ $5$ and $M$ is spherical, this was finished by Yang and the authors (%
\cite{CCY}) through showing that the developing map is injective. However in
the case $\dim M$ $=$ $5$, we need an extra condition on the growth rate of
the Green's function on $\tilde{M},$ the universal cover of $M$. So in the
case $\dim M$ $=$ $5$, the CR positive mass theorem is not really completed.
In this paper, we showed that for $\dim M$ $=$ $5,$ $M$ being spherical, if
in addition $M$ has a spin structure, then we have a CR positive mass
theorem, and hence the CR Yamabe minimizer problem is solvable (see Theorem %
\ref{YMP} below). For an asymptotically flat pseudohermitian manifold $%
(N,J,\theta )$ (see (\ref{af1})), we can talk about the p-mass $m(J,\theta )$
(see (\ref{af10})).

\begin{theorem}
\label{PMT} Suppose that $(N,J,\theta )$ is an asymptotically flat,
pseudohermitian and spin manifold of dimension 5. Assume that $J$ is
spherical and $(N,J,\theta )$ has the Tanaka-Webster scalar curvature $W\geq
0$. Then the p-mass $m(J,\theta )\geq 0$. Moreover, $m(J,\theta )=0$ if and
only if $(N,J,\theta )$ is isomorphic to the Heisenberg group $(H_{2},%
\mathring{J},\mathring{\theta})$.
\end{theorem}

\begin{corollary}
\label{PMT'} Suppose that $(M,\xi )$ is a closed (compact with no boundary),
contact and spin manifold of dimension 5. Assume that $J$ is a spherical $CR$
structure on $(M,\xi )$ with $\mathcal{Y(}M,J)$ $>$ $0.$ Then the associated
p-mass $m(J,\theta )\geq 0$. Moreover, $m(J,\theta )=0$ if and only if $%
(M,J) $ is $CR$ equivalent to the standard $CR$ $5$-sphere.
\end{corollary}

The proof of Theorem \ref{PMT} is based on a Weitzenbock-type formula:%
\begin{equation}
D_{\xi }^{2}=\nabla ^{\ast }\nabla +W-2\sum_{\beta =1}^{n}e_{\beta
}e_{n+\beta }\nabla _{T}  \label{1-4}
\end{equation}

\noindent where $D_{\xi }$ and $\nabla $\ denote the contact Dirac operator
and spin connection respectively\ (see (\ref{WF}) in Section \ref{SSCB}).
The term involving $\nabla _{T}$ ($T$ is the Reeb vector field associated to
the contact form $\theta ;$ see the Appendix for its definition) causes
difficulty to solve the Dirac equation $D_{\xi }^{2}\psi $ $=$ $0$ in
general. However, in the case of dimension 5 ($n=2$) we observe the
following algebraic fact for Clifford multiplication:%
\begin{equation}
\sum_{\beta =1}^{2}e_{\beta }e_{2+\beta }=0\text{ on }S^{+}(2n)\overset{n=2}{%
=}S^{+}(4)  \label{1-5}
\end{equation}

\noindent where $S^{+}(2n)$ denotes the space of positive spinors (see (\ref%
{keyfor})). So for the dimension equal to 5, the last term in (\ref{1-4})
disappears when acting on (sections of) $\mathbb{S}^{+},$ the bundle of
positive spinors (we often do not distinguish between the bundle $\mathbb{S}%
^{+}$ and the space of its sections $\Gamma (\mathbb{S}^{+})$ if no
confusion can occur$).$ It follows that $D_{\xi }^{2}$ is subelliptic on $%
\mathbb{S}^{+}$ and hence we can find a spinor field $\psi $ $\in $ $\mathbb{%
S}^{+}$ such that $D_{\xi }^{2}\psi =0$ and $\psi $ tends to a constant
spinor at the infinity (see Corollary \ref{Sol}). Applying (\ref{1-4}) to
this spinor field $\psi $ and integrating after taking the inner product
with $\psi ,$ we then pick up (a positive multiple of) the p-mass $%
m(J,\theta )$ from the boundary integral and obtain a Witten-type formula
for $m(J,\theta )$ (see (\ref{WTF}) and (\ref{af12}). So nonnegativity of $%
m(J,\theta )$ follows. To characterize $m(J,\theta )=0$ we need a trick,
among others, inspired by the idea of Schoen and Yau \cite{SY} to show the
torsion vanishes (see Lemma \ref{le02}). To prove Corollary \ref{PMT'} we
first blow up the closed $M$ at a point $p$ by the Green's function of the
CR invariant sublaplacian $(2+\frac{2}{n})\Delta _{b}+W$ to get an
asymptotically flat pseudohermitian manifold $N.$ Then we can apply Theorem %
\ref{PMT} to make the conclusion.

To solve the CR Yamabe minimizer problem, we need a test function estimate
(see Theorem \ref{esofya} in Section 5). The idea was rooted in an argument
used by Schoen in \cite{S} for the Riemannian case. For the CR case, it was
first treated by Z. Li (\cite{Li}) in an unpublished draft. We reorganize
his construction of a family of test functions $\phi _{\beta }$ and clarify
the arguments at some points so that the CR-Sobolev quotient $Q(\phi _{\beta
})$ is less than $\mathcal{Y}(S^{2n+1},\hat{J})$ minus a positive multiple
of the p-mass modulo the terms of higher decay rate (see (\ref{TFE})). \
From Theorem \ref{esofya} the result below follows easily.

\begin{theorem}
\label{YMP} Suppose that $(M,\xi )$ is a closed (compact with no boundary),
contact and spin manifold of dimension 5. Assume that $J$ is a spherical CR
structure on $(M,\xi )$ with $\mathcal{Y(}M,J)$ $>$ $0$. Then the CR Yamabe
minimizer problem is solvable, i.e. we can find a solution to (\ref{1-1})
with minimum energy.
\end{theorem}

In Section 6 we show that the connected sum of finitely many (duplication
allowed) 5-manifolds chosen arbitrarily from the set consisting of $S^{5}/%
\mathbb{Z}_{p},$ $p:$ odd integer, $S^{4}\times S_{(a)}^{1},$ $a>1$ and $%
\mathbb{RP}^{5}$ $\sharp $ $\mathbb{RP}^{5}$ is still a closed, contact spin
5-manifold which admit a spherical CR structure with positive CR Yamabe
constant (see Proposition \ref{prop-ex}). For instance, the following is
such an example:%
\begin{equation*}
m_{1}(S^{5}/\mathbb{Z}_{p_{1}})\#l_{1}(S^{4}\times S_{(a)}^{1})\#m_{2}(S^{5}/%
\mathbb{Z}_{p_{2}})\#l_{2}(\mathbb{RP}^{5}\sharp \mathbb{RP}^{5})
\end{equation*}

\noindent where $m_{j},$ $l_{j},$ $j=1,2$ are some positive integers and $%
p_{j},$ $j=1,2$ are odd integers. So on such a 5-dimensional closed, contact
and spin manifold, one can find a solution to the CR Yamabe equation (\ref%
{1-1}) with minimal energy for any spherical CR structure (by Theorem \ref%
{YMP} for $\mathcal{Y(}M,J)$ $>$ $0$ and Jerison-Lee \cite{JL} for $\mathcal{%
Y(}M,J)$ $\leq $ $0$ ($<$ $\mathcal{Y}(S^{5},\hat{J}))).$ It should be
mentioned that by a different approach equation (\ref{1-1}) always has a
solution (\cite{Ga}, \cite{GY}), but the solution may not be a minimizer for
the CR-Sobolev quotient $Q(v)$ (\ref{1-2}). In fact, in the case of
dimension 3 \cite{CMY1} we have exotic 3-spheres (called Rossi spheres) with
negative p-mass, on which the infimum of $Q(v)$ (with $n=1)$ is not attained
while the solution to (\ref{1-1}) exists according to the above-mentioned
result.

\begin{acknowledgement}
The authors would like to thank the Ministry of Science and Technology of
Taiwan for the support: grant no. MOST 108-2115-M-001-010 and grant no. MOST
109-2115-M-007-004-MY3 respectively. We would also like to thank Paul Yang
and Andrea Malchiodi for many discussions and constant interest in our work
during the preparation.
\end{acknowledgement}

\section{Spin structure on a contact bundle}

\label{SSCB}

Let $(M^{2n+1},\xi )$ be a contact manifold of dimension $2n+1$ with a
coorientable (i.e. $TM/\xi $ is trivial) contact structure $\xi .$ Take a
(global) contact form $\theta $ (exists by coorientation of $\xi $) such
that $\xi =\ker {\theta }$ (and $\theta \wedge (d\theta )^{n}$ $\neq $ $0$).
Let $(M^{2n+1},J,\theta )$ be a pseudohermitian manifold (see the Appendix
for an introduction). The Levi metric $L_{\theta}$ on $\xi $ is a Riemannian
structure on $\xi $ defined by $L_{\theta}(X,Y)=d\theta (X,JY)=\frac{1}{2}%
d\theta (X\wedge JY)\ $for all$\ X,Y\in \xi .$ Then $L_{\theta}(JX,JY)=L_{%
\theta}(X,Y)$ for all $X,Y\in \xi $. We equip the vector bundle $\xi $ with
this Riemannian structure $L_{\theta}$.

\subsection{Spin structure on $\protect\xi$:}

Let $SO(\xi )$ be the oriented orthonormal frame bundle. A spin structure $%
Spin(\xi )$ on $\xi $ is a principal $Spin(2n)$-bundle such that $Spin(\xi
)\times _{\rho }SO(2n)=SO(\xi )$ where $\rho :Spin(2n)\rightarrow SO(2n)$ is
the standard $2$-sheeted covering.

\begin{lemma}
Suppose $(M,\xi )$ is a coorientable contact manifold. Then the contact
bundle $\xi $ is spin if and only if the tangent bundle $TM$ is spin.
\end{lemma}

\begin{proof}
It is equivalent to showing that the second Stiefel-Whitney class $w_{2}(\xi
)=0$ if and only if the second Stiefel-Whitney class $w_{2}(TM)=0$. Let $T$
be the Reeb vector field relative to $\theta $. Then $T$ defines a trivial
line bundle $\mathbb{R}T$ over $M$. We have $TM=\xi \oplus \mathbb{R}T$. By
the Whitney product formula, we have 
\begin{equation*}
w_{2}(TM)=w_{2}(\xi )+w_{1}(\xi )w_{1}(\mathbb{R}T)+w_{2}(\mathbb{R}%
T)=w_{2}(\xi ).
\end{equation*}%
The proof follows.
\end{proof}

\subsection{Spinor bundle}

Let $\Lambda _{\mathbb{R}}^{k},\Lambda _{\mathbb{C}}^{k}$ denote the real
and complex vector spaces respectively, spanned by $\{\omega ^{j_{1}}\wedge
\cdots \wedge \omega ^{j_{k}}|\ 1\leq j_{1}<\cdots <j_{k}\leq n\}$ (view
symbols $\omega ^{1},$ $..,$ $\omega ^{n}$ as independent vectors). Let 
\begin{equation*}
\Lambda _{\mathbb{R}}^{\ast }:=\sum_{k}\Lambda _{\mathbb{R}}^{k},\ \ \Lambda
_{\mathbb{C}}^{\ast }:=\sum_{k}\Lambda _{\mathbb{C}}^{k}.
\end{equation*}%
Define the linear maps $\epsilon _{j}$ and $\iota _{j}$ by%
\begin{equation*}
\epsilon _{j}:\Lambda _{\mathbb{R}}^{\ast }\rightarrow \Lambda _{\mathbb{R}%
}^{\ast },\ \ \omega ^{j_{1}}\wedge \cdots \wedge \omega ^{j_{k}}\mapsto
\omega ^{j}\wedge \omega ^{j_{1}}\wedge \cdots \wedge \omega ^{j_{k}};
\end{equation*}%
\begin{equation*}
\iota _{j}:\Lambda _{\mathbb{R}}^{\ast }\rightarrow \Lambda _{\mathbb{R}%
}^{\ast },\ \ \omega ^{j_{1}}\wedge \cdots \wedge \omega ^{j_{k}}\mapsto
\sum_{s=1}^{k}(-1)^{s-1}\delta _{jj_{s}}\omega ^{j_{1}}\wedge \cdots \wedge 
\widehat{\omega }^{j_{s}}\wedge \cdots \wedge \omega ^{j_{k}},
\end{equation*}%
\noindent and $E_{a},\ 1\leq a\leq 2n,$ by 
\begin{equation*}
E_{2j-1}=\epsilon _{j}-\iota _{j},\ \ \ E_{2j}=i(\epsilon _{j}+\iota _{j}),\
\ \ 1\leq j\leq n.
\end{equation*}

\noindent It follows that 
\begin{equation*}
E_{a}E_{b}+E_{b}E_{a}=-2\delta _{ab},\ \ 1\leq a,b\leq 2n.
\end{equation*}%
\noindent Hence $\{E_{a}|1\leq a\leq 2n\}$ spans the Clifford algebra $%
C_{2n}(-1)$. This defines the algebra isomorphism of $C_{2n}(-1)\otimes 
\mathbb{C}$ and $End_{\mathbb{C}}(\Lambda _{\mathbb{C}}^{\ast }(n))$ through
the action of $C_{2n}(-1)$ on $\Lambda _{\mathbb{C}}^{\ast }(n)$, which is
denoted by $\mathcal{N}$. Let $S(2n):=\Lambda _{\mathbb{C}}^{\ast }(n),\
S^{+}(2n):=\Lambda _{\mathbb{C}}^{even}(n)$ and $S^{-}(2n):=\Lambda _{%
\mathbb{C}}^{odd}(n)$. Then $S(2n)$ is an irreducible $C_{2n}(-1)$-module
and $S^{\pm }(2n)$ are two irreducible $Spin(2n)$-modules with dimension $%
2^{n-1}$. Confining the action of $\mathcal{N}$ to $Spin(2n)$ $\subset $ $%
C_{2n}(-1),$ we define the vector bundles 
\begin{equation}
\mathbb{S}:=Spin(\xi )\times _{\mathcal{N}}S(2n),\ \mathbb{S}^{\pm
}:=Spin(\xi )\times _{\mathcal{N}}S^{\pm }(2n)  \label{spinor}
\end{equation}%
\noindent called spinor bundles. For $n$ $=$ $2,\ \Lambda _{\mathbb{C}%
}^{\ast }(2)$ $=\text{span}\{1,\omega ^{1},\omega ^{2},\omega ^{1}\wedge
\omega ^{2}\}$, we have $E_{1}E_{3}+E_{2}E_{4}=-2(\iota _{1}\epsilon
_{2}+\epsilon _{1}\iota _{2})$ and hence, by a straightforward computation 
\begin{equation}
\begin{split}
(E_{1}E_{3}+E_{2}E_{4})& 1=0 \\
(E_{1}E_{3}+E_{2}E_{4})& \omega ^{1}\wedge \omega ^{2}=0 \\
(E_{1}E_{3}+E_{2}E_{4})& \omega ^{1}=2\omega ^{2} \\
(E_{1}E_{3}+E_{2}E_{4})& \omega ^{2}=-2\omega ^{1}.
\end{split}
\label{keyformula}
\end{equation}%
Let $\sigma $ be a local section of $Spin(\xi )$ such that $\rho _{\ast
}(\sigma )=\{e_{1},\cdots ,e_{2n}\}$ be an orthonormal frame filed of $\xi $
and $e_{n+\beta }=Je_{\beta },\ 1\leq \beta \leq n$. By the first two
formulas in (\ref{keyformula}), we immediately obtain the following result.

\begin{lemma}
Suppose that $n=2$. Then it holds that%
\begin{equation}
\sum_{\beta =1}^{2}e_{\beta }e_{2+\beta }=0\ \ \text{on}\ S^{+}(4).
\label{keyfor}
\end{equation}
\end{lemma}

\begin{remark}
\label{r2-2} We observe that $\sum_{\beta =1}^{n}e_{\beta }e_{n+\beta }$ $%
\neq $ $0$ on $S^{+}(2n)$ unless $n=2.$
\end{remark}

\subsection{Spin connection}

Let $\sigma $ be a local section of $Spin(\xi )$ such that $\rho _{\ast
}(\sigma )=\{e_{1},\cdots ,e_{2n}\}$ is an orthonormal frame field of $\xi $
with $e_{n+\beta }=Je_{\beta },\ 1\leq \beta \leq n$. Define 
\begin{equation}
\omega _{\sigma }:=\frac{1}{4}\sum_{a,b=1}^{2n}\omega _{b}{}^{a}e_{a}e_{b},
\label{spicon}
\end{equation}%
where $\omega _{b}{}^{a}$ is the pseudohermitian connection forms with
respect to $\{e_{1},\cdots ,e_{2n}\}$ (for the complex version, see the
Appendix for an explanation). Then (cf. \cite{Yu} for such an expression)%
\begin{equation}
\varpi =\{\omega _{\sigma }\}  \label{spicon-1}
\end{equation}%
is a connection of $Spin(\xi )$ with associated spin connection $\nabla $
acting on spinors (or spinor fields), sections of $\mathbb{S}:=Spin(\xi
)\times _{\mathcal{N}}S(2n),$ denoted as $\Gamma (\mathbb{S})$ (see (\ref%
{spinor})). The curvature form $\Omega _{\sigma }$ is defined by 
\begin{equation}
\Omega _{\sigma }=d\omega _{\sigma }+\frac{1}{2}[\omega _{\sigma },\omega
_{\sigma }]  \label{curfor}
\end{equation}%
We have 
\begin{equation}
\Omega _{\sigma }=\frac{1}{4}\sum_{a,b=1}^{2n}\Omega _{b}{}^{a}e_{a}e_{b},
\end{equation}%
where 
\begin{equation}
\Omega _{b}{}^{a}=d\omega _{b}{}^{a}+\omega _{b}{}^{c}\wedge \omega
_{c}{}^{a}.
\end{equation}

\noindent For $X,$ $Y$ $\in $ $\xi $ let $R_{XY}$ denote the curvature
operator of the spin connection $\nabla $ acting on spinors while $%
R_{XY}^{p.h.}$ denotes the curvature operator of the pseudohermitian
connection $\nabla ^{p.h.}.$

\begin{lemma}
\begin{equation}  \label{currep}
R_{XY}=\frac{1}{4}\sum_{a,b=1}^{2n}\big<R^{p.h.}_{XY}(e_{a}),e_{b}\big>%
e_{a}e_{b}.
\end{equation}
\end{lemma}

\begin{proof}
For $\varphi =(\sigma ,f)$ $\in $ $Spin(\xi )\times _{\mathcal{N}}S(2n)$, we
have 
\begin{equation}
\begin{split}
R_{XY}\varphi & =(\sigma ,\Omega _{\sigma }(X\wedge Y)\cdot f) \\
& =\frac{1}{4}\sum_{a,b=1}^{2n}\Omega _{b}{}^{a}(X\wedge Y)(\sigma
,e_{a}e_{b}\cdot f) \\
& =\frac{1}{4}\sum_{a,b=1}^{2n}\big<R_{XY}^{p.h.}(e_{a}),e_{b}\big>%
e_{a}e_{b}\cdot (\sigma ,f).
\end{split}%
\end{equation}
\end{proof}

\begin{lemma}
The spin connection $\nabla $ defined by (\ref{spicon}) and (\ref{spicon-1})
has the following property: (recall that $\nabla ^{p.h.}$ denotes the
pseudohermitian connection) for $X,Y$ $\in $ $\xi ,$ $\varphi $ $\in $ $%
\Gamma (\mathbb{S})$%
\begin{equation}
\nabla _{X}(Y\cdot \varphi )=(\nabla _{X}^{p.h.}Y)\cdot \varphi +Y\cdot
\nabla _{X}\varphi .  \label{leifor}
\end{equation}
\end{lemma}

\begin{proof}
Let $\omega =\{\omega _{\sigma }\ |\ \sigma \ \text{is a local section of}\
Spin(\xi )\}$. Then locally for $Y=(\sigma ,v),$ $\varphi =(\sigma ,f),$ $%
v\in \mathbb{R}^{2n},$ $f\in \Lambda _{\mathbb{C}}^{\ast }(n)$ we compute%
\begin{equation*}
\begin{split}
\nabla _{X}(Y\cdot \varphi )& =\nabla _{X}\big((\sigma ,v)\cdot (\sigma ,f)%
\big) \\
& =\nabla _{X}(\sigma ,v\cdot f) \\
& =\big(\sigma ,X(v\cdot f)+\omega _{\sigma }(X)(v\cdot f)\big) \\
& =\big(\sigma ,(Xv)\cdot f+v\cdot (Xf)+\omega _{\sigma }(X)(v\cdot f)\big)
\end{split}%
\end{equation*}%
where%
\begin{equation*}
\begin{split}
\omega _{\sigma }(X)(v\cdot f)& =\frac{d}{dt}\Big|_{t=0}\Big(g(t)(v\cdot f)%
\Big),\ \text{here}\ \frac{d}{dt}\Big|_{t=0}g(t)=\omega _{\sigma }(X),\
g(0)=id \\
& =\frac{d}{dt}\Big|_{t=0}\Big((g(t)v)\cdot f\Big) \\
& =\frac{d}{dt}\Big|_{t=0}\Big((g(t)vg(t)^{-1}g(t))\cdot f\Big) \\
& =\frac{d}{dt}\Big|_{t=0}\Big((g(t)vg(t)^{-1}\Big)\cdot f+v\cdot \frac{d}{dt%
}\Big|_{t=0}\Big(g(t))\cdot f\Big) \\
& =(\omega _{\sigma }(X)v-v\omega _{\sigma }(X))\cdot f+v\cdot (\omega
_{\sigma }(X)\cdot f) \\
& =(\omega ^{p.h.}(X)v)\cdot f+v\cdot (\omega _{\sigma }(X)\cdot f.
\end{split}%
\end{equation*}%
(\ref{leifor}) follows.
\end{proof}

The invariant second derivative $\nabla _{V,W}^{2}:\Gamma (\mathbb{S}%
)\rightarrow \Gamma (\mathbb{S})$ is defined by $\nabla _{V,W}^{2}\varphi
:=\nabla _{V}\nabla _{W}\varphi -\nabla _{\nabla _{V}^{p.h.}W}\varphi ,\
\varphi \in \Gamma (\mathbb{S}).$ Here $\nabla _{V}^{p.h.}W$ denotes the
covariant derivative of $W$ along $V$ with respect to the pseudohermitian
connection $\nabla ^{p.h.}$. Recall that $R_{e_{a}e_{b}}$ denotes the
curvature operator with respect to the spin connection $\nabla $.

\begin{lemma}
With the notations above, it holds that on $\Gamma (\mathbb{S})$%
\begin{equation}
\begin{split}
\nabla _{e_{\alpha },e_{n+\alpha }}^{2}-\nabla _{e_{n+\alpha },e_{\alpha
}}^{2}& =R_{e_{\alpha }e_{n+\alpha }}-2\nabla _{T},\ \ \text{for all}\ 1\leq
\alpha \leq n, \\
\nabla _{e_{n+\alpha },e_{\alpha }}^{2}-\nabla _{e_{\alpha },e_{n+\alpha
}}^{2}& =R_{e_{n+\alpha }e_{\alpha }}+2\nabla _{T},\ \ \text{for all}\ 1\leq
\alpha \leq n, \\
\nabla _{e_{a},e_{b}}^{2}-\nabla _{e_{b},e_{a}}^{2}& =R_{e_{a}e_{b}},\ \ 
\text{otherwise}.
\end{split}
\label{ide8}
\end{equation}
\end{lemma}

\begin{proof}
By the definition, we have $\nabla _{e_{a},e_{b}}^{2}-\nabla
_{e_{b},e_{a}}^{2}$ $=$ $\nabla _{e_{a}}\nabla _{e_{b}}$ $-$ $\nabla
_{e_{b}}\nabla _{e_{a}}$ $-$ $\nabla _{\nabla _{e_{a}}^{p.h.}e_{b}-\nabla
_{e_{b}}^{p.h.}e_{a}}.$ This, together with (\ref{ide2}), yields (\ref{ide8}%
).
\end{proof}

\subsection{Weitzenbock-type formula}

We define the contact Dirac operator $D_{\xi }:\Gamma (\mathbb{S}^{\pm
})\rightarrow \Gamma (\mathbb{S}^{\mp })$ by%
\begin{equation}
D_{\xi }{\phi }:=\sum_{a=1}^{2n}e_{a}\cdot \nabla _{e_{a}}\phi  \label{Dxi}
\end{equation}

\begin{theorem}[\textbf{{Weitzenbock-type Formula}}]
In the preceding notations, it holds that%
\begin{equation}
D_{\xi }^{2}=\nabla ^{\ast }\nabla +W-2\sum_{\beta =1}^{n}e_{\beta
}e_{n+\beta }\nabla _{T}  \label{WF}
\end{equation}%
where $W$ is the Tanaka-Webster scalar curvature (see the Appendix for an
explanation).
\end{theorem}

\begin{proof}
Fix $x\in M$ and choose a local orthonormal frame field $\{e_{1},\cdots
e_{2n}\}$ such that $e_{n+\alpha }=Je_{\alpha },\ 1\leq \alpha \leq n$ and $%
(\nabla e_{a})_{x}=0$ for all $1\leq a\leq 2n$. Then from (\ref{Dxi}) we
have at $x$ 
\begin{equation}
\begin{split}
D_{\xi }^{2}& =\sum_{a,b=1}^{2n}e_{a}\cdot \nabla _{e_{a}}(e_{b}\cdot \nabla
_{e_{b}}) \\
& =\sum_{a,b=1}^{2n}e_{a}\cdot \big((\nabla _{e_{a}}e_{b})\cdot \nabla
_{e_{b}}+e_{b}\cdot \nabla _{e_{a}}\nabla _{e_{b}}\big) \\
& =\sum_{a,b=1}^{2n}e_{a}e_{b}\cdot \nabla _{e_{a}}\nabla _{e_{b}} \\
& =\sum_{a,b=1}^{2n}e_{a}e_{b}\cdot \nabla _{e_{a},{e_{b}}}^{2} \\
& =-\sum_{a=1}^{2n}\nabla _{e_{a},{e_{a}}}^{2}+\sum_{a<b}e_{a}e_{b}\cdot
(\nabla _{e_{a},{e_{b}}}^{2}-\nabla _{e_{b},{e_{a}}}^{2}) \\
& =\nabla ^{\ast }\nabla +\mathcal{R}-2\sum_{\alpha =1}^{n}e_{\alpha
}e_{n+\alpha }\nabla _{T}
\end{split}%
\end{equation}%
where 
\begin{equation*}
\mathcal{R}=\frac{1}{2}\sum_{a,b=1}^{2n}e_{a}e_{b}\cdot R_{e_{a}e_{b}}.
\end{equation*}%
To complete the proof, we need to show that $\mathcal{R}=W$, the
Tanaka-Webster scalar curvature. We compute%
\begin{equation}
\begin{split}
\mathcal{R}& =\frac{1}{2}\sum_{a,b=1}^{2n}e_{a}e_{b}\cdot R_{e_{a}e_{b}} \\
& =\frac{1}{8}%
\sum_{a,b,c,d=1}^{2n}<R_{e_{a},e_{b}}^{p.h.}(e_{c}),e_{d}>e_{a}e_{b}e_{c}e_{d},\ \ 
\text{by}\ (\ref{currep}), \\
& =\frac{1}{8}\sum_{d=1}^{2n}\left\{ \frac{1}{3}\Big(\sum_{a,b,c\text{\
distinct}%
}<R_{e_{a},e_{b}}^{p.h.}(e_{c})+R_{e_{b},e_{c}}^{p.h.}(e_{a})+R_{e_{c},e_{a}}^{p.h.}(e_{b}),e_{d}>e_{a}e_{b}e_{c}%
\Big)\right. \\
& \ \ \ \ \ \ \ \ \ \ \ \ \ \left.
+\sum_{a,b=1}^{2n}<R_{e_{a},e_{b}}^{p.h.}(e_{a}),e_{d}>e_{a}e_{b}e_{a}+%
\sum_{a,b=1}^{2n}<R_{e_{a},e_{b}}^{p.h.}(e_{b}),e_{d}>e_{a}e_{b}e_{b}\right%
\} e_{d}.
\end{split}
\label{biavan}
\end{equation}%
We claim that 
\begin{equation}
\sum_{d=1}^{2n}\Big(\sum_{a,b,c\text{\ distinct}%
}<R_{e_{a},e_{b}}^{p.h.}(e_{c})+R_{e_{b},e_{c}}^{p.h.}(e_{a})+R_{e_{c},e_{a}}^{p.h.}(e_{b}),e_{d}>e_{a}e_{b}e_{c}%
\Big)e_{d}=0.  \label{biavan1}
\end{equation}%
Substituting (\ref{biavan1}) into (\ref{biavan}) gives 
\begin{equation}
\begin{split}
\mathcal{R}& =\frac{1}{4}%
\sum_{a,b,d=1}^{2n}<R_{e_{a},e_{b}}^{p.h.}(e_{a}),e_{d}>e_{b}e_{d} \\
& =-\frac{1}{4}\sum_{b,d=1}^{2n}Ric^{p.h.}(e_{b},e_{d})e_{b}e_{d} \\
& =-\frac{1}{4}\sum_{b=1}^{2n}Ric^{p.h.}(e_{b},e_{b})e_{b}e_{b}-\frac{1}{4}%
\sum_{b<d}Ric^{p.h.}(e_{b},e_{d})e_{b}e_{d}-\frac{1}{4}%
\sum_{b>d}Ric^{p.h.}(e_{b},e_{d})e_{b}e_{d} \\
& =\frac{1}{4}\sum_{b=1}^{2n}Ric^{p.h.}(e_{b},e_{b})=\frac{1}{4}R\ (\text{%
scalar curvature of }\nabla ^{p.h.}) \\
& =W\ (\text{Tanaka-Webster scalar curvature})\text{ (by (\ref{SC}) in the
Appendix)}.
\end{split}%
\end{equation}%
Finally we come back to prove the claim (\ref{biavan1}). From (\ref{ide4})
in the Appendix, it follows that 
\begin{equation}
\begin{split}
&
<R_{e_{a},e_{b}}^{p.h.}(e_{c})+R_{e_{b},e_{c}}^{p.h.}(e_{a})+R_{e_{c},e_{a}}^{p.h.}(e_{b}),e_{d}>
\\
=& <\mathbb{T}(e_{a},[e_{b},e_{c}])+\mathbb{T}(e_{b},[e_{c},e_{a}])+\mathbb{T%
}(e_{c},[e_{a},e_{b}]),e_{d}>.
\end{split}
\label{biavan2}
\end{equation}%
By formulae (\ref{ide1}) and (\ref{ide2}) and noting that $a,b,c$ are
distinct, we observe that the right hand side of (\ref{biavan2}) does not
vanish if and only if $\{a,b,c\}$ contains a pair $\{\beta ,n+\beta \}$ for
some $\beta $. That is, $\{a,b,c\}=\{\beta ,n+\beta ,\gamma \}\ $or$\
\{\beta ,n+\beta ,n+\gamma \}$ for some $\beta ,\gamma $ with $1\leq \beta
,\gamma \leq n$ and $\beta \neq \gamma $. For example, $\mathbb{T}(e_{\beta
},[e_{n+\beta },e_{\gamma }])+\mathbb{T}(e_{n+\beta },[e_{\gamma },e_{\beta
}])+\mathbb{T}(e_{\gamma },[e_{\beta },e_{n+\beta }])=-2\mathbb{T}(e_{\gamma
},T)\ $mod$\ T.$ Therefore, substituting (\ref{biavan2}), together with (\ref%
{ide2}), into the left hand side of (\ref{biavan1}), we have 
\begin{equation}
\begin{split}
& \sum_{d=1}^{2n}\Big(\sum_{a,b,c\text{\ distinct}%
}<R_{e_{a},e_{b}}^{p.h.}(e_{c})+R_{e_{b},e_{c}}^{p.h.}(e_{a})+R_{e_{c},e_{a}}^{p.h.}(e_{b}),e_{d}>e_{a}e_{b}e_{c}%
\Big)e_{d} \\
=& 6\sum_{d=1}^{2n}\sum_{\beta ,\gamma =1;\beta \neq \gamma }^{n}\Big(<%
\mathbb{T}(e_{r},T),e_{d}>e_{\beta }e_{n+\beta }e_{r}e_{d}+<\mathbb{T}%
(e_{n+r},T),e_{d}>e_{\beta }e_{n+\beta }e_{n+r}e_{d}\Big) \\
=& 6\sum_{\alpha =1}^{n}\sum_{\beta ,\gamma =1;\beta \neq \gamma }^{n}\Big(%
-ReA^{\bar{\alpha}}{}_{\gamma }e_{\beta }e_{n+\beta }e_{\gamma }e_{\alpha
}+ImA^{\bar{\alpha}}{}_{\gamma }e_{\beta }e_{n+\beta }e_{\gamma }e_{n+\alpha
} \\
& \ \ \ \ \ \ \ \ \ \ \ \ \ \ \ \ \ \ \ \ \ \ +ImA^{\bar{\alpha}}{}_{\gamma
}e_{\beta }e_{n+\beta }e_{n+\gamma }e_{\alpha }+ReA^{\bar{\alpha}}{}_{\gamma
}e_{\beta }e_{n+\beta }e_{n+\gamma }e_{n+\alpha }\Big) \\
=& 6\sum_{\alpha ,\beta ,\gamma =1}^{n}\Big(-ReA^{\bar{\alpha}}{}_{\gamma
}e_{\beta }e_{n+\beta }e_{\gamma }e_{\alpha }+ImA^{\bar{\alpha}}{}_{\gamma
}e_{\beta }e_{n+\beta }e_{\gamma }e_{n+\alpha } \\
& \ \ \ \ \ \ \ \ \ \ \ \ \ \ \ \ \ \ \ \ \ \ +ImA^{\bar{\alpha}}{}_{\gamma
}e_{\beta }e_{n+\beta }e_{n+\gamma }e_{\alpha }+ReA^{\bar{\alpha}}{}_{\gamma
}e_{\beta }e_{n+\beta }e_{n+\gamma }e_{n+\alpha }\Big) \\
=& 6\sum_{\beta =1}^{n}e_{\beta }e_{n+\beta }\Big(\sum_{\alpha ,\gamma
=1}^{n}-ReA^{\bar{\alpha}}{}_{\gamma }e_{\gamma }e_{\alpha }+ImA^{\bar{\alpha%
}}{}_{\gamma }e_{\gamma }e_{n+\alpha }+ImA^{\bar{\alpha}}{}_{\gamma
}e_{n+\gamma }e_{\alpha } \\
& \ \ \ \ \ \ \ \ \ \ \ \ \ \ \ \ \ \ \ \ \ \ \ \ +ReA^{\bar{\alpha}%
}{}_{\gamma }e_{n+\gamma }e_{n+\alpha }\Big) \\
=& 0
\end{split}
\label{biavan3}
\end{equation}%
where, for the third and fourth equalities in (\ref{biavan3}) we have used
the symmetry of pseudohermitian torsion $A_{\alpha \beta }$ and the basic
relation $e_{a}e_{b}+e_{b}e_{a}=-2\delta _{ab}$ .This completes the proof of
the claim (\ref{biavan1}) and hence the proof of the theorem.
\end{proof}

\begin{corollary}
\label{WFn2} (Weizenbock-type formula for $n=2$) In the preceding notations,
suppose further $n=2$ (dimension $=5$). Then it holds that 
\begin{equation}
D_{\xi }^{2}=\nabla ^{\ast }\nabla +W\text{ \ on }\Gamma (\mathbb{S}^{+})
\label{WF5D}
\end{equation}%
where $W$ is the Tanaka-Webster scalar curvature.
\end{corollary}

\begin{proof}
(\ref{WF5D}) \ follows from (\ref{WF}) and $(\ref{keyfor}).$
\end{proof}

\section{Asymptotically flat pseudohermitian manifolds}

Recall the standard coframe $\{\mathring{\theta},\sqrt{2}dz^{\beta },\sqrt{2}%
dz^{\bar{\beta}}\}$ for the Heisenberg group $H_{n}$ (see the Appendix),
where $\mathring{\theta}$ $=$ $dt$ $+$ $iz^{\beta }dz^{\bar{\beta}}$ $-$ $%
iz^{\bar{\beta}}dz^{\beta }$ (cf. (\ref{A1})). Let $B_{\rho _{0}}$ denote
the Heisenberg ball of radius $\rho _{0}$, i.e. $\{|z|^{4}+t^{2}$ $<$ $\rho
_{0}^{4}\}$ where $|z|^{2}$ $=$ $\sum_{\beta =1}^{n}|z^{\beta }|^{2}$ (cf. (%
\ref{A2})).

\begin{definition}
A ($2n+1)$-dimensional pseudohermitian manifold $(N,J,\theta )$ is said to
be \textbf{asymptotically flat pseudohermitian} if $N=N_{0}\cup N_{\infty }$%
, with $N_{0}$ compact and $N_{\infty }$ diffemorphic to $H_{n}\setminus
B_{\rho _{0}}$ in which $(J,\theta )$ is close to $(\mathring{J},\mathring{%
\theta})$ in the sense that 
\begin{equation}
\begin{split}
\theta & =\big(1+c_{n}A\rho ^{-2n}+O(\rho ^{-2n-1})\big)\mathring{\theta}%
+O(\rho ^{-2n-1})_{\beta }dz^{\beta }+O(\rho ^{-2n-1})_{\bar{\beta}}dz^{\bar{%
\beta}}; \\
\theta ^{\alpha }& =O(\rho ^{-2n-1})\mathring{\theta}+O(\rho ^{-2n-2})_{\bar{%
\beta}}{}^{\alpha }dz^{\bar{\beta}}+\big(1+\tilde{c}_{n}A\rho ^{-2n}+O(\rho
^{-2n-1})\big)\sqrt{2}dz^{\alpha },
\end{split}
\label{af1}
\end{equation}%
for some $A\in \mathbb{R}$ and a unitary coframe $\theta ^{\alpha }$ in
coordinates $(z^{\beta },$ $z^{\bar{\beta}},$ $t)$ (called asymptotic
coordinates) for $N_{\infty }$ on which $\rho $ $=$ $((\sum_{\beta
=1}^{n}|z^{\beta }|^{2})^{2}+t^{2})^{1/4}$ $=$ $(|z|^{4}+t^{2})^{1/4}$ is
defined. We also require the Tanaka-Webster scalar curvature $W\in L^{1}(N)$.
\end{definition}

\begin{remark}
For the case of blowing up through the Green's function (see Subsection \ref%
{blup}.), we have $c_{n}=\frac{4\pi }{na_{n}}$ (see (\ref{exofgr}) for the
definition of $a_{n}$) and $\tilde{c}_{n}=\frac{2\pi }{na_{n}}$, where $%
a_{1}=1$.
\end{remark}

Let $\{Z_{\alpha },Z_{\bar{\alpha}},T\}$ be the frame dual to $\{\theta
^{\alpha },\theta ^{\bar{\alpha}},\theta \}$. It follows from (\ref{af1})
that 
\begin{equation}
\begin{split}
Z_{\alpha }& =\left( 1-\tilde{c}_{n}A\rho ^{-2n}+O(\rho ^{-2n-1})\right) 
\mathring{Z}_{\alpha }+\sum_{\beta \neq \alpha }O(\rho ^{-2n-2})_{\alpha
}{}^{\beta }\mathring{Z}_{\beta } \\
& +\sum_{\beta =1}^{n}O(\rho ^{-2n-2})_{\alpha }{}^{\bar{\beta}}\mathring{Z}%
_{\bar{\beta}}+O(\rho ^{-2n-1})\mathring{T}.
\end{split}
\label{af2}
\end{equation}%
Writing 
\begin{equation}
\begin{split}
\theta & =a\mathring{\theta}+b_{\gamma }dz^{\gamma }+c_{\bar{\gamma}}dz^{%
\bar{\gamma}}; \\
\theta ^{\alpha }& =a^{\alpha }\mathring{\theta}+b\sqrt{2}dz^{\alpha }+c_{%
\bar{\gamma}}{}^{\alpha }dz^{\bar{\gamma}},
\end{split}
\label{af3}
\end{equation}%
where 
\begin{equation}
\begin{split}
a& =1+c_{n}A\rho ^{-2n}+O(\rho ^{-2n-1}), \\
b_{\gamma }& =O(\rho ^{-2n-1}),\ \ c_{\bar{\gamma}}=O(\rho ^{-2n-1})_{\bar{%
\gamma}}, \\
b& =1+\tilde{c}_{n}A\rho ^{-2n}+O(\rho ^{-2n-1}), \\
a^{\alpha }& =O(\rho ^{-2n-1})^{\alpha },\ \ c_{\bar{\gamma}}{}^{\alpha
}=O(\rho ^{-2n-2})_{\bar{\gamma}}{}^{\alpha }.
\end{split}
\label{af4}
\end{equation}%
Substituting (\ref{af3}) into the Levi equation $d\theta =i\delta _{\alpha
\beta }\theta ^{\alpha }\wedge \theta ^{\bar{\beta}}$ ((\ref{Levi}) with $%
h_{\alpha \bar{\beta}}=\delta _{\alpha \beta }$) gives 
\begin{equation}
a^{\alpha }=\frac{-c_{n}Anz^{\alpha }\omega }{\sqrt{2}\rho ^{2n+4}}+O(\rho
^{-2n-2})  \label{af5}
\end{equation}

\noindent in view of (\ref{af4}), where $\omega $ $=$ $t+i|z|^{2}$. Next we
would like to compute the asymptotic behavior of the connection forms $%
\theta _{\alpha }$ $^{\beta }$ and the torsion forms $\tau ^{\beta }$. Write 
\begin{equation}
\begin{split}
\theta _{\alpha }{}^{\beta }& =A_{\alpha }{}^{\beta }\mathring{\theta}%
+B_{\alpha }{}^{\beta }{}_{\gamma }dz^{\gamma }+C_{\alpha }{}^{\beta }{}_{%
\bar{\gamma}}dz^{\bar{\gamma}}; \\
\tau ^{\beta }& =A^{\beta }\mathring{\theta}+B^{\beta }{}_{\gamma
}dz^{\gamma }+C^{\beta }{}_{\bar{\gamma}}{}dz^{\bar{\gamma}}.
\end{split}
\label{af6}
\end{equation}

\begin{lemma}
With the notations above, it holds that in $N_{\infty },$ for the
coefficients of $\theta _{\alpha }$ $^{\beta }$%
\begin{equation}
A_{\alpha }{}^{\beta }=O(\rho ^{-2n-2}),\ \ \ B_{\alpha }{}^{\beta
}{}_{\gamma }=-\overline{C_{\beta }{}^{\alpha }{}_{\bar{\gamma}}};
\label{af7}
\end{equation}%
and for any fixed $\alpha ,$ 
\begin{equation}
\begin{split}
C_{\alpha }{}^{\beta }{}_{\bar{\gamma}}& =O(\rho ^{-2n-2}),\ \ \text{for}\
\beta \neq \alpha ,\ \gamma \neq \alpha \\
C_{\alpha }{}^{\beta }{}_{\bar{\alpha}}& =\frac{-inc_{n}Az^{\beta }\omega }{%
\rho ^{2n+4}}+O(\rho ^{-2n-2}),\ \ \text{for}\ \beta \neq \alpha \\
C_{\alpha }{}^{\alpha }{}_{\bar{\gamma}}& =\frac{-in\tilde{c}_{n}Az^{\gamma
}\omega }{\rho ^{2n+4}}+O(\rho ^{-2n-2}),\ \ \text{for}\ \gamma \neq \alpha
\\
C_{\alpha }{}^{\alpha }{}_{\bar{\alpha}}& =\frac{-in(\tilde{c}%
_{n}+c_{n})Az^{\alpha }\omega }{\rho ^{2n+4}}+O(\rho ^{-2n-2});
\end{split}
\label{af8}
\end{equation}%
(recall $\omega $ $=$ $t+i|z|^{2})$ and for the coefficients of $\tau
^{\beta },$ 
\begin{equation}
\begin{split}
A_{\bar{\beta}\bar{\gamma}}& =O(\rho ^{-2n-2}),\ \ \ A^{\beta }=A_{\bar{\beta%
}\bar{\gamma}}a^{\bar{\gamma}} \\
B^{\beta }{}_{\sigma }& =A_{\bar{\beta}\bar{\gamma}}\overline{c_{\bar{\sigma}%
}{}^{\gamma }},\ \ \ C^{\beta }{}_{\bar{\gamma}}=\sqrt{2}A_{\bar{\beta}\bar{%
\gamma}}b.
\end{split}
\label{af9}
\end{equation}
\end{lemma}

\begin{proof}
These formulas are easily seen from the structure equations (\ref{SE}) in
pseudohermitian geometry.
\end{proof}

\subsection{Pseudohermitian mass}

We define the pseudohermitian mass (p-mass in short) $m(J,\theta )$ on an
asymptotically flat pseudohermitian manifold $(N,J,\theta )$ by 
\begin{equation}
m(J,\theta ):=\lim_{\Lambda \rightarrow \infty }ni\oint_{S_{\Lambda
}}\sum_{\gamma =1}^{n}\theta _{\gamma }{}^{\gamma }\wedge \theta \wedge
(d\theta )^{n-1}  \label{af10}
\end{equation}

\noindent where $S_{\Lambda }$ $\subset $ $N_{\infty }$ denotes a Heisenberg
sphere $\partial B_{\Lambda }$ $=$ $\{|z|^{4}+t^{2}$ $=$ $\Lambda ^{4}\}$ of
large radius $\Lambda $.

The p-mass of the Heisenberg group $H_{n}$ is $m(J,\mathring{\theta})$ $=$ $%
0 $ since $\mathring{\theta}_{\alpha }{}^{\gamma }$ $=$ $0$ for all $\alpha
, $ $\gamma $ ((\ref{A1-1})). The notion of the p-mass was motivated by an
idea in general relativity (see \cite{CMY} for the case of dimension 3). It
has a variational meaning as shown below. Consider the Einstein-Hilbert type
action integral 
\begin{equation}
\mathfrak{A}(J,\theta )=-\int_{N}W\ \theta \wedge (d\theta )^{n}.
\end{equation}%
For $\theta $ fixed, consider the variation $J_{(t)}$ of $J$ that maintains
the asymptotically flat structure. Then writing $\frac{d}{dt}\big|%
_{t=0}J_{(t)}=2E$, $E$ $=$ $E_{\gamma }{}^{\bar{\beta}}\theta ^{\gamma
}\otimes Z_{\bar{\beta}}+E_{\bar{\gamma}}{}^{\beta }\theta ^{\bar{\gamma}%
}\otimes Z_{\beta },$ we have 
\begin{equation}
\frac{d}{dt}\Big|_{t=0}\big(\mathfrak{A}(J_{(t)},\theta )+m(J_{(t)},\theta )%
\big)=n\int_{N}\sum_{\beta ,\gamma }(A^{\beta }{}_{\bar{\gamma}}E_{\beta
}{}^{\bar{\gamma}}+A^{\bar{\beta}}{}_{\gamma }E_{\bar{\beta}}{}^{\gamma })\
\theta \wedge (d\theta )^{n}
\end{equation}%
where $A^{\beta }{}_{\bar{\gamma}}$.are coefficients of the torsion forms $%
\tau ^{\beta }:$ $\tau ^{\beta }=A^{\beta }{}_{\bar{\gamma}}\theta ^{\bar{%
\gamma}}.$

\begin{lemma}
\label{CRmass} In the preceding notations, it holds that 
\begin{equation}
m(J,\theta )=\left( n!(2^{2n}n^{2})(n\tilde{c}_{n}+c_{n})\alpha _{n}\Omega
_{n}\right) A  \label{af12}
\end{equation}%
where $\Omega _{n}$ is the Euclidean volume of the unit ball of $%
C^{n},\alpha _{1}=2,\alpha _{2}=\frac{\pi }{2}$ and 
\begin{equation}
\alpha _{n}=\left\{ 
\begin{array}{ll}
\frac{\prod_{k=1}^{m}(2k)}{\prod_{k=1}^{m}(2k+1)}\alpha _{1}, & n=2m+1,\
m\geq 1 \\ 
&  \\ 
\frac{\prod_{k=2}^{m}(2k-1)}{\prod_{k=2}^{m}(2k)}\alpha _{2}, & n=2m,\ m\geq
2%
\end{array}%
\right.  \label{af13}
\end{equation}
\end{lemma}

\begin{proof}
By (\ref{af6}),(\ref{af7}) and (\ref{af8}), we have 
\begin{equation}
\begin{split}
ni\sum_{\gamma }\theta _{\gamma }{}^{\gamma }& =ni\sum_{\gamma }A_{\gamma
}{}^{\gamma }\mathring{\theta}+B_{\gamma }{}^{\gamma }{}_{\beta }dz^{\beta
}+C_{\gamma }{}^{\gamma }{}_{\bar{\beta}}dz^{\bar{\beta}} \\
& =ni\left[ O(\rho ^{-2n-2})\mathring{\theta}+\left( \frac{-i(n^{2}\tilde{c}%
_{n}+nc_{n})Az^{\bar{\beta}}\bar{w}}{\rho ^{2n+4}}+O(\rho ^{-2n-2})\right)
dz^{\beta }\right. \\
& \left. \ \ \ \ \ \ \ \ \ \ \ \ +\left( \frac{-i(n^{2}\tilde{c}%
_{n}+nc_{n})Az^{\bar{\beta}}\bar{w}}{\rho ^{2n+4}}+O(\rho ^{-2n-2})\right)
dz^{\bar{\beta}}\right] .
\end{split}
\label{af14}
\end{equation}

Observing that $d\rho =\frac{|z|^{2}z^{\bar{\beta}}dz^{\beta
}+|z|^{2}z^{\beta }dz^{\bar{\beta}}+tdt}{2\rho ^{3}}=O(\rho )$, via (\ref%
{af1}) we have 
\begin{equation}
(d\theta )^{n-1}=(d\mathring{\theta})^{n-1}+O(\rho ^{-2}).  \label{af15}
\end{equation}%
Substituting (\ref{af1}), (\ref{af14}) and (\ref{af15}) into (\ref{af10}),
we get (recall $\omega $ $=$ $t+i|z|)$ 
\begin{equation}
\begin{split}
m(J,\theta )& =\lim_{\Lambda \rightarrow \infty }ni\oint_{S_{\Lambda
}}\sum_{\gamma }\theta _{\gamma }{}^{\gamma }\wedge \theta \wedge (d\theta
)^{n-1} \\
& =n(n^{2}\tilde{c}_{n}+nc_{n})A\oint_{\infty }\rho ^{-(2n+4)}\sum_{\beta
}(z^{\bar{\beta}}\bar{\omega}dz^{\beta }+z^{\beta }\omega dz^{\bar{\beta}%
})\wedge \mathring{\theta}\wedge (d\mathring{\theta})^{n-1} \\
& =n(n^{2}\tilde{c}_{n}+nc_{n})A\oint_{S_{1}}\sum_{\beta }(z^{\bar{\beta}}%
\bar{\omega}dz^{\beta }+z^{\beta }\omega dz^{\bar{\beta}})\wedge \mathring{%
\theta}\wedge (d\mathring{\theta})^{n-1},\ S_{1}=\{\rho =1\}.
\end{split}
\label{af16}
\end{equation}%
We claim that 
\begin{equation}
\oint_{S_{1}}\sum_{\beta }(z^{\bar{\beta}}\bar{\omega}dz^{\beta }+z^{\beta
}\omega dz^{\bar{\beta}})\wedge \mathring{\theta}\wedge (d\mathring{\theta}%
)^{n-1}=2^{2n}n!\alpha _{n}\Omega _{n}  \label{af17}
\end{equation}%
where 
\begin{equation*}
\alpha _{n}:=\int_{-1}^{1}(1-t^{2})^{\frac{n-1}{2}}dt.
\end{equation*}%
By trigonometric substitution, $t=\sin {\theta },\ -\frac{\pi }{2}\leq
\theta \leq \frac{\pi }{2}$, it is easily seen that 
\begin{equation*}
\alpha _{n}=2\int_{0}^{\frac{\pi }{2}}\cos ^{n}{\theta }d\theta .
\end{equation*}%
Then (\ref{af13}) follows from Wallis' formulas. Substituting (\ref{af17})
into (\ref{af16}), we complete the proof of (\ref{af12}). Now, for (\ref%
{af17}) let $z=r\varphi $ where $\varphi =(\varphi _{1},\cdots ,\varphi
_{n})\in $ the unit sphere of $C^{n}$, i.e. $\sum_{\beta =1}^{n}|\varphi
_{\beta }|^{2}=1$. Then, on $S_{1}$ we have 
\begin{equation}
\begin{split}
z^{\beta }dz^{\bar{\beta}}-z^{\bar{\beta}}dz^{\beta }& =r^{2}(\varphi
_{\beta }d\varphi _{\bar{\beta}}-\varphi _{\bar{\beta}}d\varphi _{\beta }),
\\
(z^{\beta }dz^{\bar{\beta}}+z^{\bar{\beta}}dz^{\beta })& =2r|\varphi _{\beta
}|^{2}dr, \\
\mathring{\theta}=dt& +ir^{2}\sum_{\beta }(\varphi _{\beta }d\varphi _{\bar{%
\beta}}-\varphi _{\bar{\beta}}d\varphi _{\beta }), \\
d\mathring{\theta}=2i\Big(r^{2}\sum_{\beta }d\varphi _{\beta }\wedge
d\varphi _{\bar{\beta}}& +rdr\wedge \sum_{\beta }(\varphi _{\beta }d\varphi
_{\bar{\beta}}-\varphi _{\bar{\beta}}d\varphi _{\beta })\Big) \\
=2i(r^{2}B_{2}+rdr\wedge B_{1})& ,
\end{split}
\label{af18}
\end{equation}%
where $B_{1}=\sum_{\beta }(\varphi _{\beta }d\varphi _{\bar{\beta}}-\varphi
_{\bar{\beta}}d\varphi _{\beta })$ and $B_{2}=\sum_{\beta }d\varphi _{\beta
}\wedge d\varphi _{\bar{\beta}}$. Observe that 
\begin{equation}
(d\mathring{\theta})^{n-1}=(2i)^{n-1}\Big((n-1)(r^{2}B_{2})^{n-2}\wedge
(rdr\wedge B_{1})+(r^{2}B_{2})^{n-1}\Big),  \label{af19}
\end{equation}%
Using (\ref{af18}), (\ref{af19}) and $r^{4}+t^{2}=1$, we get 
\begin{equation}
\begin{split}
& \oint_{S}\sum_{\beta }(z^{\bar{\beta}}\bar{\omega}dz^{\beta }+z^{\beta
}\omega dz^{\bar{\beta}})\wedge \mathring{\theta}\wedge (d\mathring{\theta}%
)^{n-1} \\
=& \oint_{S}\sum_{\beta }i|z|^{2}(z^{\beta }dz^{\bar{\beta}}-z^{\bar{\beta}%
}dz^{\beta })\wedge \mathring{\theta}\wedge (d\mathring{\theta}%
)^{n-1}+\oint_{S}\sum_{\beta }t(z^{\beta }dz^{\bar{\beta}}+z^{\bar{\beta}%
}dz^{\beta })\wedge \mathring{\theta}\wedge (d\mathring{\theta})^{n-1} \\
=& (2i)^{n-1}\oint_{S}(ir^{4}B_{1}+2trdr)\wedge (dt+ir^{2}B_{1})\wedge
(r^{2}B_{2})^{n-1} \\
=& (2i)^{n-1}\oint_{S}(ir^{4}B_{1}\wedge dt+2itr^{3}dr\wedge B_{1})\wedge
(r^{2}B_{2})^{n-1} \\
=& 2^{n-1}\oint_{S}r^{2n-2}(iB_{1})\wedge dt\wedge (iB_{2})^{n-1} \\
=& \oint_{S}r^{2n-2}(iB_{1})\wedge (2iB_{2})^{n-1}\wedge dt \\
=& \int_{\{z\in C^{n}|r=1\}}\left( \int_{-1}^{1}(1-t^{2})^{\frac{n-1}{2}%
}dt\right) (iB_{1})\wedge (2iB_{2})^{n-1} \\
=& \alpha _{n}\int_{\{z\in C^{n}|r\leq 1\}}d\left( \sum_{\beta }i(z^{\beta
}dz^{\bar{\beta}}-z^{\bar{\beta}}dz^{\beta })\wedge (2i\sum_{\beta
}dz^{\beta }\wedge dz^{\bar{\beta}})^{n-1}\right) \\
=& \alpha _{n}\int_{\{z\in C^{n}|r\leq 1\}}(2i\sum_{\beta }dz^{\beta }\wedge
dz^{\bar{\beta}})^{n} \\
=& 2^{2n}n!\alpha _{n}\Omega _{n}.
\end{split}
\label{af20}
\end{equation}%
This completes the proof of (\ref{af17}) and hence the lemma.
\end{proof}

We also define (a variant of p-mass in terms of real version of connection
forms):%
\begin{equation}
\widetilde{m}(J,\theta ):=\lim_{\Lambda \rightarrow \infty
}\oint_{S_{\Lambda }}\sum_{j,k=1}^{2n}\omega _{j}{}^{k}(e_{j})e_{k}\lrcorner
dV,  \label{af11}
\end{equation}%
where $\{\omega _{j}{}^{k}\}$ denote the (real version of) connection forms
with respect to an orthonormal frame $\{e_{j}\}$ (see the Appendix) and%
\begin{equation*}
dV:=\frac{1}{2^{n}n!}\theta \wedge (d\theta )^{n}=\theta \wedge \left(
\bigwedge_{\beta =1}^{n}\omega ^{\beta }\wedge \omega ^{n+\beta }\right) .
\end{equation*}

\begin{lemma}
In the preceding notations, it holds that 
\begin{equation}
\widetilde{m}(J,\theta )=\left( 2^{n+\frac{3}{2}}n^{2}(\tilde{c}%
_{n}+nc_{n})\alpha _{n}\Omega _{n}\right) A  \label{af21}
\end{equation}
\end{lemma}

\begin{proof}
By the definition of interior product, we have 
\begin{equation}
\sum_{j,k}\omega _{j}{}^{k}(e_{j})e_{k}\lrcorner dV=\sum_{\gamma
=1}^{n}\left( \omega _{j}{}^{\gamma }(e_{j})\omega ^{n+\gamma }-\omega
_{j}{}^{n+\gamma }(e_{j})\omega ^{\gamma }\right) \wedge \theta \wedge
\left( \bigwedge_{\beta =1,\beta \neq \gamma }^{n}\omega ^{\beta }\wedge
\omega ^{n+\beta }\right) .  \label{af22}
\end{equation}%
Observe that 
\begin{equation}
\theta \wedge \left( \bigwedge_{\beta =1,\beta \neq \gamma }^{n}\omega
^{\beta }\wedge \omega ^{n+\beta }\right) =\left( \frac{i}{2}\right) \theta
\wedge \left( \bigwedge_{\beta =1,\beta \neq \gamma }^{n}\theta ^{\beta
}\wedge \theta ^{\bar{\beta}}\right)  \label{af23}
\end{equation}%
and 
\begin{equation}
\begin{split}
& \sum_{j}\left( \omega _{j}{}^{\gamma }(e_{j})\omega ^{n+\gamma }-\omega
_{j}{}^{n+\gamma }(e_{j})\omega ^{\gamma }\right) \\
=& \sum_{\alpha }\left( -\theta _{\bar{\alpha}}{}^{\bar{\gamma}}(Z_{\alpha
})\theta ^{\gamma }+\theta _{\alpha }{}^{\gamma }(Z_{\bar{\alpha}})\theta ^{%
\bar{\gamma}}\right) \\
=& i\left( \sum_{\alpha }C_{\alpha }{}^{\gamma }{}_{\bar{\alpha}}\right)
\theta ^{\bar{\gamma}}+\ \text{conj.} \\
=& \sqrt{2}n(nc_{n}+\tilde{c}_{n})A\left( \frac{z^{\gamma }\omega dz^{\bar{%
\gamma}}+z^{\bar{\gamma}}\bar{\omega}dz^{\gamma }}{\rho ^{2n+4}}\right) +%
\text{h.d.o.t.}
\end{split}
\label{af24}
\end{equation}%
where \text{h.d.o.t. means "higher decay order term(s)" }and recall $\omega $
$=$ $t+i|z|^{2}$. Substituting (\ref{af23}) and (\ref{af24}) into (\ref{af22}%
), and noting that 
\begin{equation}
\mathring{\theta}\wedge (d\mathring{\theta})^{n-1}=2^{n-1}(n-1)!\left( \frac{%
i}{2}\right) ^{n-1}\mathring{\theta}\wedge \sum_{\alpha =1}^{n}\left(
\bigwedge_{\beta =1,\beta \neq \alpha }^{n}\sqrt{2}dz^{\beta }\wedge \sqrt{2}%
dz^{\bar{\beta}}\right) ,  \label{af25}
\end{equation}%
we have 
\begin{equation*}
\sum_{j,k}\omega _{j}{}^{k}(e_{j})e_{k}\lrcorner dV=\frac{\sqrt{2}n(nc_{n}+%
\tilde{c}_{n})A}{2^{n-1}(n-1)!}\sum_{\gamma }\frac{z^{\gamma }\omega dz^{%
\bar{\gamma}}+z^{\bar{\gamma}}\bar{\omega}dz^{\gamma }}{\rho ^{2n+4}}\wedge 
\mathring{\theta}\wedge (d\mathring{\theta})^{n-1}+\text{h.d.o.t.}
\end{equation*}%
Therefore 
\begin{equation*}
\begin{split}
\widetilde{m}(J,\theta )& =\lim_{\Lambda \rightarrow \infty
}\oint_{S_{\Lambda }}\sum_{j,k}\omega _{j}{}^{k}(e_{j})e_{k}\lrcorner dV \\
& =\frac{\sqrt{2}n(nc_{n}+\tilde{c}_{n})A}{2^{n-1}(n-1)!}\int_{S_{1}}\sum_{%
\beta }(z^{\bar{\beta}}\bar{\omega}dz^{\beta }+z^{\beta }\omega dz^{\bar{%
\beta}})\wedge \mathring{\theta}\wedge (d\mathring{\theta})^{n-1}.
\end{split}%
\end{equation*}%
Together with (\ref{af17}), we get the formula (\ref{af21}).
\end{proof}

\subsection{Blow-up through the Green's function}

\label{blup}

Suppose that $(M,J)$ is spherical in a neighborhood of a point $p\in M$.
Then we can find a contact form $\hat{\theta}$ and local coordinates $(z,t)$
such that $(z,t)(p)=0$ and near $p$%
\begin{equation*}
\hat{\theta}=\mathring{\theta},\ \ \ \hat{\theta}^{\alpha }=\sqrt{2}%
dz^{\alpha }.
\end{equation*}%
We call such kind of coordinates the CR normal coordinates (see \ref{JL1}).
Recall $\rho $ $=$ $((\sum_{\beta =1}^{n}|z^{\beta }|^{2})^{2}+t^{2})^{1/4}.$

\begin{proposition}
\label{GE} With the notations above, in CR normal coordinates $(z,t)$ the
Green's function $G_{p}$ of the CR invariant sublaplacian $L_{b}$ (see (\ref%
{Lb}) and (\ref{Lb-2}) for the definition of $L_{b}$ and $G_{p}$) admits the
following expansion 
\begin{equation}
G_{p}=\frac{a_{n}}{2\pi }\rho ^{-2n}+A_{p}+O(\rho )  \label{exofgr}
\end{equation}%
for some dimensional constant $a_{n}>0$ and some constant $A_{p}\in \mathbb{R%
}.$
\end{proposition}

\begin{proof}
To deduce formula (\ref{exofgr}), recall $L_{b}$ :$=$ $b_{n}\Delta _{b}+W$
where $b_{n}$ $=$ $2+\frac{2}{n},$ $W$ is the Tanaka-Webster scalar
curvature (see \cite{JL} or (\ref{Lb})). Since, near $p\in M,\ \hat{\theta}=%
\mathring{\theta}$ hence $L_{b}=b_{n}\mathring{\Delta}_{b}$. Thus near $p$
we have%
\begin{eqnarray}
b_{n}\mathring{\Delta}_{b}\left( G_{p}-\frac{a_{n}}{2\pi }\rho ^{-2n}\right)
&=&b_{n}\mathring{\Delta}_{b}G_{p}-b_{n}\mathring{\Delta}_{b}\left( \frac{%
a_{n}}{2\pi }\rho ^{-2n}\right)  \label{gr0} \\
&=&L_{b}G_{p}-\frac{a_{n}b_{n}}{2\pi }\mathring{\Delta}_{b}\rho ^{-2n} 
\notag \\
&=&16\delta _{p}-16\delta _{p}=0.  \notag
\end{eqnarray}%
\noindent Here we have used the following formulae (cf. (\ref{Lb-2}))%
\begin{equation}
L_{b}G_{p}=16\delta _{p},  \label{gr1}
\end{equation}%
\begin{equation}
\mathring{\Delta}_{b}\rho ^{-2n}=\frac{32\pi }{a_{n}b_{n}}\delta _{0}.
\label{gr2}
\end{equation}

\noindent By (\ref{gr0}) we conclude that $G_{p}-\frac{a_{n}}{2\pi }\rho
^{-2n}$ is a smooth function near $p$ since $\mathring{\Delta}_{b}$ is
hypoelliptic. So (\ref{exofgr}) follows.
\end{proof}

Note that (\ref{gr2}) determines the value of $a_{n}.$ We will consider the
following pseudohermitian manifold called the blow-up at $p$ (through the
Green's function):%
\begin{equation}
\left( M\setminus \{p\},J,\theta :=G_{p}^{\frac{2}{n}}\hat{\theta}\right) .
\label{blow-up}
\end{equation}

\begin{proposition}
\label{BlmG} With the preceding notations, we have (1) $\left( M\setminus
\{p\},J,\theta :=G_{p}^{\frac{2}{n}}\hat{\theta}\right) $ is asymptotically
flat; (2) The p-mass $m(J,\theta )$ is a positive multiple of $A_{p}.$
\end{proposition}

\begin{proof}
Take an admissible coframe $\{\theta ^{\alpha }\}$ for $\theta $ as follows: 
\begin{equation}
\theta ^{\alpha }=G_{p}^{\frac{1}{n}}\left( \hat{\theta}^{\alpha }+2i(\log {%
G_{p}^{\frac{1}{n}}})^{\alpha }\hat{\theta}\right) .
\end{equation}%
\noindent Recall (\cite[p.421]{Le1}) that this coframe is taken to have $%
h_{\alpha \beta }=\hat{h}_{\alpha \beta }=\delta _{\alpha \beta }$. From (%
\ref{exofgr}) it follows that 
\begin{equation}
\begin{split}
\theta & =G_{p}^{\frac{2}{n}}\hat{\theta}=\left( \frac{a_{n}}{2\pi }\rho
^{-2n}+A_{p}+O(\rho )\right) ^{\frac{2}{n}}\mathring{\theta} \\
& =\left( \frac{a_{n}}{2\pi }\right) ^{\frac{2}{n}}\left( \rho ^{-4}+\frac{%
4\pi }{na_{n}}A_{p}\rho ^{2n-4}+O(\rho ^{2n-3})\right) \mathring{\theta},
\end{split}%
\end{equation}%
\noindent and\ (recall $\omega =t+i|z|^{2})$%
\begin{equation}
\begin{split}
\theta ^{\alpha }& =G_{p}^{\frac{1}{n}}\left( \hat{\theta}^{\alpha }+2i(\log 
{G_{p}^{\frac{1}{n}}})^{\alpha }\hat{\theta}\right) \\
& =\left( \frac{a_{n}}{2\pi }\rho ^{-2n}+A_{p}+O(\rho )\right) ^{\frac{1}{n}}%
\sqrt{2}dz^{\alpha }+\frac{2i}{n}G_{p}^{\frac{1-n}{n}}\left( \hat{Z}_{\bar{%
\alpha}}G_{p}\right) \mathring{\theta} \\
& =\left( \frac{a_{n}}{2\pi }\right) ^{\frac{1}{n}}\left( \rho ^{-2}+\frac{%
2\pi }{na_{n}}A_{p}\rho ^{2n-2}+O(\rho ^{2n-1})\right) \sqrt{2}dz^{\alpha }
\\
& +\left( \frac{a_{n}}{2\pi }\right) ^{\frac{1}{n}}\frac{2i}{n}\left( \frac{%
in}{\sqrt{2}}\frac{z^{\alpha }\omega }{\rho ^{2n+4}}+O(1)\right) \left( \rho
^{2n-2}+\frac{2(1-n)\pi }{na_{n}}A_{p}\rho ^{4n-2}+O(\rho ^{4n-1})\right) 
\mathring{\theta}.
\end{split}%
\end{equation}%
Next we would like to express $\theta $ and $\theta ^{\alpha }$ in inverted
CR normal coordinates $(z_{\ast },t_{\ast })$. If $(z,t)$ are CR normal
coordinates in a neighborhood $U$ of $p$, we define the inverted CR normal
coordinates as%
\begin{equation*}
z_{\ast }^{\alpha }=\frac{z^{\alpha }}{\omega };\ t_{\ast }=-\frac{t}{%
|\omega |^{2}}\ \text{on}\ U\setminus \{p\},\text{ }\omega =t+i|z|^{2},\text{
}|z|^{2}=\sum_{\alpha =1}^{n}|z^{\alpha }|^{2}.
\end{equation*}%
\noindent Then we have 
\begin{eqnarray}
\theta &=&\left( \frac{a_{n}}{2\pi }\right) ^{\frac{2}{n}}\left( 1+\frac{%
4\pi }{na_{n}}A_{p}\rho _{\ast \rho }^{-2n}+O(\rho _{\ast }^{-2n-1})\right) (%
\mathring{\theta})_{\ast }  \label{3-36} \\
\theta ^{\alpha } &=&\left( \frac{a_{n}}{2\pi }\right) ^{\frac{1}{n}}\left(
1+\frac{2\pi }{na_{n}}A_{p}\rho _{\ast }^{-2n}+O(\rho _{\ast
}^{-2n-1})\right) \left( \sum_{\beta =1}^{n}a_{\alpha \beta }\sqrt{2}%
dz_{\ast }^{\beta }\right)  \notag \\
&&+\left( \frac{a_{n}}{2\pi }\right) ^{\frac{2}{n}}\left( \frac{2\sqrt{2}\pi 
}{a_{n}}A_{p}\frac{z^{\alpha }}{\omega ^{2}}\rho _{\ast }^{-2n+2}+O(\rho
_{\ast }^{-2n-2})\right) (\mathring{\theta})_{\ast }  \notag
\end{eqnarray}%
\noindent where 
\begin{equation}
a_{\alpha \beta }=\left\{ 
\begin{array}{ll}
\frac{(2i|z^{\alpha }|^{2}-\omega )\rho ^{2}}{\omega ^{2}}, & \alpha =\beta
\\ 
&  \\ 
\frac{(2iz^{\alpha }z^{\bar{\beta}})\rho ^{2}}{\omega ^{2}}, & \alpha \neq
\beta .%
\end{array}%
\right.
\end{equation}%
\noindent Changing coordinates $(z_{\ast },t_{\ast })$ by rescaling to
absorb constant $(a_{n}/2\pi )^{2/n},$ we obtain that 
\begin{equation}
A\text{ in }(\ref{af1})\text{ }=\text{ }(\frac{a_{n}}{2\pi })^{2}A_{p}.
\label{Ap}
\end{equation}%
\noindent Note that setting $\tilde{z}^{\alpha }=(a_{n}/2\pi )^{1/n}z_{\ast
}^{\alpha },$ $\tilde{t}=(a_{n}/2\pi )^{2/n}t_{\ast }$ gives $\rho _{\ast
}^{-2n}=(a_{n}/2\pi )^{2}\tilde{\rho}^{-2n}.$ We have shown that $\left(
M\setminus \{p\},J,\theta =G_{p}^{\frac{2}{n}}\hat{\theta}\right) $ is an
asymptotically flat pseudohermitian manifold. By (\ref{af12}) and (\ref{Ap})
we also conclude that $m(J,\theta )$ is a positive multiple of $A_{p}.$
\end{proof}

\section{Proofs of Theorem \protect\ref{PMT} and Corollary \protect\ref{PMT'}%
}

\subsection{Proof for $m\geq 0$}

First we need to find a positive spinor $\psi $ satisfying the equation $%
D_{\xi }^{2}\psi $ $=$ $0$ and approaching a constant spinor at the
infinity. Let us begin with the definition of weighted Folland-Stein spaces
on the Heisenberg group $H_{n}$. See the Appendix for basic material about $%
H_{n}.$ Let $\sigma $ $:=$ $(1+\rho ^{4})^{1/4}$ where we recall $\rho $ $=$ 
$((\sum_{\beta =1}^{n}|z^{\beta }|^{2})^{2}+t^{2})^{1/4}.$ The weighted
Lebesgue spaces $L_{\delta }^{p}(H_{n}),\ 1\leq p\leq \infty $, with weight $%
\delta \in \mathbb{R}$ are the spaces of measurable functions in $%
L_{loc}^{p}(H_{n})$ such that the norms $\Vert \cdot \Vert _{p,\delta }$
defined by 
\begin{equation}
\Vert u\Vert _{p,\delta }:=\left\{ 
\begin{array}{cc}
\left( \int_{H_{n}}|u|^{p}\sigma ^{-\delta p-Q}\ \mathring{\theta}\wedge (d%
\mathring{\theta})^{n}\right) ^{1/p}, & 1\leq p<\infty \\ 
\text{ess sup}_{H_{n}}(\sigma ^{-\delta }|u|), & p=\infty ,%
\end{array}%
\right.
\end{equation}%
are finite. Here $Q=2n+2$ is the homogeneous dimension. The weighted
Folland-Stein spaces $S_{k,\delta }^{p}(H_{n})$ are now defined with respect
to the norm 
\begin{equation}
\Vert u\Vert _{k,p,\delta }:=\sum_{j=0}^{k}\Vert \nabla ^{j}u\Vert
_{p,\delta -j},
\end{equation}%
where $\nabla ^{j}u:=\sum_{|I|=j}\mathring{e}_{I}u,\ I=(i_{1},\cdots ,i_{j})$
a multiindex and $\mathring{e}_{I}u=\mathring{e}_{i_{1}}\cdots \mathring{e}%
_{i_{j}}u$ (see (\ref{7-8a}) for the definition of $\mathring{e}_{j}$). We
can then extend the definition to the sections of $\mathbb{S}^{+}$ over an
asymptotically flat pseudohermitian manifold $N$ with $\nabla $ being a spin
connection.

\begin{proposition}
\label{Draiso} (\cite{Chiu}) Suppose that $(N,J,\theta )$ is an
asymptotically flat, pseudohermitian and spin manifold of dimension 5.
Assume that $J$ is spherical$.$ Then for $0<\eta <4$, the square of the
contact Dirac operator 
\begin{equation}
D_{\xi }^{2}:S_{3,-\eta }^{2}(\mathbb{S}^{+})\longrightarrow S_{1,-\eta
-2}^{2}(\mathbb{S}^{+})  \label{diraop}
\end{equation}%
is an isomorphism, where $S_{k,\delta }^{p}(\mathbb{S}^{+})$ denotes the
weighted Folland-Stein space of sections of $\mathbb{S}^{+}$ over $N$.
\end{proposition}

\begin{proof}
(sketch) By Corollary \ref{WFn2} (Weitzenbock formula for $n=2,$ dimension $%
= $ $2n+1$ $=$ $5),$ we have%
\begin{equation}
D_{\xi }^{2}=\nabla ^{\ast }\nabla +W  \label{WBF}
\end{equation}%
on $S_{3,-\eta }^{2}(\mathbb{S}^{+}),$ which is subelliptic. We can then
apply similar ideas in \cite{B} to show that (\ref{diraop}) is an
isomorphism (a subelliptic analogue of Proposition 2.2 in \cite{B}. We refer
the details to a separate paper \cite{Chiu}. \ 
\end{proof}

\begin{corollary}
\label{Sol} With the same assumptions and notations as in Proposition \ref%
{Draiso}, let $\psi _{0}\in $ $\mathbb{S}^{+}$ be a spinor field on $N$
which is constant near the infinity. Then there is a spinor field $\psi \in $
$\mathbb{S}^{+}$ such that 
\begin{eqnarray}
D_{\xi }^{2}\psi &=&0,  \label{MEq} \\
\psi -\psi _{0} &\in &S_{3,-4+\varepsilon }^{2}(\mathbb{S}^{+})\text{ for
small }\varepsilon >0.  \notag
\end{eqnarray}
\end{corollary}

\begin{proof}
The asymptotic conditions imply that connection forms acting on orthonormal
frame fields are $O(\rho ^{-5}).$ It follows that $D_{\xi }\psi _{0}$ $\in $ 
$S_{2,-5}^{2}(\mathbb{S}^{+})$ and hence $D_{\xi }^{2}\psi _{0}\in
S_{1,-6+\varepsilon }^{2}(\mathbb{S}^{+})$ for small $\varepsilon >0$.
Therefore by Proposition \ref{Draiso} we find a unique $\psi
_{-4+\varepsilon }\in S_{3,-4+\varepsilon }^{2}(\mathbb{S}^{+})$ by taking $%
\eta $ $=$ $4-\varepsilon $ such that $D_{\xi }^{2}\psi _{-4+\varepsilon
}=-D_{\xi }^{2}\psi _{0}$. Then $\psi =\psi _{0}$ $+$ $\psi _{-4+\varepsilon
}$ is the required spinor field.
\end{proof}

\begin{proposition}
\label{RegDO} (Regularity for the decay order) With the same assumptions and
notations as in Proposition \ref{Draiso} and Corollary \ref{Sol}, we have $%
D_{\xi }\psi $ $\in $ $S_{2,-6+\varepsilon }^{2}(\mathbb{S}^{+}).$
\end{proposition}

\begin{proof}
From Corollary \ref{Sol} and the asymptotic conditions, we learn that $%
D_{\xi }\psi $ $=$ $D_{\xi }\psi _{0}$ $+$ $D_{\xi }\psi _{-4+\varepsilon }$ 
$\in $ $S_{2,-5+\varepsilon }^{2}$ (omitting $\mathbb{S}^{+}).$ We want to
show that $D_{\xi }\psi $ gains one more decay order by applying the
following scale-broken estimate: for $u\in S_{2,\delta }^{2}(\mathbb{S}^{+})$
on $H_{2}$ it holds that%
\begin{equation}
||u||_{2,2,\delta }\leq C(||D_{\xi }^{2}u||_{0,2,\delta -2}+||u||_{L^{2}(B_{%
\bar{R}})})  \label{sbe}
\end{equation}%
where $B_{\bar{R}}$ is a Heisenberg ball of radius $\bar{R}$ in $H_{2}$ (see 
\cite{Chiu}). Let $\chi _{r}$ denote a cutoff function on $H_{2}$ such that $%
\chi _{r}=1$ on $B_{r^{2}}$ and $\chi _{r}=0$ on $H_{2}\backslash B_{2r^{2}}$
where $B_{a}$ denotes the Heisenberg ball of radius $a.$ Let $\varphi _{1/r}$
be a family of ($C^{\infty })$ smooth functions with compact support $\bar{B}%
_{1/r}$ $\subset $ $H_{2},$ tending to $\delta _{0},$ the delta function at
the origin as $r\rightarrow \infty .$ By identifying the end $N_{\infty }$
with $H_{2}\backslash B_{R}$ through asymptotic coordinates$,$ we extend $%
\psi $ into $B_{R}$ smoothly and denote the extension (now defined on $%
H_{2}) $ by $\tilde{\psi}.$ Consider%
\begin{equation*}
(D_{\xi }\tilde{\psi})_{(r)}:=(\chi _{r}(D_{\xi }\tilde{\psi}))\ast \varphi
_{1/r}\text{ on }H_{2}
\end{equation*}%
where "$\ast $" denotes the convolution with respect to the Heisenberg
multiplication (see \cite[Ch.10]{CS}). Observe that $(D_{\xi }\tilde{\psi}%
)_{(r)}$ has the following properties:%
\begin{eqnarray}
(1)\text{ }(D_{\xi }\tilde{\psi})_{(r)}(x) &\rightarrow &D_{\xi }\tilde{\psi}%
(x)\text{ as }r\rightarrow \infty \text{ since }\chi _{r}\rightarrow
1,\varphi _{1/r}\rightarrow \delta _{0},  \label{moll} \\
(2)\text{ (}D_{\xi }\tilde{\psi})_{(r)}(x)\text{ at }x &=&\infty \text{ has
any decay order for a given }r\text{.}  \notag
\end{eqnarray}%
Note that ($D_{\xi }\tilde{\psi})_{(r)}\in S_{2,\delta }^{2}(\mathbb{S}^{+})$
on $H_{2}$ for any $\delta \in \mathbb{R}$ by (2) in (\ref{moll}). So we can
substitute $u=(D_{\xi }\tilde{\psi})_{(r)}$ into (\ref{sbe}). Observing that 
$D_{\xi }^{2}\tilde{\psi}=D_{\xi }^{2}\psi =0$ on $H_{2}\backslash
B_{R}\cong N_{\infty }$, we compute%
\begin{eqnarray*}
D_{\xi }(\chi _{r}(D_{\xi }\tilde{\psi})) &=&e_{a}(\chi _{r})e_{a}\cdot
D_{\xi }\tilde{\psi}+\chi _{r}D_{\xi }^{2}\tilde{\psi} \\
&=&e_{a}(\chi _{r})e_{a}\cdot D_{\xi }\tilde{\psi}\text{ for }x\in
H_{2}\backslash B_{R}\cong N_{\infty }
\end{eqnarray*}%
and 
\begin{eqnarray}
D_{\xi }^{2}(\chi _{r}(D_{\xi }\tilde{\psi})) &=&e_{b}e_{a}(\chi
_{r})e_{b}\cdot e_{a}\cdot D_{\xi }\tilde{\psi}  \label{D2-0} \\
&&+e_{a}(\chi _{r})e_{b}\cdot \nabla _{e_{b}}^{LC}e_{a}\cdot D_{\xi }\tilde{%
\psi}  \notag \\
&&+e_{a}(\chi _{r})e_{b}\cdot e_{a}\cdot \nabla _{e_{b}}(D_{\xi }\tilde{\psi}%
).  \notag
\end{eqnarray}%
Given a point $p,$ we can choose an orthonormal frame field such that $%
\nabla _{e_{b}}^{LC}e_{a}$ $=$ $0$ at $p.$ Together with $e_{b}e_{a}$ $=$ $%
-e_{a}e_{b}-2\delta _{ab}$ we reduce (\ref{D2-0}) to%
\begin{eqnarray}
D_{\xi }^{2}(\chi _{r}(D_{\xi }\tilde{\psi})) &=&e_{b}e_{a}(\chi
_{r})e_{b}\cdot e_{a}\cdot D_{\xi }\tilde{\psi}  \label{D2-1} \\
&&+e_{a}(\chi _{r})e_{a}\cdot D_{\xi }^{2}\tilde{\psi}-2e_{a}(\chi
_{r})\nabla _{e_{a}}(D_{\xi }\tilde{\psi})  \notag \\
&=&e_{b}e_{a}(\chi _{r})e_{b}\cdot e_{a}\cdot D_{\xi }\tilde{\psi}%
-2e_{a}(\chi _{r})\nabla _{e_{a}}(D_{\xi }\tilde{\psi})  \notag
\end{eqnarray}%
for $p\in H_{2}\backslash B_{R}\cong N_{\infty }.$ Note that in (\ref{D2-1}) 
$e_{a}(\chi _{r})$ $\sim $ $\frac{1}{r^{2}},$ $e_{b}e_{a}(\chi _{r})$ $\sim $
$\frac{1}{r^{3}}$ as $r$ large and $D_{\xi }\tilde{\psi}$ (resp. $\nabla
_{e_{a}}(D_{\xi }\tilde{\psi}))$ decays at $\infty $ in order $%
-5+\varepsilon $ (resp. $-6+\varepsilon )$. It follows that $||D_{\xi
}^{2}(\chi _{r}(D_{\xi }\tilde{\psi}))||_{0,2,-8+\varepsilon }$ $\leq $ $C,$
independent of $r.$ So (see \cite[Ch.10]{CS})) in view of (\ref{D2-1}) it
holds that%
\begin{equation*}
D_{\xi ,x}^{2}(D_{\xi }\tilde{\psi})_{(r)}(x)=\int_{y\in B_{1/r}\subset
H_{2}}D_{\xi ,x}^{2}(\chi _{r}(D_{\xi }\tilde{\psi}))(xy^{-1})\varphi
_{1/r}(y)dV(y)
\end{equation*}%
tends to $0$ (resp. $D_{\xi ,x}^{2}(D_{\xi }\tilde{\psi})(x))$ for $x$ $\in $
$H_{2}\backslash B_{R}\cong N_{\infty }$ (resp. for $x$ $\in $ $B_{R}$) as $%
r\rightarrow \infty $ and it converges in the norm $||$ $\cdot $ $%
||_{0,2,-8+\varepsilon }.$ Here we have used the Heisenberg translation
invariance of $D_{\xi }.$ On the other hand, it is easy to get $(D_{\xi }%
\tilde{\psi})_{(r)}\rightarrow D_{\xi }\tilde{\psi}$ in $L^{2}(B_{\bar{R}})$
as $r\rightarrow \infty .$ Now for $r_{1},r_{2}$ large enough, by (\ref{sbe}%
) we have%
\begin{eqnarray}
&&||(D_{\xi }\tilde{\psi})_{(r_{1})}-(D_{\xi }\tilde{\psi}%
)_{(r_{2})}||_{2,2,-6+\varepsilon }  \label{sbe-1} \\
&\leq &C(||D_{\xi }^{2}(D_{\xi }\tilde{\psi})_{(r_{1})}-D_{\xi }^{2}(D_{\xi }%
\tilde{\psi})_{(r_{2})}||_{0,2,-8+\varepsilon }+||(D_{\xi }\tilde{\psi}%
)_{(r_{1})}-(D_{\xi }\tilde{\psi})_{(r_{2})}||_{L^{2}(B_{\bar{R}})}.  \notag
\end{eqnarray}%
But the right hand side of (\ref{sbe-1}) is small since $D_{\xi }^{2}(D_{\xi
}\tilde{\psi})_{(r)}$ and $(D_{\xi }\tilde{\psi})_{(r)}$ are Cauchy in $r$
with respect to their respective norms. Therefore \{($D_{\xi }\tilde{\psi}%
)_{(r)}\}_{r}$ is a Cauchy sequence in $r$ with respect to the norm $||\cdot
||_{2,2,-6+\varepsilon }.$ Thus as $r\rightarrow \infty ,$ ($D_{\xi }\tilde{%
\psi})_{(r)}\rightarrow D_{\xi }\tilde{\psi}$ $\in $ $S_{2,-6+\varepsilon
}^{2}.$ Since $D_{\xi }\psi $ $=$ $D_{\xi }\tilde{\psi}$ on $H_{2}\backslash
B_{R}\cong N_{\infty }.$ It follows that $D_{\xi }\psi $ $\in $ $%
S_{2,-6+\varepsilon }^{2}.$
\end{proof}

Choose a constant spinor $\psi _{0}$ $\in $ $\mathbb{S}^{+}$ with $|\psi
_{0}|=1$ at infinity, and extend it to a smooth spinor on the whole space $N$%
. By (\ref{MEq}) in Corollary \ref{Sol} we can find a spinor field $\psi $ $%
= $ $\psi _{0}$ $+$ $\psi _{-4+\varepsilon }$ satisfying $D_{\xi }^{2}\psi
=0 $ and $\psi _{-4+\varepsilon }$ $\in $ $S_{2,-4+\varepsilon }^{2}(\mathbb{%
S}^{+})$. Now applying the Weitzenbock formula (\ref{WF5D}) to $\psi $ and
integrating by parts over the region $N_{R}=\{\rho \leq R\}$ (we abuse the
notation; more accurately $N_{R}$ $=$ $N\backslash N_{\infty }$, $N_{\infty
} $ is diffeomorphic to $\{\rho >R\}$ in the Heisenberg group), we have

\begin{equation}
\begin{split}
& \int_{N_{R}}|\nabla \psi |^{2}+W|\psi |^{2}dV\text{ (}dV:=\theta \wedge
(d\theta )^{2}\text{)} \\
=& Re\int_{S_{R}}\big<\psi ,\nabla _{i}\psi \big>e_{i}\lrcorner dV\text{ (}%
S_{R}:=\partial N_{R}=\{\rho =R\}\text{)} \\
=& Re\int_{S_{R}}\left( \big<\psi _{0},\nabla _{i}\psi _{0}\big>+\big<\psi
_{0},\nabla _{i}\psi _{-4+\varepsilon }\big>+\big<\psi _{-4+\varepsilon
},\nabla _{i}\psi _{0}\big>+\big<\psi _{-4+\varepsilon },\nabla _{i}\psi
_{-4+\varepsilon }\big>\right) e_{i}\lrcorner dV.
\end{split}
\label{pmt1}
\end{equation}

Since 
\begin{equation}
\begin{split}
\big<\psi _{0},[e_{j},e_{k}]\psi _{0}\big>& =\big<\psi
_{0},(e_{j}e_{k}-e_{k}e_{j})\psi _{0}\big> \\
& =\big<e_{j}\psi _{0},-e_{k}\psi _{0}\big>+\big<e_{k}\psi _{0},e_{j}\psi
_{0}\big> \\
& =\big<e_{k}e_{j}\psi _{0},\psi _{0}\big>+\big<-e_{j}e_{k}\psi _{0},\psi
_{0}\big> \\
& =\big<\lbrack e_{k},e_{j}]\psi _{0},\psi _{0}\big>=-\big<\lbrack
e_{j},e_{k}]\psi _{0},\psi _{0}\big> \\
& =-\overline{\big<\psi _{0},[e_{j},e_{k}]\psi _{0}\big>},
\end{split}%
\end{equation}%
we have $Re\big<\psi _{0},[e_{j},e_{k}]\psi _{0}\big>=0$, and hence the
first term in the RHS of (\ref{pmt1}) vanishes as $R\rightarrow \infty $.
Furthermore, since $\psi _{-4+\varepsilon }=O(\rho ^{-4+\varepsilon }),\
\nabla \psi _{-4+\varepsilon }=O(\rho ^{-5+\varepsilon })$ and $\nabla \psi
_{0}=O(\rho ^{-5+\varepsilon })$, the third and the fourth terms in the RHS
of (\ref{pmt1}) also vanish.

Finally, we are going to show that the second term in the RHS of (\ref{pmt1}%
) will catch the p-mass as $R\rightarrow \infty $. For this, let $L_{i}$
denote the operator 
\begin{equation}
L_{i}=\nabla _{i}+e_{i}D_{\xi }.
\end{equation}%
Noting that $e_{i}e_{j}=\frac{1}{2}[e_{i},e_{j}]-\delta _{ij}$, we have 
\begin{equation}
L_{i}=(\delta _{ij}+e_{i}e_{j})\nabla _{j}=\frac{1}{2}[e_{i},e_{j}]\nabla
_{j}.
\end{equation}%
Let $\alpha $ denote the 3-form $\big<\lbrack e_{i},e_{j}]\psi _{0},\psi
_{-4+\varepsilon }\big>e_{i}\lrcorner e_{j}\lrcorner dV$. Then 
\begin{equation}
d\alpha =-4\left( \big<L_{i}\psi _{0},\psi _{-4+\varepsilon }\big>-\big<\psi
_{0},L_{i}\psi _{-4+\varepsilon }\big>\right) e_{i}\lrcorner dV.
\end{equation}%
Therefore, by Stokes' theorem we have 
\begin{equation}
\int_{S_{R}}\big<L_{i}\psi _{0},\psi _{-4+\varepsilon }\big>e_{i}\lrcorner
dV.=\int_{S_{R}}\big<\psi _{0},L_{i}\psi _{-4+\varepsilon }\big>%
e_{i}\lrcorner dV,
\end{equation}%
hence the second term in the RHS of (\ref{pmt1}) becomes 
\begin{equation}
\begin{split}
-Re\int_{S_{R}}\big<\psi _{0},\nabla _{i}\psi _{-4+\varepsilon }\big>%
e_{i}\lrcorner dV& =Re\int_{S_{R}}\big<\psi _{0},(e_{i}D_{\xi }-L_{i})\psi
_{-4+\varepsilon }\big>e_{i}\lrcorner dV \\
& =Re\int_{S_{R}}\left( \big<\psi _{0},-e_{i}D_{\xi }\psi _{0}\big>-\big<%
L_{i}\psi _{0},\psi _{-4+\varepsilon }\big>+O(\rho ^{-6+\varepsilon
})\right) e_{i}\lrcorner dV,
\end{split}
\label{pmt2}
\end{equation}%
where in the last equality, we have used the fact $D_{\xi }\psi
_{-4+\varepsilon }=-D_{\xi }\psi _{0}$ + $O(\rho ^{-6+\varepsilon })$ by
Proposition \ref{RegDO}. As before, $\big<L_{i}\psi _{0},\psi
_{-4+\varepsilon }\big>=O(\rho ^{-9+2\varepsilon })$, so the second and
third terms in the RHS of (\ref{pmt2}) vanishes as $R\rightarrow \infty $.
On the other hand 
\begin{equation}
e_{i}D_{\xi }\psi _{0}=e_{i}e_{k}\nabla _{k}\psi _{0}=e_{i}e_{k}\left(
e_{k}\psi _{0}-\frac{1}{4}\sum_{m,l=1}^{2n}\omega
_{m}{}^{l}(e_{k})e_{l}e_{m}\psi _{0}\right) .  \label{pmt3}
\end{equation}%
Substituting (\ref{pmt3}) into (\ref{pmt2}), we get 
\begin{equation}
\begin{split}
& \ \ \ -Re\int_{S_{R}}\big<\psi _{0},\nabla _{i}\psi _{-4+\varepsilon }\big>%
e_{i}\lrcorner dV \\
& =-\int_{S_{R}}\frac{1}{4}\sum_{m,l=1}^{2n}\omega _{m}{}^{l}(e_{k})Re\big<%
\psi _{0},e_{i}e_{k}e_{l}e_{m}\psi _{0}\big>e_{i}\lrcorner dV \\
& =\int_{S_{R}}\frac{1}{4}\omega _{k}{}^{i}(e_{k})|\psi
_{0}|^{2}e_{i}\lrcorner dV+\sum_{m\neq i}\sum_{l\neq k}\int_{S_{R}}\frac{1}{4%
}\omega _{l}{}^{m}(e_{k})Re\big<\psi _{0},e_{i}e_{k}e_{l}e_{m}\psi _{0}\big>%
e_{i}\lrcorner dV,
\end{split}
\label{pmt4}
\end{equation}%
where the first integral on the RHS of (\ref{pmt4}) is taken over all terms
with {m=i} and $l=k$ in the middle term, and the last integral in (\ref{pmt4}%
) is taken over all terms with either $m\neq i$ or $l\neq k$, and hence $%
m\neq i$ and $l\neq k$ (This is because that $Re\big<\psi ,[e_{j},e_{k}]\psi %
\big>=0$ for any spinor $\psi $). In addition, since $\omega _{l}{}^{l}=0$,
the last integral survives only for $m\neq i,l\neq k,l\neq m$ as well as $%
i\neq k$. It follows that 
\begin{equation}
\begin{split}
& \sum_{m\neq i}\sum_{l\neq k}\int_{S_{R}}\frac{1}{4}\omega
_{l}{}^{m}(e_{k})Re\big<\psi _{0},e_{i}e_{k}e_{l}e_{m}\psi _{0}\big>%
e_{i}\lrcorner dV \\
=& \sum_{i\neq k,i\neq m,l\neq k,l\neq m}\int_{S_{R}}\frac{1}{4}\omega
_{l}{}^{m}(e_{k})Re\big<\psi _{0},e_{i}e_{k}e_{l}e_{m}\psi _{0}\big>%
e_{i}\lrcorner dV \\
=& \int_{S_{R}}\frac{1}{4}\omega _{k}{}^{i}(e_{k})|\psi
_{0}|^{2}e_{i}\lrcorner dV \\
& \ +\sum_{(i,k,l.m)\in I_{(2)}}\int_{S_{R}}\frac{1}{4}\omega
_{l}{}^{m}(e_{k})Re\big<\psi _{0},e_{i}e_{k}e_{l}e_{m}\psi _{0}\big>%
e_{i}\lrcorner dV
\end{split}
\label{pmt5}
\end{equation}%
where the index set $I_{(n)}$ is defined as follows:%
\begin{eqnarray}
I_{(n)} &:=&\{(i,k,l,m):i\neq k,i\neq l,i\neq m,k\neq l,k\neq m,
\label{index} \\
&&\ \ \ \ \ l \neq m,\text{ }1\leq i,k,l,m\leq 2n\}.  \notag
\end{eqnarray}%
Substituting (\ref{pmt5}) into (\ref{pmt4}) gives 
\begin{equation}
\begin{split}
& \ \ \ -Re\int_{S_{R}}\big<\psi _{0},\nabla _{i}\psi _{-4+\varepsilon }\big>%
e_{i}\lrcorner dV \\
& =-\int_{S_{R}}\frac{1}{4}\sum_{m,l=1}^{2n}\omega _{m}{}^{l}(e_{k})Re\big<%
\psi _{0},e_{i}e_{k}e_{l}e_{m}\psi _{0}\big>e_{i}\lrcorner dV \\
& =\int_{S_{R}}\frac{1}{2}\omega _{k}{}^{i}(e_{k})|\psi
_{0}|^{2}e_{i}\lrcorner dV \\
& \ +\sum_{(i,k,l.m)\in I_{(2)}}\int_{S_{R}}\frac{1}{4}\omega
_{l}{}^{m}(e_{k})Re\big<\psi _{0},e_{i}e_{k}e_{l}e_{m}\psi _{0}\big>%
e_{i}\lrcorner dV
\end{split}
\label{pmt6}
\end{equation}

Putting (\ref{pmt6}) into (\ref{pmt1}), letting $R\rightarrow \infty $ and
using (\ref{af11}), (\ref{af21}) and Lemma \ref{pmt7} below, we obtain a
Witten-type formula%
\begin{eqnarray}
\int_{N}|\nabla \psi |^{2}+W|\psi |^{2}dV &=&\frac{1}{2}\widetilde{m}%
(J,\theta )+16(c_{2}-\tilde{c}_{2})\alpha _{2}\Omega _{2}A  \label{WTF} \\
&=&\Big(16\big[(2\sqrt{2}+1)c_{2}+(\sqrt{2}-1)\tilde{c}_{2}\big]\alpha
_{2}\Omega _{2}\Big)A.  \notag
\end{eqnarray}%
This shows that $A\geq 0$, and hence the p-mass is nonnegative by (\ref{af12}%
). In the proof above we have used the following result.

\begin{lemma}
\label{pmt7} For the case $n=2$, it holds that%
\begin{equation}
\begin{split}
& \sum_{(i,k,l,m)\in I_{(2)}}\int_{S_{R}}\frac{1}{4}\omega
_{l}{}^{m}(e_{k})Re\big<\psi _{0},e_{i}e_{k}e_{l}e_{m}\psi _{0}\big>%
e_{i}\lrcorner dV \\
=& 16(c_{2}-\tilde{c}_{2})\alpha _{2}\Omega _{2}A,\ \ \text{as}\
R\rightarrow \infty ,
\end{split}
\label{4-18}
\end{equation}
\end{lemma}

\begin{remark}
\label{r-4-18} For blow-ups through the Green's function, we have $c_{2}>%
\tilde{c}_{2}$. Note that (\ref{4-18}) may not hold in general for $n>2.$
\end{remark}

\begin{proof}
First for any positive spinor $\psi $, $(e_{1}e_{3}+e_{2}e_{4})\psi =0$
implies 
\begin{equation*}
\begin{split}
\big<\psi ,e_{1}e_{3}e_{2}e_{4}\psi \big>& =\big<-e_{1}e_{3}\psi
,e_{2}e_{4}\psi \big> \\
& =\big<e_{2}e_{4}\psi ,e_{2}e_{4}\psi \big> \\
& =\big<\psi ,e_{4}e_{2}e_{2}e_{4}\psi \big>=|\psi |^{2}.
\end{split}%
\end{equation*}%
Therefore $Re\big<\psi _{0},e_{i}e_{k}e_{l}e_{m}\psi _{0}\big>=\pm |\psi
_{0}|^{2}$ for any indices $i,k,l,m$ such that each one is different from
others. Let 
\begin{equation*}
I=\sum_{(i,k,l.m)\in I_{(2)}}\int_{S_{R}}\frac{1}{4}\omega
_{l}{}^{m}(e_{k})Re\big<\psi _{0},e_{i}e_{k}e_{l}e_{m}\psi _{0}\big>%
e_{i}\lrcorner dV
\end{equation*}%
(see (\ref{index}) for the definition of the index set $I_{(2)}$). Then we
have 
\begin{equation}
\begin{split}
I& =\frac{1}{2}\int_{S_{R}}\left[ (-\omega _{1}{}^{2}(e_{3})+\omega
_{1}{}^{3}(e_{2})-\omega _{2}{}^{3}(e_{1}))\omega ^{2}\right] |\psi
_{0}|^{2}\wedge \theta \wedge \omega ^{1}\wedge \omega ^{3} \\
& +\frac{1}{2}\int_{S_{R}}\left[ (\omega _{1}{}^{3}(e_{4})-\omega
_{1}{}^{4}(e_{3})+\omega _{3}{}^{4}(e_{1}))\omega ^{4}\right] |\psi
_{0}|^{2}\wedge \theta \wedge \omega ^{1}\wedge \omega ^{3} \\
& +\frac{1}{2}\int_{S_{R}}\left[ (\omega _{1}{}^{2}(e_{4})-\omega
_{1}{}^{4}(e_{2})+\omega _{2}{}^{4}(e_{1}))\omega ^{1}+\right] |\psi
_{0}|^{2}\wedge \theta \wedge \omega ^{2}\wedge \omega ^{4} \\
& +\frac{1}{2}\int_{S_{R}}\left[ (\omega _{2}{}^{4}(e_{3})-\omega
_{2}{}^{3}(e_{4})-\omega _{3}{}^{4}(e_{2}))\omega ^{3}\right] |\psi
_{0}|^{2}\wedge \theta \wedge \omega ^{2}\wedge \omega ^{4},
\end{split}
\label{pmt8}
\end{equation}%
where 
\begin{equation}
\begin{split}
& 2(-\omega _{1}{}^{2}(e_{3})+\omega _{1}{}^{3}(e_{2})-\omega
_{2}{}^{3}(e_{1})) \\
& =-(\theta _{1}{}^{2}+\theta _{\bar{1}}{}^{\bar{2}})(iZ_{1}-iZ_{\bar{1}%
})-2i\theta _{1}{}^{1}(Z_{2}+Z_{\bar{2}})+i(\theta _{2}{}^{1}-\theta _{\bar{2%
}}{}^{\bar{1}})(Z_{1}+Z_{\bar{1}}) \\
& =2i(\theta _{1}{}^{2}(Z_{\bar{1}})+\theta _{2}{}^{1}(Z_{1})-\theta
_{1}{}^{1}(Z_{2})-\theta _{1}{}^{1}(Z_{\bar{2}})) \\
& =2i(\theta _{1}{}^{2}(\mathring{Z}_{\bar{1}})+\theta _{2}{}^{1}(\mathring{Z%
}_{1})-\theta _{1}{}^{1}(\mathring{Z}_{2})-\theta _{1}{}^{1}(\mathring{Z}_{%
\bar{2}}))+\ \text{h.d.o.t.\ (by }(\ref{af2})) \\
& =2i(C_{1}{}^{2}{}_{\bar{1}%
}+B_{2}{}^{1}{}_{1}-B_{1}{}^{1}{}_{2}-C_{1}{}^{1}{}_{\bar{2}})+\ \text{%
h.d.o.t.\ (by}\ (\ref{af6})) \\
& =\frac{2\sqrt{2}}{\rho ^{8}}(c_{2}-\tilde{c}_{2})A(z^{\bar{2}}\bar{\omega}%
+z^{2}\omega )+\ \text{h.d.o.t.\ (by}\ (\ref{af7}),(\ref{af8}))
\end{split}
\label{pmt9}
\end{equation}%
(h.d.o.t. means "higher decay order term(s)"). Similarly, we have 
\begin{equation}
\begin{split}
& 2(\omega _{1}{}^{3}(e_{4})-\omega _{1}{}^{4}(e_{3})+\omega
_{3}{}^{4}(e_{1})) \\
& =\frac{2\sqrt{2}i}{\rho ^{8}}(c_{2}-\tilde{c}_{2})A(z^{\bar{2}}\bar{\omega}%
-z^{2}\omega )+\ \text{h.d.o.t.},
\end{split}
\label{pmt10}
\end{equation}%
\begin{equation}
\begin{split}
& 2(\omega _{1}{}^{2}(e_{4})-\omega _{1}{}^{4}(e_{2})+\omega
_{2}{}^{4}(e_{1})) \\
& =\frac{2\sqrt{2}}{\rho ^{8}}(c_{2}-\tilde{c}_{2})A(z^{\bar{1}}\bar{\omega}%
+z^{1}\omega )+\ \text{h.d.o.t.},
\end{split}
\label{pmt11}
\end{equation}%
\begin{equation}
\begin{split}
& 2(\omega _{2}{}^{4}(e_{3})-\omega _{2}{}^{3}(e_{4})+\omega
_{3}{}^{4}(e_{2})) \\
& =\frac{2\sqrt{2}i}{\rho ^{8}}(c_{2}-\tilde{c}_{2})A(z^{\bar{1}}\bar{\omega}%
-z^{1}\omega )+\ \text{h.d.o.t.}.
\end{split}
\label{pmt12}
\end{equation}%
Substituting (\ref{pmt9}),(\ref{pmt10}),(\ref{pmt11}),(\ref{pmt12}) into (%
\ref{pmt8}), we have, modulo the higher decay order terms (h.d.o.t. in
short), 
\begin{equation}
\begin{split}
I& =\frac{(c_{2}-\tilde{c}_{2})}{\sqrt{2}}A\int_{S_{R}}|\psi _{0}|^{2}\frac{%
(z^{\bar{2}}\bar{\omega}\theta ^{2}+z^{2}\omega \theta ^{\bar{2}})\wedge
\theta \wedge \frac{i}{2}\theta ^{1}\wedge \theta ^{\bar{1}}+(z^{\bar{1}}%
\bar{\omega}\theta ^{1}+z^{1}\omega \theta ^{\bar{1}})\wedge \theta \wedge 
\frac{i}{2}\theta ^{2}\wedge \theta ^{\bar{2}}}{\rho ^{8}} \\
& =\frac{(c_{2}-\tilde{c}_{2})}{\sqrt{2}}A\int_{S_{R}}|\psi _{0}|^{2}\frac{%
\sum_{\beta =1}^{2}(z^{\bar{\beta}}\bar{\omega}\theta ^{\beta }+z^{\beta
}\omega \theta ^{\bar{\beta}})}{\rho ^{8}}\wedge \frac{i}{2}\theta \wedge
(\theta ^{1}\wedge \theta ^{\bar{1}}+\theta ^{2}\wedge \theta ^{\bar{2}}) \\
& =\frac{(c_{2}-\tilde{c}_{2})A}{2}\int_{S_{R}}|\psi _{0}|^{2}\frac{%
\sum_{\beta =1}^{2}(z^{\bar{\beta}}\bar{\omega}dz^{\beta }+z^{\beta }\omega
dz^{\bar{\beta}})}{\rho ^{8}}\wedge \mathring{\theta}\wedge d\mathring{\theta%
}\ (\text{by}\ (\ref{af1}),\text{ }(\ref{af25})).
\end{split}%
\end{equation}%
Therefore, by (\ref{af17}) we have 
\begin{equation}
\lim_{R\rightarrow \infty }I=16(c_{2}-\tilde{c}_{2})\alpha _{2}\Omega _{2}A.
\end{equation}
\end{proof}

\subsection{The case $m=0$}

Next, we prove that if $m(J,\theta )=0$, then $N$ is isomorphic to the
Heisenberg group $H_{2}$ as a pseudohermitian manifold.

\begin{lemma}
\label{le01} The p-mass $m(J,\theta )=0$ implies the Tanaka-Webster scalar
curvature $W=0$.
\end{lemma}

\begin{proof}
This is immediately obtained from the Witten-type formula (\ref{WTF}) and
note that $W\geq 0$.
\end{proof}

\begin{lemma}
\label{le02} The p-mass $m(J,\theta )=0$ implies the torsion (forms) $\tau
^{\beta }=A^{\beta }{}_{\bar{\gamma}}\theta ^{\bar{\gamma}}\equiv 0$ for all 
$\beta .$
\end{lemma}

\begin{proof}
Motivated by the idea of Schoen and Yau in \cite{SY}, we consider the flow $%
\varphi _{s}$ generated by the Reeb vector field $T$ of $N$, and set $%
J_{s}=\varphi _{s}^{\ast }J$.\ ($\dot{J}=L_{T}J=2iA^{\beta }$ $_{\bar{\alpha}%
}\theta ^{\bar{\alpha}}\otimes Z_{\beta }-2iA^{\bar{\beta}}$ $_{\alpha
}\theta ^{\alpha }\otimes Z_{\bar{\beta}},.$see (\ref{LTJ})). We need the
following proposition.

\begin{proposition}
\label{pr4-6} Suppose $m(J,\theta )=0.$ For $s$, $|s|$ small enough, there
is a unique positive function $u_{s}$ on $N$ such that $(N,J_{s},u_{s}^{2/n}%
\theta )$ is asymptotically flat with zero Tanaka-Webster curvature, and 
\begin{equation}
m(J_{s},u_{s}^{2/n}\theta )=C_{n}^{\prime }\int_{N}W_{s}u_{s}\ \theta \wedge
(d\theta )^{n}  \label{massofdef}
\end{equation}%
\noindent for some negative constant $C_{n}^{\prime }$, where $W_{s}$
denotes the Tanaka-Webster scalar curvature with respect to $(J_{s},\theta
). $
\end{proposition}

\begin{proof}
(\textbf{of Proposition \ref{pr4-6}}) In order for the structure $%
(J_{s},u_{s}^{2/n}\theta )$ to be scalar flat, the function $u_{s}$ must
satisfy 
\begin{equation}
b_{n}\Delta _{b(s)}u_{s}+W_{s}u_{s}=0
\end{equation}%
where $\Delta _{b(s)}$ denotes the sublaplacian with respect to $%
(J_{s},\theta )$. Let $W_{s-}$ denote the negative part of $W_{s},$ meaning
that $W_{s-}(x)$ $:=$ $\min \{W_{s}(x),0\}.$ Write $dV_{\theta }$ :$=$ $%
\theta \wedge (d\theta )^{n}.$ For $|s|$ small enough, $W_{s}$ is not too
negative in the sense that 
\begin{equation}
\left( \int_{N}|W_{s-}|^{n+1}\ dV_{\theta }\right) ^{\frac{1}{n+1}}\leq
\varepsilon _{0}  \label{enercon}
\end{equation}%
for some given small $\varepsilon _{0}$ (which we will specify later). Using
(\ref{enercon}) and the Sobolev type inequality ($p=b_{n}=2+\frac{2}{n})$%
\begin{equation}
\left( \int_{N}|u|^{p}\ dV_{\theta }\right) ^{1/p}\leq C\left(
\int_{N}|\nabla _{b(s)}u|^{2}\ dV_{\theta }\right) ^{1/2}  \label{soboineq}
\end{equation}%
for any function $u$ with compact support on $N$, we have 
\begin{equation}
\begin{split}
& \int_{N}\big<(b_{n}\Delta _{b(s)}+W_{s})u,u\big>\ dV_{\theta } \\
=& b_{n}\int_{N}|\nabla _{b(s)}u|^{2}\ dV_{\theta }+\int_{N}W_{s}u^{2}\
dV_{\theta } \\
\geq & b_{n}\int_{N}|\nabla _{b(s)}u|^{2}\ dV_{\theta
}-\int_{N}|W_{s-}|u^{2}\ dV_{\theta } \\
\geq & b_{n}\int_{N}|\nabla _{b(s)}u|^{2}\ dV_{\theta }-\left(
\int_{N}|W_{s-}|^{n+1}dV_{\theta }\right) ^{1/(n+1)}\left( \int_{N}u^{p}\
dV_{\theta }\right) ^{2/p} \\
\geq & (b_{n}-\varepsilon _{0}C^{2})\int_{N}|\nabla _{b(s)}u|^{2}\
dV_{\theta }.
\end{split}
\label{est01}
\end{equation}%
Choose $\varepsilon _{0}$ $>$ 0 small so that $b_{n}-\varepsilon _{0}C^{2}$ $%
>$ $0$. The estimate (\ref{est01}) implies that $b_{n}\Delta _{b(s)}+W_{s}$
is \textbf{coercive} for $|s|$ small. Therefore there exists a solution $%
v_{s}$ of 
\begin{equation}
b_{n}\Delta _{b(s)}v_{s}+W_{s}v_{s}=W_{s}  \label{est02}
\end{equation}%
on $N$, which decays to zero at infinity. More precisely, using estimates
similar to those of Lemma 3.2 in \cite{SY}, together with (\ref{gr2}), one
finds that 
\begin{equation}
v_{s}=\frac{1}{c\pi \rho ^{2n}}\int_{N}(W_{s}-W_{s}v_{s})\ dV_{\theta
}+O(\rho ^{-2n-1})  \label{est03}
\end{equation}%
near $\infty $ for some $c>0$. Let $u_{s}=1-v_{s}$. Then, by (\ref{est02}) $%
u_{s}$ is the unique positive function $u_{s}$ on $N$ satisfying $%
b_{n}\Delta _{b(s)}u_{s}+W_{s}u_{s}=0$. Therefore, $(N,J_{s},u_{s}^{2/n}%
\theta )$ is asymptotically flat with zero Tanaka-Webster curvature. It
follows from (\ref{est03}) that 
\begin{equation*}
u_{s}=1-\frac{1}{c\pi \rho ^{2n}}\int_{N}W_{s}u_{s}\ dV_{\theta }+O(\rho
^{-2n-1})
\end{equation*}%
and hence 
\begin{equation}
\begin{split}
u_{s}^{p-2}\theta & =\left( 1-\frac{1}{c\pi \rho ^{2n}}\int_{N}W_{s}u_{s}\
dV_{\theta }+O(\rho ^{-2n-1})\right) ^{p-2} \\
& \ \ \ \left( \big(1+c_{n}A\rho ^{-2n}+O(\rho ^{-2n-1})\big)\mathring{\theta%
}+O(\rho ^{-2n-1})_{\beta }dz^{\beta }+O(\rho ^{-2n-1})_{\bar{\beta}}dz^{%
\bar{\beta}}\right) \\
& =\left( 1-\left[ \frac{(p-2)}{c\pi }\int_{N}W_{s}u_{s}\ dV_{\theta }\right]
\rho ^{-2n}+O(\rho ^{-2n-1})\right) \mathring{\theta} \\
& \ \ \ \ \ \ \ +O(\rho ^{-2n-1})_{\beta }dz^{\beta }+O(\rho ^{-2n-1})_{\bar{%
\beta}}dz^{\bar{\beta}}.
\end{split}
\label{est04}
\end{equation}%
For the second equality of (\ref{est04}), we have used $A=0$ by (\ref{af12})
since $m(J,\theta )$ $=$ $0$ by assumption. Therefore, from (\ref{est04})
the p-mass formula (\ref{massofdef}) follows by comparing (\ref{af12}) with (%
\ref{af1}).
\end{proof}

Now we proceed to prove lemma \ref{le02}. Generalizing \cite[(2.20)]{CLee}
to higher dimensions, we have 
\begin{equation}
\frac{d}{ds}\Big|_{s=0}W_{s}=\sum_{\alpha ,\gamma }i(E_{\alpha \gamma ,\bar{%
\gamma}\bar{\alpha}}-E_{\bar{\alpha}\bar{\gamma},\gamma \alpha
})-n\sum_{\alpha ,\gamma }(A_{\bar{\alpha}\bar{\gamma}}E_{\alpha \gamma
}+A_{\alpha \gamma }E_{\bar{\alpha}\bar{\gamma}})  \label{est05}
\end{equation}%
where $\frac{d}{ds}\Big|_{s=0}J_{s}=2E=2E_{\gamma }{}^{\bar{\beta}}\theta
^{\gamma }\otimes Z_{\bar{\beta}}+2E_{\bar{\gamma}}{}^{\beta }\theta ^{\bar{%
\gamma}}\otimes Z_{\beta }$ with $E_{\alpha \gamma }=-iA_{\alpha \gamma }$.
From (\ref{massofdef}) and (\ref{est05}), we have

\begin{equation}
\begin{split}
\frac{d}{ds}\Big|_{s=0}m(J_{s},u_{s}^{2/n}\theta )& =C_{n}^{\prime }\int_{N}(%
\frac{d}{ds}\Big|_{s=0}W_{s}\cdot u_{0}+W_{0}\frac{d}{ds}\Big|_{s=0}u_{s})\
dV_{\theta } \\
& =C_{n}^{\prime }\int_{N}(\frac{d}{ds}\Big|_{s=0}W_{s}\ dV_{\theta },\ \ \
(u_{0}=1,\ \ W_{0}=0), \\
& =-2nC_{n}^{\prime }\int_{N}\sum_{\alpha ,\gamma }|A_{\alpha \gamma }|^{2}\
dV_{\theta } \\
& >0\ \ \ (C_{n}^{\prime }<0)
\end{split}
\label{m_var}
\end{equation}%
if some $A_{\alpha \gamma }$ $\neq $ $0$ at some point, where for the third
equality we have used the divergence theorem and the decay order of the
torsion. Therefore, if the torsion does not vanish identically, we can
construct an asymptotically flat pseudohermitian manifold $%
(N,J_{s},u_{s}^{2/n}\theta )$ for $s<0,$ $|s|$ small with zero
Tanaka-Webster curvature and negative p-mass by (\ref{m_var}). This
contradicts (\ref{WTF}) and Lemma \ref{le02} follows.
\end{proof}

\begin{lemma}
\label{le03} The p-mass $m(J,\theta )=0$ implies the pseudohermitian
curvature $R_{\alpha \bar{\beta}\rho \bar{\sigma}}$ $\equiv $ $0$.
\end{lemma}

\begin{proof}
Recall (\cite{Le2}) that if $A_{\alpha \beta }=0$ (implied by Lemma \ref%
{le02}), then we have the Bianchi identity 
\begin{equation}
\begin{split}
R_{\alpha \bar{\beta}\rho \bar{\sigma},\gamma }-R_{\alpha \bar{\beta}\gamma 
\bar{\sigma},\rho }& =0; \\
R_{\rho \bar{\sigma},\gamma }-R_{\gamma \bar{\sigma},\rho }& =0,
\end{split}
\label{biid1}
\end{equation}%
and the contracted identity 
\begin{equation}
R_{\gamma \bar{\sigma},\sigma }=0.  \label{biid2}
\end{equation}%
That $N$ is spherical implies 
\begin{equation}
\begin{split}
0& =S_{\alpha \bar{\beta}\rho \bar{\sigma}} \\
& =R_{\alpha \bar{\beta}\rho \bar{\sigma}}-\frac{1}{n+2}(R_{\alpha \bar{\beta%
}}h_{\rho \bar{\sigma}}+R_{\rho \bar{\beta}}h_{\alpha \bar{\sigma}}+\delta
_{\alpha }{}^{\beta }R_{\rho \bar{\sigma}}+\delta _{\rho }{}^{\beta
}R_{\alpha \bar{\sigma}}) \\
& \ \ \ +\frac{W}{(n+1)(n+2)}(\delta _{\alpha }{}^{\beta }h_{\rho \bar{\sigma%
}}+\delta _{\rho }{}^{\beta }h_{\alpha \bar{\sigma}}).
\end{split}
\label{biid6}
\end{equation}%
Taking covariant derivative of (\ref{biid6}), we obtain via Lemma \ref{le01}
($W=0)$ 
\begin{equation}
R_{\alpha \bar{\beta}\rho \bar{\sigma},\gamma }=\frac{1}{n+2}(R_{\alpha \bar{%
\beta},\gamma }h_{\rho \bar{\sigma}}+R_{\rho \bar{\beta},\gamma }h_{\alpha 
\bar{\sigma}}+\delta _{\alpha }{}^{\beta }R_{\rho \bar{\sigma},\gamma
}+\delta _{\rho }{}^{\beta }R_{\alpha \bar{\sigma},\gamma });  \label{biid3}
\end{equation}%
and (by interchanging $\rho $ and $\gamma )$ 
\begin{equation}
R_{\alpha \bar{\beta}\gamma \bar{\sigma},\rho }=\frac{1}{n+2}(R_{\alpha \bar{%
\beta},\rho }h_{\gamma \bar{\sigma}}+R_{\gamma \bar{\beta},\rho }h_{\alpha 
\bar{\sigma}}+\delta _{\alpha }{}^{\beta }R_{\gamma \bar{\sigma},\rho
}+\delta _{\gamma }{}^{\beta }R_{\alpha \bar{\sigma},\rho });  \label{biid4}
\end{equation}%
Subtracting (\ref{biid4}) from (\ref{biid3}) and using the Bianchi
identities (\ref{biid1}), we get 
\begin{equation}
0=\frac{1}{n+2}(R_{\alpha \bar{\beta},\gamma }h_{\rho \bar{\sigma}}+\delta
_{\rho }{}^{\beta }R_{\alpha \bar{\sigma},\gamma }-R_{\alpha \bar{\beta}%
,\rho }h_{\gamma \bar{\sigma}}-\delta _{\gamma }{}^{\beta }R_{\alpha \bar{%
\sigma},\rho }).  \label{biid5}
\end{equation}%
Considering (\ref{biid5}) for $\beta =\gamma $ and taking the sum over $%
\beta $, we have in view of (\ref{biid2}) 
\begin{equation*}
0=\frac{1}{n+2}(-nR_{\alpha \bar{\sigma},\rho }).
\end{equation*}%
That is, $R_{\alpha \bar{\sigma}}$ is paralell and hence vanishing since $N$
is asymptotically flat. This together with (\ref{biid6}) gives the
pseudohermitian curvature $R_{\alpha \bar{\beta}\rho \bar{\sigma}}$ $\equiv $
$0$.
\end{proof}

Take $q_{0}$ $\in $ $N_{\infty }$ $=$ $N\backslash N_{0},$ a simply
connected neighborhood. By Lemma \ref{le02} and Lemma \ref{le03}, we find a
pseudohermitian isomorphism between $N_{\infty }$ and its image $V$ in $%
H_{2}.$ Call $\Psi :V\rightarrow N_{\infty },$ the inverse of this map. Note
that $\Psi $ is an isometry with respect to the adapted (Webster's) metrics $%
L_{\mathring{\theta}}+\mathring{\theta}\otimes \mathring{\theta}$ and $%
L_{\theta }$ $+$ $\theta \otimes \theta $ (recall that $L_{\theta }$ denotes
the Levi metric, cf. (\ref{Levi-0})) respectively$.$ Observe that the
distance between $q_{0}$ and $\infty $ is $\infty $ and so $V$ $\subset $ $%
H_{2}$ must be a neighborhood of $\infty $ by a simple topological argument.
Now extend $\Psi $ to a covering map $\tilde{\Psi}$ $:$ $H_{2}$ $\rightarrow 
$ $N$ via the pseudohermitian development. Note that $V$ is contained in a
fundamental domain. If $\tilde{\Psi}$ is not 1-1, then there are at least
two fundamental domains. But one of them (the one which contains $V)$ has
infinite volume while any other one has finite volume. The contradiction
shows $\tilde{\Psi}$ is 1-1 and a pseudohermitian isomorphism. This
concludes $N\simeq H_{2}$ as pseudohermitian manifolds. We have completed
the proof of Theorem \ref{PMT}.

\begin{proof}
\textbf{(of Corollary \ref{PMT'}) }Consider the blow-up $(N$ $=$ $%
M\backslash \{p\},$ $J,$ $\theta $ $=$ $G_{p}^{2/n}\hat{\theta}).$ By
Proposition \ref{BlmG} (1) $(N,J,\theta )$ is asymptotically flat. It is
obvious that $N=M\backslash \{p\}$ $\subset $ $M$ is spin since $M$ is spin.
From the transformation law (\ref{Lb-1}) it follows that $W_{J,\theta }$ $=$ 
$0$ on $N.$ We can now apply Theorem \ref{PMT} to complete the proof.
\end{proof}

\section{Proof of Theorem \protect\ref{YMP}}

Let $(M^{2n+1},J,\theta )$ be a closed pseudohermitian manifold with the
Tanaka-Webster scalar curvature $W$ $=$ $W_{J,\theta }>0$. For each point $%
p\in M^{2n+1}$, let $G_{p}$ be the Green's function (exists since $%
W_{J,\theta }$ $>$ $0)$ of the CR invariant sublaplacian $L_{b}$ with pole
at $p$, namely $0<G_{p}\in C^{\infty }(M^{2n+1}\setminus \{{p\}})$ such that 
$L_{b}G_{p}=16\delta _{p}$ (\ref{gr1}) where $L_{b}$ $=$ $b_{n}\Delta
_{b}+W, $ $b_{n}=2+\frac{2}{n}$. Recall (\ref{blow-up}) that the blow-up (or
the "generalized Cayley transform") at $p$ is the noncompact pseudohermitian
manifold $(M^{2n+1}\setminus \{{p\}},\ J,$ $G_{p}^{\frac{2}{n}}\theta ).$ By
the transformation law (\ref{Lb-1}), we obtain $W_{J,G_{p}^{\frac{2}{n}%
}\theta }$ $=$ $0$ on $M^{2n+1}\setminus \{{p\}}$. Let $\hat{\theta}$ $=$ $%
G_{p}^{\frac{2}{n}}\theta $ and the volume form $dV_{\hat{\theta}}$ $:=$ $%
\hat{\theta}\wedge (d\hat{\theta})^{n}$. A standard cut-off function
argument implies that 
\begin{equation}
\mathcal{Y}{(M^{2n+1},J)}=\inf_{\phi \in C_{0}^{\infty }(M\setminus \{{p\}})}%
\frac{\int_{M\setminus \{{p\}}}b_{n}|\nabla _{b}^{\hat{\theta}}\phi |^{2}\
dV_{\hat{\theta}}}{(\int_{M\setminus \{{p\}}}|\phi |^{b_{n}}\ dV_{\hat{\theta%
}})^{2/b_{n}}},
\end{equation}%
(see (\ref{YMJ}) for the definition of the CR Yamabe constant $\mathcal{Y}{%
(M^{2n+1},J))}$ where the infimum is taken for $\phi \in C_{0}^{\infty
}(M\setminus \{p\})$ with both numerator and denominator finite. We remark
that by a routine approximation argument, the set of test functions may be
enlarged to consist of positive Lipchitz functions on $M^{2n+1}\setminus
\{p\}$ with both numerator and denominator finite, which is called the set
of \textbf{admissible} test functions. Let $E_{\hat{\theta}}(\phi
):=\int_{M\setminus {p}}b_{n}|\nabla _{b}^{\hat{\theta}}\phi |^{2}\ dV_{\hat{%
\theta}}$. Let $s$ $=$ $b_{n}$ $=$ $2+\frac{2}{n}$ and $||$ $\cdot $ $||_{s}$
denote the $L^{s}$-norm with respect to the volume $dV_{\hat{\theta}}$. We
have the following test function estimate.

\begin{theorem}
\label{esofya} With the notations above, we assume that ($M,J)$ is a closed
spherical CR manifold of dimension $2n+1$. Then for each $n\geq 1$ we
construct a family of test functions $\phi _{\beta }$ such that%
\begin{equation}
E_{\hat{\theta}}(\phi _{\beta })\leq \mathcal{Y}(S^{2n+1},\hat{J})\Vert \phi
_{\beta }\Vert _{2+\frac{2}{n}}^{2}-C_{n}A_{p}\beta ^{-2n}+O(\beta ^{-2n-1})
\label{TFE}
\end{equation}%
\ for $\beta $ large, where $A_{p}$ is the constant in the expansion of the
Green's function $G_{p}$ (see (\ref{exofgr})) and $C_{n}$ is a positive
dimensional constant.
\end{theorem}

\subsection{The constant in the expansion of the Green's function}

By Proposition \ref{BlmG} (1) (with notations $\theta $ and $\hat{\theta}$
switched), $G_{p}^{\frac{2}{n}}\cdot \theta $ is asymptotically flat. Hence
writing $G_{p}^{\frac{2}{n}}\cdot \theta =h^{\frac{2}{n}}\cdot \mathring{%
\theta}$ in asymptotic coordinates $(z,t)$ (recall (\ref{3-36}) rescaled to
absorb $(a_{n}/2\pi )^{2/n})$ for a positive smooth function $h=h(z,t)$, we
have 
\begin{equation}
\begin{split}
h(\infty )& =\lim_{(z,t)\rightarrow \infty }h(z,t)=1 \\
h(z,t)& =1+\frac{a_{n}}{2\pi }A_{p}\cdot \rho ^{-2n}+O(\rho ^{-2n-1}),
\end{split}
\label{h}
\end{equation}%
where $\rho =\rho (z,t)=(|z|^{4}+t^{2})^{1/4}$ and $A_{p}$ is the constant
in the expansion of the Green's function $G_{p}$ (see (\ref{exofgr})). We
would like to remark that the constant $A_{p}$ doesn't depend on the choice
of local coordinates near $p\in M^{2n+1}$.

\subsection{Test function estimate: proof of Theorem \protect\ref{esofya}}

First, let us recall that the family of extremals to Sobolev inequality on
the Heisenberg group $(H^{n},\mathring{\theta})$ (suppressing the obvious CR
structure $\mathring{J})$ is given by the following 
\begin{equation*}
u_{\beta }(z,t)=\beta ^{n}\cdot |\omega +i\beta ^{2}|^{-n}
\end{equation*}%
where $\omega =t+i|z|^{2}$ and $\beta >0$. The family $\{u_{\beta }(z,t):\
\beta >0\}$ satisfy the following CR Yamabe equation on $(H^{n},\mathring{%
\theta})$: 
\begin{equation}
D_{\mathring{\theta}}(u_{\beta }/K_{n})=\mathcal{Y}(S^{2n+1},\hat{J})\cdot
(u_{\beta }/K_{n})^{1+\frac{2}{n}},
\end{equation}%
where $\Vert u_{\beta }\Vert _{s}=K_{n}$ (recall $s$ $=$ $b_{n}$ $=$ $2+%
\frac{2}{n}$), $D_{\mathring{\theta}}$ $=$ $b_{n}\Delta _{b}^{\mathring{%
\theta}}$ $(=$ $L_{b}$ on $(H^{n},\mathring{\theta})$ since $W$ $=$ $0)$ and 
$\mathcal{Y}(S^{2n+1},\hat{J})$ is the CR Yamabe constant of the standard CR
sphere $(S^{2n+1},\hat{J}).$ Recall that $\Vert $ $\cdot $ $\Vert _{s}$
denotes the $L^{s}$-norm with repect to the volume form $dV_{\mathring{\theta%
}}$ $=$ $\mathring{\theta}\wedge (d\mathring{\theta})^{n}$.

We would like to transplant the family $\{u_{\beta }(z,t):\ \beta >0\}$ onto
manifold $M^{2n+1}$ near a point $p$. Let $(z,t)$ be a system of asymptotic
coordinates in $H^{n}\backslash ($big compact set) near $p$ such that $%
(z(p),t(p))$ $=$ $\infty $ and $(z,t)$ is near $\infty $. We consider the
following family of level sets ($\subset H^{n}$) parametrized by $\beta $: 
\begin{equation}
\{(z,t):\ u_{\beta }=\beta ^{n}\cdot |\omega +i\beta ^{2}|^{-n}=\beta
^{-n}\cdot (1+\varepsilon )^{-1}\},  \label{leset}
\end{equation}%
where $\varepsilon =R\cdot \beta ^{-2}$ with $R$ being a fixed large
positive number. Denote the interior of the level set containing $\infty $
by $U_{\beta }(\infty )$. We have the following lemma

\begin{lemma}
\label{keylem1} For $\beta >>R$: we have 
\begin{equation}
\{(z,t):|z|^{2}>\frac{R\gamma_{2}}{2}\}\subset U_{\beta }(\infty)\subset
\{(z,t):\ \rho (z,t)^{2}\geq \frac{\gamma_{1}}{2}R\},
\end{equation}
for some constants $\gamma_{1}, \gamma_{2}, 0<\gamma_{1}\leq\gamma_{2}$,
which are independent of $\beta$.
\end{lemma}

\begin{proof}
Define 
\begin{equation}
f_{\beta }(z,t)=(\frac{t}{\beta ^{2}})^{2}+2|\frac{z}{\beta }|^{2}+|\frac{z}{%
\beta }|^{4},
\end{equation}%
then 
\begin{equation}
u_{\beta }(z,t)=\beta ^{-n}\cdot (1+\varepsilon )^{-1}\Leftrightarrow
f_{\beta }(z,t)=(1+\varepsilon )^{\frac{2}{n}}-1.
\end{equation}%
That is, for each fixed $\beta $, we may rewrite the level set in %
\eqref{leset} in the following form: 
\begin{equation}
\{(z,t):\ f_{\beta }(z,t)=(1+\varepsilon )^{\frac{2}{n}}-1\}.
\end{equation}%
If $\beta \geq \sqrt{R}$, then, by the binomial series, we have 
\begin{equation}
\begin{split}
(1+\varepsilon )^{\frac{2}{n}}-1& =\sum_{k=0}^{\infty }\Big(%
\begin{array}{c}
\frac{2}{n} \\ 
k%
\end{array}%
\Big)\varepsilon ^{k}-1=\sum_{k=1}^{\infty }\Big(%
\begin{array}{c}
\frac{2}{n} \\ 
k%
\end{array}%
\Big)\varepsilon ^{k} \\
& =\beta ^{-2}\cdot R\cdot \sum_{k=1}^{\infty }\Big(%
\begin{array}{c}
\frac{2}{n} \\ 
k%
\end{array}%
\Big)\varepsilon ^{k-1},
\end{split}%
\end{equation}%
where $\Big(%
\begin{array}{c}
\frac{2}{n} \\ 
k%
\end{array}%
\Big)=\frac{\frac{2}{n}(\frac{2}{n}-1)(\frac{2}{n}-2)\cdots (\frac{2}{n}-k+1)%
}{k!}$. We have $(1+\varepsilon )^{\frac{2}{n}}-1=2\varepsilon +\varepsilon
^{2}$ for $n=1$ and, $(1+\varepsilon )^{\frac{2}{n}}-1=\varepsilon $ for $%
n=2 $. On the other hand, since the series $\sum_{k=1}^{\infty }\Big(%
\begin{array}{c}
\frac{2}{n} \\ 
k%
\end{array}%
\Big)\varepsilon ^{k-1}$ is an alternating series for $n\geq 3$, it is easy
to see that 
\begin{equation}
\beta ^{-2}\cdot R\cdot \left[ \Big(%
\begin{array}{c}
\frac{2}{n} \\ 
1%
\end{array}%
\Big)+\Big(%
\begin{array}{c}
\frac{2}{n} \\ 
2%
\end{array}%
\Big)\varepsilon \right] \leq (1+\varepsilon )^{\frac{2}{n}}-1\leq \beta
^{-2}\cdot R\cdot \Big(%
\begin{array}{c}
\frac{2}{n} \\ 
1%
\end{array}%
\Big),
\end{equation}%
which implies that 
\begin{equation}
\beta ^{-2}\cdot R\cdot \frac{n+2}{n^{2}}\leq (1+\varepsilon )^{\frac{2}{n}%
}-1\leq \beta ^{-2}\cdot R\cdot \frac{2}{n}.
\end{equation}%
We hence conclude that 
\begin{equation}
\beta ^{-2}\cdot R\cdot \gamma _{1}\leq (1+\varepsilon )^{\frac{2}{n}}-1\leq
\beta ^{-2}\cdot R\cdot \gamma _{2},  \label{keyin1}
\end{equation}%
for some constants $\gamma _{1},\gamma _{2},0<\gamma _{1}\leq \gamma _{2}$,
which are independent of $\beta $. Now suppose that $(z,t)\in U_{\beta
}(\infty )$, i.e., $f_{\beta }(z,t)>(1+\varepsilon )^{\frac{2}{n}}-1$. In
terms of \eqref{keyin1}, this implies that $\beta ^{4}f_{\beta }(z,t)>\beta
^{2}\cdot R\cdot \gamma _{1}$. We can rewrite this inequality as $\rho
^{4}+2|z|^{2}\beta ^{2}>\beta ^{2}\cdot R\cdot \gamma _{1}$, which implies 
\begin{equation}
\rho ^{4}+2\beta ^{2}\rho ^{2}>\beta ^{2}\cdot R\cdot \gamma _{1}.
\label{keyin2}
\end{equation}%
Inequality \eqref{keyin2} is equivalent to 
\begin{equation}
\rho ^{2}>-\beta ^{2}+\sqrt{\beta ^{4}+R\gamma _{1}\beta ^{2}}.
\label{keyin3}
\end{equation}%
On the other hand, suppose that $a$ is an arbitrary positive constant such
that $\gamma _{1}-2a>0$. We have $(\gamma _{1}-2a)\beta ^{2}>a^{2}R$, for $%
\beta >>1$. This is equivalent to 
\begin{equation}
\beta ^{2}R\gamma _{1}>2aR\beta ^{2}+a^{2}R^{2}.  \label{keyin4}
\end{equation}%
If we add $\beta ^{4}$ on both side of \eqref{keyin4}, we get 
\begin{equation}
\beta ^{4}+\beta ^{2}R\gamma _{1}>(\beta ^{2}+aR)^{2},  \label{keyin5}
\end{equation}%
which, together with \eqref{keyin3}, implies that 
\begin{equation}
\rho ^{2}>aR  \label{keyin6}
\end{equation}%
for any $a$ with $\gamma _{1}-2a>0$. We hence prove that $U_{\beta }(\infty
)\subset \{(z,t):\ \rho (z,t)\geq (\frac{\gamma _{1}}{2}R)^{1/2}\}$.
Finally, in terms of \eqref{keyin1}, we have 
\begin{equation}
\{(z,t):\rho ^{4}+2|z|^{2}\beta ^{2}>R\gamma _{2}\beta ^{2}\}\subset
U_{\beta }(\infty ),
\end{equation}%
which implies 
\begin{equation}
\{(z,t):|z|^{4}+2|z|^{2}\beta ^{2}>R\gamma _{2}\beta ^{2}\}\subset U_{\beta
}(\infty ),
\end{equation}%
or, equivalently, 
\begin{equation}
\{(z,t):|z|^{2}>-\beta ^{2}+\sqrt{\beta ^{4}+R\gamma _{2}\beta ^{2}}%
\}\subset U_{\beta }(\infty ).  \label{keyin7}
\end{equation}%
Define $g(\beta )$ to be the function $g(\beta )=-\beta ^{2}+\sqrt{\beta
^{4}+R\gamma _{2}\beta ^{2}}$. A straightforward computation shows that $%
g^{\prime }(\beta )>0$ for all $\beta >0$ and $g(\beta )\rightarrow \frac{%
R\gamma _{2}}{2}$ as $\beta \rightarrow \infty $. We thus have that $g(\beta
)\leq \frac{R\gamma _{2}}{2}$ for all $\beta >0$. This, together with %
\eqref{keyin7}, shows that 
\begin{equation}
\{(z,t):|z|^{2}>\frac{R\gamma _{2}}{2}\}\subset U_{\beta }(\infty ).
\label{keyin8}
\end{equation}%
We therefore complete the proof of Lemma \ref{keylem1}.
\end{proof}

Since $U_{\beta }(\infty )$ is contained in a fixed neighborhood of $\infty $
for all large $\beta $, we may fix coordinates $(z,t)$ as described earlier
and transplant the family of extremals $\{u_{\beta }(z,t):\ \beta >0\}$ onto
manifold $M^{2n+1}$ near the point $p$ as follows: 
\begin{equation}
\phi _{\beta }=\left\{ 
\begin{array}{ll}
u_{\beta }(z,t) & \ \text{in}\ U_{\beta }(\infty ) \\ 
\beta ^{-n}\cdot (1+\varepsilon )^{-1} & \ \text{elsewhere in}\
M^{2n+1}\setminus p.%
\end{array}%
\right.
\end{equation}%
It is easy to check directly that $\{\phi _{\beta }:\ \beta >0\}$ is a
family of admissible test functions.

Let us use $U_{\beta }(L)$ to denote $U_{\beta }(\infty )\cap B_{L}(0)$.
Note that in $U_{\beta }(\infty )$, we have $\hat{\theta}=h(z,t)^{\frac{2}{n}%
}\cdot \mathring{\theta}(z,t)$ (with $h$ having the expansion in (\ref{h})).
It then follows that 
\begin{equation}
\begin{split}
E_{\hat{\theta}}(\phi _{\beta })& =\int_{M\setminus \{{p\}}}b_{n}|\nabla
_{b}^{\hat{\theta}}\phi _{\beta }|^{2}\ dV_{\hat{\theta}} \\
& =\int_{U_{\beta }(\infty )}b_{n}|\nabla _{b}^{\mathring{\theta}}\phi
_{\beta }|^{2}\cdot h^{2}\ dV_{\mathring{\theta}} \\
& =\lim_{L\rightarrow \infty }\int_{U_{\beta }(L)}b_{n}|\nabla _{b}^{%
\mathring{\theta}}\phi _{\beta }|^{2}\cdot h^{2}\ dV_{\mathring{\theta}}.
\end{split}%
\end{equation}%
Using the divergence theorem, we compute

\begin{equation}
\begin{split}
& \int_{U_{\beta }(L)}b_{n}|\nabla _{b}^{\mathring{\theta}}\phi _{\beta
}|^{2}\cdot h^{2}\ dV_{\mathring{\theta}} \\
=& \int_{U_{\beta }(L)}D_{\mathring{\theta}}(u_{\beta })\cdot u_{\beta
}\cdot h^{2}\ dV_{\mathring{\theta}}-\int_{U_{\beta }(L)}b_{n}\cdot u_{\beta
}\big<\nabla _{b}^{\mathring{\theta}}u_{\beta },\nabla _{b}^{\mathring{\theta%
}}(h^{2})\big>\ dV_{\mathring{\theta}} \\
& +n\int_{\partial U_{\beta }(L)}b_{n}\cdot u_{\beta }(e_{2n}u_{\beta
})\cdot h^{2}\ \mathring{\theta}\wedge (d\mathring{\theta})^{n-1}\wedge
e^{n}.
\end{split}
\label{521}
\end{equation}%
(recall that $D_{\mathring{\theta}}$ $=$ $b_{n}\Delta _{b}^{\mathring{\theta}%
}$) In what follows, we would like to demonstrate that the first term in the
right hand side of \eqref{521} enables us to compare with $\mathcal{Y}%
(S^{2n+1},\hat{J})$; the second term gives us a crucial term and the third
term turns out to be higher order term. For the first term, we have the
following estimate.

\begin{lemma}
\label{keylem2} For all large $L$, it holds that 
\begin{equation*}
\int_{U_{\beta }(L)}D_{\mathring{\theta}}(u_{\beta })\cdot u_{\beta }\cdot
h^{2}\ dV_{\mathring{\theta}}\leq \mathcal{Y}(S^{2n+1},\hat{J})\Vert \phi
_{\beta }\Vert _{s}^{2}
\end{equation*}%
where $\Vert \phi _{\beta }\Vert _{s}$ (recall $s=2+\frac{2}{n})$ denotes
the $L^{s}$-norm with respect to the volume form $dV_{\hat{\theta}}$ $=$ $%
\hat{\theta}\wedge (d\hat{\theta})^{n}$ (recall that $\hat{\theta}$ is the
contact form on $M^{2n+1}\setminus p)$.
\end{lemma}

\begin{proof}
Recall that $K_{n}=\left[ \int_{H^{n}}u_{\beta }^{s}dV_{\Theta }\right]
^{1/s}$ and 
\begin{equation*}
D_{\mathring{\theta}}(u_{\beta }/K_{n})=\mathcal{Y}(S^{2n+1})\cdot (u_{\beta
}/K_{n})^{s-1},
\end{equation*}%
where we recall $s=2+2/n$. Using this identity, and by H\"{o}lder
inequality, we have 
\begin{equation*}
\begin{split}
& \int_{U_{\beta }(L)}D_{\mathring{\theta}}(u_{\beta })\cdot u_{\beta }\cdot
h^{2}dV_{\mathring{\theta}}=\mathcal{Y}(S^{2n+1})K_{n}^{2-s}\int_{U_{\beta
}(L)}u_{\beta }^{s}\cdot h^{2}dV_{\mathring{\theta}} \\
& \leq \mathcal{Y}(S^{2n+1})K_{n}^{2-s}\left[ \int_{U_{\beta }(L)}u_{\beta
}^{s}dV_{\mathring{\theta}}\right] ^{(s-2)/s}\left[ \int_{U_{\beta
}(L)}(u_{\beta }h)^{s}dV_{\mathring{\theta}}\right] ^{2/s} \\
& \leq \mathcal{Y}(S^{2n+1})\left[ \int_{U_{\beta }(L)}u_{\beta }^{s}\cdot
h^{s}dV_{\mathring{\theta}}\right] ^{2/s}\leq \mathcal{Y}(S^{2n+1})\Vert
u_{\beta }\Vert _{s}^{2},
\end{split}%
\end{equation*}%
where $\Vert u_{\beta }\Vert _{s}$ is taken with respect to the contact form 
$\hat{\theta}$ and notice that $dV_{\hat{\theta}}=h^{s}dV_{\mathring{\theta}%
} $. The lemma hence follows.
\end{proof}

For the second term, we have the following estimate.

\begin{lemma}
\label{keylem3} For all large $L$, it holds that%
\begin{equation*}
\int_{U_{\beta }(L)}u_{\beta }\big<\nabla _{b}^{\mathring{\theta}}u_{\beta
},\nabla _{b}^{\mathring{\theta}}(h^{2})\big>\ dV_{\mathring{\theta}}\geq
C_{n}\cdot A_{p}\cdot \beta ^{-2n}+O(\beta ^{-2n-1})
\end{equation*}%
where $C_{n}$ is a positive dimensional constant and $A_{p}$ is the constant
in the expansion of the Green's function $G_{p}$ (\ref{exofgr}).
\end{lemma}

\begin{proof}
Recall that $\mathring{Z}_{\alpha }$ $=$ $\frac{1}{\sqrt{2}}(\frac{\partial 
}{\partial z^{\alpha }}+iz^{\bar{\alpha}}\frac{\partial }{\partial t}),\
\omega =t+i|z|^{2}$ and $h=h(z,t)=1+A_{p}\cdot \rho ^{-2n}+O(\rho ^{-2n-1})$
(\ref{h}). After a straightforward computation, we have 
\begin{equation}
\begin{split}
\mathring{Z}_{\alpha }u_{\beta }& =-\frac{n}{\sqrt{2}}\beta ^{n}\frac{iz^{%
\bar{\alpha}}(\bar{\omega}-i\beta ^{2})}{|\omega +i\beta ^{2}|^{n+2}},\ 
\mathring{Z}_{\bar{\alpha}}\rho ^{-2n}=in\rho ^{-2(n+2)}z^{\alpha }\omega ,
\\
\mathring{Z}_{\bar{\alpha}}h^{2}& =\sqrt{2}A_{p}(\mathring{Z}_{\bar{\alpha}%
}\rho ^{-2n})+O(\rho ^{-2n-2}).
\end{split}
\label{541}
\end{equation}%
Thus 
\begin{equation}
\begin{split}
& u_{\beta }\big<\nabla _{b}^{\mathring{\theta}}u_{\beta },\nabla _{b}^{%
\mathring{\theta}}(h^{2})\big> \\
& =u_{\beta }\sum_{\alpha =1}^{n}(\mathring{Z}_{\alpha }u_{\beta })(%
\mathring{Z}_{\bar{\alpha}}h^{2})+\text{conjugate} \\
& =\left[ 2A_{p}\sum_{\alpha =1}^{n}u_{\beta }(\mathring{Z}_{\alpha
}u_{\beta })(\mathring{Z}_{\bar{\alpha}}\rho ^{-2n})\right] +\sum_{\alpha
=1}^{n}u_{\beta }(\mathring{Z}_{\alpha }u_{\beta })O(\rho ^{-2n-2})+\text{%
conjugate}.
\end{split}
\label{542}
\end{equation}%
Using \eqref{541}, it is easy to check that 
\begin{equation*}
\sum_{\alpha =1}^{n}u_{\beta }(\mathring{Z}_{\alpha }u_{\beta })(\mathring{Z}%
_{\bar{\alpha}}\rho ^{-2n})=\frac{n^{2}}{2}\beta ^{2n}\frac{\rho
^{-2(n+2)}(|\omega |^{2}+\beta ^{2}|z|^{2})|z|^{2}-i(\rho ^{-2(n+2)}\beta
^{2}t|z|^{2})}{|\omega +i\beta ^{2}|^{2n+2}},
\end{equation*}%
and thus, using the non-isotropic scaling $\hat{t}=t/\beta ^{2},\ \hat{z}%
=z/\beta $ and notice that 
\begin{equation*}
dV_{\mathring{\theta}}=dV_{\mathring{\theta}}(z,t)=\beta ^{2n+2}dV_{%
\mathring{\theta}}(\hat{z},\hat{t}),
\end{equation*}%
we have 
\begin{equation}
\begin{split}
& \int_{U_{\beta }(L)}2A_{p}\Big(\sum_{\alpha =1}^{n}u_{\beta }(\mathring{Z}%
_{\alpha }u_{\beta })(\mathring{Z}_{\bar{\alpha}}\rho ^{-2n})+\text{conjugate%
}\Big)dV_{\mathring{\theta}} \\
& =2n^{2}A_{p}\beta ^{-2n}\int_{U_{\beta }(L)}\left( \frac{\hat{\rho}%
^{-2(n+2)}(|\hat{\omega}|^{2}+|\hat{z}|^{2})|\hat{z}|^{2}}{|\hat{\omega}%
+i|^{2n+2}}\right) dV_{\mathring{\theta}}(\hat{z},\hat{t})
\end{split}
\label{543}
\end{equation}%
It is easy to see that the integrand on the right hand side of \eqref{543}
is a function of $O(\hat{\rho}^{-6n-2})$ and that $U_{\beta }(\infty )$
contains a fixed neighborhood of $\infty $ for all large $\beta $ (by Lemma %
\ref{keylem1}), we conclude that, as $L\rightarrow \infty $, 
\begin{equation*}
\int_{U_{\beta }(L)}2A_{p}\Big(\sum_{\alpha =1}^{n}u_{\beta }(\mathring{Z}%
_{\alpha }u_{\beta })(\mathring{Z}_{\bar{\alpha}}\rho ^{-2n})+\text{conjugate%
}\Big)dV_{\mathring{\theta}}\geq \tilde{C}_{n}A_{p}\beta ^{-2n},
\end{equation*}%
for some positive dimensional constant $\tilde{C}_{n}$. It is similar to
check that the integral of the higher decay order of \eqref{542} is of order 
$\beta ^{-2n-1}$, we thus complete the proof of the lemma (with $C_{n}=%
\tilde{C}_{n}\frac{a_{n}}{2\pi })$.
\end{proof}

Finally, we need to estimate the boundary term.

\begin{lemma}
\label{keylem4} It holds that%
\begin{equation*}
\lim_{L\rightarrow \infty }\int_{\partial U_{\beta }(L)}u_{\beta
}(e_{2n}u_{\beta })\cdot h^{2}\ \mathring{\theta}\wedge (d\mathring{\theta}%
)^{n-1}\wedge e^{n}=O(\beta ^{-2n-1}).
\end{equation*}
\end{lemma}

\begin{proof}
We first split the boundary $\partial U_{\beta }(L)$ into inner and outer
parts: 
\begin{equation*}
\partial U_{\beta }(L)=\partial _{1}U_{\beta }(L)\cup \partial _{2}U_{\beta
}(L),
\end{equation*}%
where the inner boundary is 
\begin{equation*}
\partial _{1}U_{\beta }(L)=\{(z,t):(\frac{t}{\beta ^{2}})^{2}+2|\frac{z}{%
\beta }|^{2}+|\frac{z}{\beta }|^{4}=(1+\varepsilon )^{\frac{2}{n}}-1\}
\end{equation*}%
and the outer boundary is 
\begin{equation*}
\partial _{2}U_{\beta }(L)=\{(z,t):t^{2}+|z|^{4}=L^{4}\},
\end{equation*}%
which actually is the Heisenberg sphere with radius $L$. First we estimate
the outer boundary. For each fixed $\beta $, on $\partial _{2}U_{\beta }(L)$%
, we have 
\begin{equation*}
\begin{split}
u_{\beta }(z,t)& =O(L^{-2n}),\ |\nabla _{b}^{\mathring{\theta}}u_{\beta
}(z,t)|=O(L^{-2n-1}), \\
h(z,t)& =1+A_{p}L^{-2n}+O(L^{-2n-1}),
\end{split}%
\end{equation*}%
and hence 
\begin{equation}
|u_{\beta }(e_{2n}u_{\beta })\cdot h^{2}|\text{ }=\text{ }O(L^{-4n-1}).
\label{551}
\end{equation}%
On the other hand, let $d\sigma (\rho )=\mathring{\theta}\wedge (d\mathring{%
\theta})^{n-1}\wedge e^{n}$ be the $p$-area form on the Heisenberg sphere $%
\{z,t):t^{2}+|z|^{4}=\rho ^{4}\}$. Then we have $d\sigma (L)=L^{2n+1}d\sigma
(1)$, which, together with \eqref{551}, implies that 
\begin{equation}
|\int_{\partial U_{\beta }(L)}u_{\beta }(e_{2n}u_{\beta })\cdot h^{2}\ 
\mathring{\theta}\wedge (d\mathring{\theta})^{n-1}\wedge e^{n}|\text{ }=%
\text{ }O(L^{-2n}).  \label{552}
\end{equation}%
Next, we estimate the inner boundary. Let $d\sigma =\mathring{\theta}\wedge
(d\mathring{\theta})^{n-1}\wedge e^{n}$ be the $p$-area form on the inner
boundary. Using the non-isotropic scaling $\hat{t}=t/\beta ^{2},\ \hat{z}%
=z/\beta $, we have 
\begin{equation}
\begin{split}
d\sigma & =\beta ^{2n+1}d\hat{\sigma} \\
u_{\beta }& =\beta ^{n}\cdot |\omega +i\beta ^{2}|^{-n}=\beta ^{-n}\cdot |%
\hat{\omega}+i|^{-n}.
\end{split}
\label{553}
\end{equation}%
Since $|\nabla _{b}^{\mathring{\theta}}\omega |=2|z|$, it follows that 
\begin{equation}
|\nabla _{b}^{\mathring{\theta}}u_{\beta }|\leq c_{n}\beta ^{n}|\omega
+i\beta ^{2}|^{-n-1}=c_{n}\beta ^{-n-1}|\hat{\omega}+i|^{-n-1}|\hat{z}|.
\label{554}
\end{equation}%
From \eqref{553}, \eqref{554}, and notice that $h$ is bounded, we have 
\begin{equation}
\int_{\partial _{1}U_{\beta }(L)}u_{\beta }(e_{2n}u_{\beta })\cdot h^{2}\
d\sigma \leq \int_{\partial _{1}U_{\beta }(L)}c_{n}|\hat{\omega}+i|^{-2n-1}|%
\hat{z}|d\hat{\sigma},  \label{555}
\end{equation}%
here $\partial _{1}U_{\beta }(L)=\{(\hat{z},\hat{t}):\hat{t}^{2}+2|\hat{z}%
|^{2}+|\hat{z}|^{4}=(1+\varepsilon )^{\frac{2}{n}}-1\}$. And, from %
\eqref{keyin1}, it is easy to see that 
\begin{equation}
|\hat{\omega}+i|^{2}=\hat{t}^{2}+2|\hat{z}|^{2}+|\hat{z}|^{4}+1=(1+%
\varepsilon )^{\frac{2}{n}}\leq \frac{R\gamma _{2}}{\beta ^{2}}+1
\label{556}
\end{equation}%
on $\partial _{1}U_{\beta }(L)$. Substituting \eqref{556} into \eqref{555},
we obtain 
\begin{equation}
\begin{split}
\int_{\partial _{1}U_{\beta }(L)}u_{\beta }(e_{2n}u_{\beta })\cdot h^{2}\
d\sigma & \leq c_{n}\left( 1+\frac{R\gamma _{2}}{\beta ^{2}}\right)
^{-(2n+1)/2}\int_{\partial _{1}U_{\beta }(L)}|\hat{z}|d\hat{\sigma} \\
& \leq c_{n}\left( 1+\frac{R\gamma _{2}}{\beta ^{2}}\right)
^{-(2n+1)/2}\int_{\partial _{1}U_{\beta }(L)}|\hat{z}|d\hat{A},
\end{split}
\label{557}
\end{equation}%
where $d\hat{A}$ is the area form with respect to the Riemannian metric
induced from the adapted metric $\frac{1}{2}d\hat{\theta}(\cdot ,J\cdot )$ $%
+ $ $\hat{\theta}^{2}$ of $(M^{2n+1},\hat{\theta})$ and, for the last
inequality, we have used the basic result $d\hat{\sigma}\leq d\hat{A}$. Now,
using the Euclidean dilation $\tilde{t}=\hat{t}\beta ,\ \tilde{z}=\hat{z}%
\beta $, we have 
\begin{equation*}
\begin{split}
d\hat{A}& =\beta ^{-2n}d\tilde{A}, \\
\partial _{1}U_{\beta }(L)& =\left\{ (\tilde{z},\tilde{t}):\tilde{t}^{2}+2|%
\tilde{z}|^{2}+\frac{|\tilde{z}|^{4}}{\beta ^{2}}=\beta ^{2}\left(
(1+\varepsilon )^{\frac{2}{n}}-1\right) \right\} .
\end{split}%
\end{equation*}%
Let $C(R)=\beta ^{2}\left( (1+\varepsilon )^{\frac{2}{n}}-1\right) $. Then,
from \eqref{keyin1}, 
\begin{equation*}
R\cdot \gamma _{1}\leq C(R)\leq R\cdot \gamma _{2},
\end{equation*}%
where $\gamma _{1},\gamma _{2}$ are independent of $\beta $. Therefore 
\begin{equation}
\begin{split}
\int_{\partial _{1}U_{\beta }(L)}|\hat{z}|d\hat{A}& \leq \beta
^{-2n-1}\int_{\left\{ (\tilde{z},\tilde{t}):\tilde{t}^{2}+2|\tilde{z}|^{2}+%
\frac{|\tilde{z}|^{4}}{\beta ^{2}}=C(R)\right\} }|\tilde{z}|d\tilde{A} \\
& =O(\beta ^{-2n-1}),
\end{split}
\label{558}
\end{equation}%
where, for the last equality, we have used the fact that, as $\beta
\rightarrow \infty $ 
\begin{equation*}
\int_{\left\{ (\tilde{z},\tilde{t}):\tilde{t}^{2}+2|\tilde{z}|^{2}+\frac{|%
\tilde{z}|^{4}}{\beta ^{2}}=C(R)\right\} }|\tilde{z}|d\tilde{A}\rightarrow
\int_{\left\{ (\tilde{z},\tilde{t}):\tilde{t}^{2}+2|\tilde{z}%
|^{2}=C(R)\right\} }|\tilde{z}|d\tilde{A},
\end{equation*}%
which is a finite number. Due to \eqref{557} and \eqref{558}, the lemma
follows.
\end{proof}

By Lemma \ref{keylem2}, Lemma \ref{keylem3} and Lemma \ref{keylem4}, we
reduce (\ref{521}) to (\ref{TFE}). We have completed the proof of Theorem %
\ref{esofya}.

\subsection{Proof of Theorem \protect\ref{YMP}}

By Corollary \ref{PMT'} we obtain that the associated p-mass $m\geq 0.$ In
case $m>0,$ we have $A_{p}>0$ by Proposition \ref{BlmG} (2). It follows from
(\ref{TFE}) in Theorem \ref{esofya} that 
\begin{equation*}
\mathcal{Y}{(M,J)<}\mathcal{Y}(S^{5},\hat{J}).
\end{equation*}

\noindent Here we have used $\mathcal{Y}(M,J)\leq E_{\hat{\theta}}(\phi
_{\beta })/\Vert \phi _{\beta }\Vert _{s}^{2}$ by the definition of $%
\mathcal{Y}(M,J)$. Then a fundamental theorem in \cite{JL} tells us that $%
\mathcal{Y}{(M,J)}$ can be attained by a minimizer. In case $m=0,$ we
conclude that $({M,J)}$ is CR equivalent to $(S^{5},\hat{J})$ by Corollary %
\ref{PMT'}. Then the standard contact form on $S^{5}$ is a minimizer to
attain $\mathcal{Y}(S^{5},\hat{J}).$

\section{Examples}

In this section we are going to provide many examples satisfying the
assumption of Corollary \ref{PMT'} and Theorem \ref{YMP}, namely those
closed (compact with no boundary), contact spin 5-manifolds which admit a
spherical CR structure with positive CR Yamabe constant.

\bigskip

\textbf{Example 1}. ($S^{5}/\mathbb{Z}_{p},$ $p:$ odd integer) Let $S^{5}$
denote the unit sphere in $\mathbb{C}^{3}.$ Let $\mathbb{Z}_{p}$ := $\mathbb{%
Z}/p\mathbb{Z}$ denote the finite cyclic group of order $p.$ $\mathbb{Z}_{p}$
acts on $\mathbb{C}^{3}$ through the diagonal matrices:%
\begin{equation*}
\left( 
\begin{array}{ccc}
e^{\frac{2\pi ki}{p}} & 0 & 0 \\ 
0 & e^{\frac{2\pi ki}{p}} & 0 \\ 
0 & 0 & e^{\frac{2\pi ki}{p}}%
\end{array}%
\right) \in U(3)
\end{equation*}

\noindent where $k=0,$ $1,$ $2,$ $...,$ $p-1$ and it induces a free action
on $S^{5}.$ So $S^{5}/\mathbb{Z}_{p}$ is a manifold. The standard contact
structure (or bundle) $\hat{\xi}$ on $S^{5}$ is given by the complex
invariant tangent bundle:%
\begin{equation*}
\hat{\xi}:=TS^{5}\cap J_{\mathbb{C}^{3}}TS^{5}
\end{equation*}

\noindent where $J_{\mathbb{C}^{3}}$ denotes the complex structure of $%
\mathbb{C}^{3}.$ Since $J_{\mathbb{C}^{3}}$ is invariant under the $\mathbb{Z%
}_{p}$-action, $\hat{\xi}$ is invariant under the $\mathbb{Z}_{p}$-action
too. So $S^{5}/\mathbb{Z}_{p}$ is a contact manifold with the induced
contact structure, still denoted as $\hat{\xi}.$

To see whether $S^{5}/\mathbb{Z}_{p}$ is spin, we observe that $H^{2}(S^{5}/%
\mathbb{Z}_{p},\mathbb{Z}_{2})$ $=$ $0$ if $p$ is an odd integer. This fact
can be seen as follows. First we will use the known result below:%
\begin{eqnarray}
H^{k}(S^{5}/\mathbb{Z}_{p},\mathbb{Z)}&=&\mathbb{Z}, \text{ }k=0,5
\label{6-1} \\
&=&0,\text{ }k=1,3  \notag \\
&=&\mathbb{Z}_{p},\text{ }k=2,4.  \notag
\end{eqnarray}

\noindent By Poincare duality we learn that $H_{2}(S^{5}/\mathbb{Z}_{p},%
\mathbb{Z)}$ $\mathbb{\cong }$ $H^{3}(S^{5}/\mathbb{Z}_{p},\mathbb{Z)}$ $=$ $%
0$ and $H_{1}(S^{5}/\mathbb{Z}_{p},\mathbb{Z)}$ $\mathbb{\cong }$ $%
H^{4}(S^{5}/\mathbb{Z}_{p},\mathbb{Z)}$ $=$ $\mathbb{Z}_{p}.$ It follows that%
\begin{eqnarray}
H^{2}(S^{5}/\mathbb{Z}_{p},\mathbb{Z}_{2}\mathbb{)} &\mathbb{\cong }%
&Hom(H_{2}(S^{5}/\mathbb{Z}_{p},\mathbb{Z}_{2}\mathbb{)}\oplus
Ext(H_{1}(S^{5}/\mathbb{Z}_{p},\mathbb{Z)},\mathbb{Z}_{2})  \label{6-2} \\
&=&0\oplus Ext(\mathbb{Z}_{p},\mathbb{Z}_{2})=\mathbb{Z}_{(p,2)}=0  \notag
\end{eqnarray}

\noindent if $p$ is an odd integer (so $(p,2)=1).$ Now the second
Stiefel-Whitney class $w_{2}(T(S^{5}/\mathbb{Z}_{p}))$ $\in $ $H^{2}(S^{5}/%
\mathbb{Z}_{p},\mathbb{Z}_{2}\mathbb{)}=0$ by (\ref{6-2}), so $w_{2}(T(S^{5}/%
\mathbb{Z}_{p}))$ $=$ $0.$ Therefore $S^{5}/\mathbb{Z}_{p}$ is spin when $p$
is an odd integer. We remark that $\mathbb{RP}^{5}$ $:=$ $S^{5}/\mathbb{Z}%
_{2}$ is not spin (see, for instance, Proposition 4.5 on page 235 in \cite%
{Hu}).

The standard CR structure $\hat{J}$ on ($S^{5},\hat{\xi})$ defined by $J_{%
\mathbb{C}^{3}}$ restricted on $\hat{\xi}$ decends to $S^{5}/\mathbb{Z}_{p}$
since the $\mathbb{Z}_{p}$-action preserves $\hat{J}.$ This CR structure on $%
S^{5}/\mathbb{Z}_{p}$ is spherical since $\hat{J}$ is. The standard contact
form $\hat{\theta}$ on $S^{5}$ \cite[p.176]{JL} is invariant under the $%
\mathbb{Z}_{p}$-action. So it decends to $S^{5}/\mathbb{Z}_{p}.$ On the
other hand, the pseudohermitian (or Tanaka-Webster) scalar curvature with
respect to $(\hat{J},\hat{\theta})$ is a positive constant. It follows that
the CR Yamabe constant for $S^{5}/\mathbb{Z}_{p}$ is positive.

\bigskip

\textbf{Example 2.} ($S^{4}\times S_{(a)}^{1},$ $a>1)$ Let $H_{n}$ be the
Heisenberg group (see the first paragraph of the Appendix for the
description). On $H_{n}\backslash \{0\},$ for $a>1$ we define the dilations $%
\tau _{a}$ by%
\begin{equation*}
\tau _{a}(z,t)=(az,a^{2}t)
\end{equation*}

\noindent where $(z,t)\in H_{n},$ $z=(z^{1},$ $...,$ $z^{n})$ $\in $ $%
\mathbb{C}^{n},$ $t\in \mathbb{R}$. Consider the contact form 
\begin{equation*}
\check{\theta}:=\frac{\mathring{\theta}}{\rho ^{2}}
\end{equation*}

\noindent on $H_{n}\backslash \{0\},$ where 
\begin{eqnarray*}
\mathring{\theta} &:=&dt+iz^{\beta }dz^{\bar{\beta}}-iz^{\bar{\beta}%
}dz^{\beta }, \\
\rho &:=&(|z|^{4}+t^{2})^{1/4}
\end{eqnarray*}

\noindent (cf. (\ref{A1}), (\ref{A2})). Observe that $\tau _{a}^{\ast }(%
\check{\theta})=\check{\theta}$ and ($\tau _{a})_{\ast }\mathring{J}$ $=$ $%
\mathring{J}$($\tau _{a})_{\ast }$ (cf. (\ref{A1})), i.e. $\tau _{a}$ is a
pseudohermitian automorphism of $(H_{n}\backslash \{0\},$ $\mathring{J},$ $%
\check{\theta}).$ So the pseudohermitian structure $(\mathring{J},$ $\check{%
\theta})$ decends to the quotient space ($H_{n}\backslash \{0\})/\Gamma _{a}$
where $\Gamma _{a}$ $:=$ $\{...,\tau _{a^{-1}},$ $1,$ $\tau _{a},$ $\tau
_{a^{2}},$ $...\}$. In particular, (($H_{n}\backslash \{0\})/\Gamma _{a},$ $%
\mathring{J})$ is a spherical CR manifold since ($H_{n},$ $\mathring{J})$ is
spherical. Topologically $H_{n}\backslash \{0\}$ $=$ $(0,\infty )$ $\times $ 
$S^{2n}(1)$ where $S^{2n}(1)$ $:=$ $\{\rho =1\}$ $\subset $ $H_{n}.$ For $a$ 
$>$ $1$ each slice $[a^{m-1},$ $a^{m})$ $\times $ $S^{2n}(1)$ is isomorphic
to one another as pseudohermitian manifolds through an element of $\Gamma
_{a}.$ Thus we use $S^{2n}\times S_{(a)}^{1}$ (to indicate the dependence on 
$a;$ topologically $S_{(a)}^{1}$ is the same as $S^{1})$ to denote the
quotient space ($H_{n}\backslash \{0\})/\Gamma _{a}.$

Next we claim that $H^{2}(S^{2n}\times S_{(a)}^{1},\mathbb{Z}_{2})=0$ and
hence $S^{2n}\times S_{(a)}^{1}$ is spin (since $w_{2}(S^{2n}\times
S_{(a)}^{1})$ will then be zero). Noting that $Tor(H^{p}(S^{2n},\mathbb{Z}%
_{2}),H^{q}(S^{1},\mathbb{Z))}$ $=$ $0$ since $H^{q}(S^{1},\mathbb{Z})$ is a
free abelian group for each $q,$ we then have, from the Kunneth formula
(e.g. p.123 in \cite{Vick}),%
\begin{eqnarray}
&&H^{2}(S^{2n}\times S^{1},\mathbb{Z}_{2}\otimes \mathbb{Z})
=\sum_{p+q=2}H^{p}(S^{2n},\mathbb{Z}_{2})\otimes H^{q}(S^{1},\mathbb{Z)}
\label{6-3} \\
&&\ \ \ \ \ =H^{2}(S^{2n},\mathbb{Z}_{2})\otimes H^{0}(S^{1},\mathbb{%
Z)\oplus }H^{1}(S^{2n},\mathbb{Z}_{2})\otimes H^{1}(S^{1},\mathbb{Z)}  \notag
\end{eqnarray}

\noindent in view of $H^{2}(S^{1},\mathbb{Z)=}0.$ For $n\geq 2$ we can
easily obtain $H^{2}(S^{2n},\mathbb{Z}_{2})$ $=$ $H^{1}(S^{2n},\mathbb{Z}%
_{2})$ $=$ $0$ (say, apply the universal coefficient theorem for $G$ $=$ $%
\mathbb{Z}_{2}$ in Theorem 3.14 on page 97 in \cite{Vick}, note that $Ext($%
free abelian, $\cdot $ $)$ $=$ $0$ and use the fact that $H_{2}(S^{2n},%
\mathbb{Z})$ $=$ $H_{1}(S^{2n},\mathbb{Z})$ $=$ $0).$ Substituting this into
(\ref{6-3}) and noting that $\mathbb{Z}_{2}\otimes \mathbb{Z}$ $=$ $\mathbb{Z%
}_{2},$ we obtain $H^{2}(S^{2n}\times S^{1},\mathbb{Z}_{2})$ $=$ $0$ for $%
n\geq 2.$

It is not hard to compute the Tanaka-Webster scalar curvature $W$ for $%
(S^{2n}\times S_{(a)}^{1},J_{0},$ $\check{\theta})$ as follows (cf. p.994 in 
\cite{CMMM}):%
\begin{equation*}
W=\frac{n(n+1)|z|^{2}}{2\rho ^{2}}.
\end{equation*}

\noindent Note that $W$ is nonnegative, but only vanishes on the circle $%
\{(0,t)$ $|$ $t\in \mathbb{R}^{\ast }\}/\Gamma _{a}.$ It follows that the CR
Yamabe constant $\mathcal{Y(}S^{2n}\times S_{(a)}^{1},\mathring{J})$ $>$ $0$
for any $n$ $\in $ $\mathbb{N}.$ Otherwise we have $\mathcal{Y(}S^{2n}\times
S_{(a)}^{1},\mathring{J})$ $=$ $0.$ Then by solving the Yamabe minimizer
problem (\cite{JL}), we can find (smooth) $u$ $>$ $0$ such that%
\begin{equation*}
(2+\frac{2}{n})\Delta _{b}u+Wu=0
\end{equation*}%
\noindent (note that our $\Delta _{b}$ is the negative sublaplacian).
Multiplying the above equation by $u$ and integrating give%
\begin{equation*}
\int_{S^{2n}\times S_{(a)}^{1}}[(2+\frac{2}{n})|\nabla _{b}u|^{2}+Wu^{2}]%
\check{\theta}\wedge (d\check{\theta})^{n}=0.
\end{equation*}%
\noindent It follows that $u\equiv 0,$ a contradiction. Altogether we have
shown that $S^{4}\times S_{(a)}^{1},$ $a>1,$ is a contact spin $5$-manifold
which admits a spherical CR structure $\mathring{J}$ with positive CR Yamabe
constant.

\bigskip

\textbf{Example 3 }($\mathbb{RP}^{5}$ $\sharp $ $\mathbb{RP}^{5})$ Although $%
\mathbb{RP}^{5}$ $:=$ $S^{5}/\mathbb{Z}_{2}$ is not spin as remarked in
Example 1, the connected sum of two copies of $\mathbb{RP}^{5}$ is indeed
spin. This fact can be seen as follows. First we define a $\mathbb{Z}_{2}$%
-action $\tau $ (identified with $\tau (1))$ on $S^{4}\times S^{1}$ by%
\begin{equation*}
\tau :((x_{1},...,x_{5}),z)\rightarrow ((-x_{1},...,-x_{5}),\bar{z})
\end{equation*}

\noindent where the coordinates are given in view of $S^{4}$ $\subset $ $%
\mathbb{R}^{5}$ and $S^{1}$ $\subset $ $\mathbb{C}$ as the unit sphere and
the unit circle respectively. It is not hard to see that topologically or
differentiably $\mathbb{RP}^{5}$ $\sharp $ $\mathbb{RP}^{5}$ has a $S^{1}$
fibration over $\mathbb{RP}^{4},$ which is the same as ($S^{4}\times
S^{1})/\tau .$ We compute the tangent bundle%
\begin{eqnarray}
T(\mathbb{RP}^{5}\sharp \mathbb{RP}^{5}) &\simeq &T((S^{4}\times S^{1})/\tau
)  \label{6-4} \\
&\simeq &T(S^{4}\times S^{1})/\tau .  \notag
\end{eqnarray}

\noindent On the other hand, $T(S^{4}\times S^{1})$ $\simeq $ $TS^{4}\oplus
TS^{1}$ $\simeq $ $TS^{4}\oplus \mathcal{E}^{1}$ $\simeq $ $\mathcal{E}^{5}$
where $\mathcal{E}^{1}$ and $\mathcal{E}^{5}$ denote the trivial bundle of
rank 1 and rank 5 respectively. Together with (\ref{6-4}) we conclude that $%
T(\mathbb{RP}^{5}\sharp \mathbb{RP}^{5})$ is trivial. Therefore $\mathbb{RP}%
^{5}\sharp \mathbb{RP}^{5}$ is spin. It is clear that $\mathbb{RP}^{5}$ is a
spherical CR 5-manifold with CR Yamabe constant $\mathcal{Y(}\mathbb{RP}%
^{5}) $ $>$ 0 since its Tanaka-Webster scalar curvature is a positive
constant, same as the standard pseudohermitian $S^{5}.$

To show that $\mathbb{RP}^{5}\sharp \mathbb{RP}^{5}$ in Example 3 is still a
spherical CR 5-manifold with $\mathcal{Y(}\mathbb{RP}^{5}\sharp \mathbb{RP}%
^{5})$ $>$ $0,$ we employ the following theorem:

\begin{theorem}
\label{connect} (\cite{CChiu}) Suppose $(M_{1},J_{1})$ and $(M_{2},J_{2})$
are two closed (compact with no boundary), spherical CR manifolds with $%
\mathcal{Y(}M_{k},J_{k})$ $>$ $0$ for $k=1,$ $2.$ Then their connected sum $%
M_{1}\#M_{2}$ admits a spherical CR structure $\tilde{J}$ with $\mathcal{Y(}%
M_{1}\#M_{2},\tilde{J})$ $>$ $0.$
\end{theorem}

We remark that Theorem \ref{connect} still holds without the constraint on
CR manifolds being spherical (see \cite{CCH}). On the other hand, we have
the following fact about the connected sum of two spin manifolds:

\begin{remark}
\label{spin-conn} (\cite{Mil}; Remark 2.17 on p.91 in \cite{LM}) Given two
spin manifolds $M_{1}$ and $M_{2},$ we can equip their connected sum $%
M_{1}\#M_{2}$ with a spin structure (so that $M_{1}\#M_{2}$ and the disjoint
union $M_{1}\amalg M_{2}$ are spin cobordant).
\end{remark}

According to Theorem \ref{connect} and Remark \ref{spin-conn}, we conclude
the following result:

\begin{proposition}
\label{prop-ex} The connected sum of finitely many (duplication allowed)
5-manifolds chosen arbitrarily from the set consisting of $S^{5}/\mathbb{Z}%
_{p},$ $p:$ odd integer, $S^{4}\times S_{(a)}^{1},$ $a>1$ and $\mathbb{RP}%
^{5}$ $\sharp $ $\mathbb{RP}^{5}$ in Examples 1, 2 and 3 above is still a
closed, contact spin 5-manifold which admit a spherical CR structure with
positive CR Yamabe constant.
\end{proposition}

\section{Appendix: Some basic materials in pseudohermitian geometry}

We introduce some basic materials in pseudohermitian geometry. Some formulas
are used to deduce the Weizenbock formula\textbf{\ }in Section \ref{SSCB}.
We refer the reader to N. Tanaka \cite{Ta} and S. Webster \cite{We}.

Let $(M^{2n+1},\xi )$ denote a contact manifold with a coorientable (i.e. $%
TM/\xi $ is trivial) contact structure (or bundle) $\xi .$ A CR manifold $%
(M^{2n+1},\xi ,J)$ or $(M^{2n+1},J)$ (with $\xi $ suppressed) is a contact
manifold $(M^{2n+1},\xi )$ equipped with an almost complex structure, i.e.
an endomorphism $J:\xi \rightarrow \xi $ defined on $\xi $ such that $%
J^{2}=-1$. The endomorphism $J$ decomposes the complexification of $\xi $
into the direct sum of bundles of holomorphic vectors and anti-holomorphic
vectors $\xi \otimes \mathbb{C}=\xi _{1,0}\oplus \xi _{0,1}.$ We assume that 
$J$ is integrable, that is, $J$ satisfies the formal Frobenius condition $%
[\xi _{1,0},\xi _{1,0}]\subset \xi _{1,0}$ (as sections). A contact form $%
\theta $ is a global one-form such that $\xi =\ker {\theta }$ (exists by
coorientation of $\xi )$. A pseudohermitian manifold $(M^{2n+1},J,\theta )$
(with $\xi $ suppressed) is a contact manifold with a choice of CR structure 
$J$ together with a choice of contact form $\theta $. The Levi metric $%
L_{\theta }$ is defined by 
\begin{equation}
L_{\theta }(X,Y):=\frac{1}{2}d\theta (X,JY)\text{ for all }X,Y\in \xi
\label{Levi-0}
\end{equation}
(we use the convention that $\eta \wedge \vartheta (V,W)$ = $\eta
(V)\vartheta (W)$ $-$ $\eta (W)\vartheta (V)$ for 1-forms $\eta ,$ $%
\vartheta ,$ vectors $V,$ $W$). Let $T$ denote the Reeb vector field
associated to $\theta ,$ the unique vector field such that $\theta (T)$ $=$ $%
1$ and $L_{T}\theta $ $=$ $0$ ($L_{T}$ means the Lie derivative in the
direction $T).$For a choice of (admissible) coframe $\theta ^{\alpha }$ with 
$\theta ^{\alpha }(T)$ $=$ $0,$ we have the Levi equation%
\begin{equation}
d\theta =ih_{\alpha \bar{\beta}}\theta ^{\alpha }\wedge \theta ^{\bar{\beta}%
}.  \label{Levi}
\end{equation}

In 1978, S. Webster \cite{We} (cf. an equivalent formulation in \cite{Ta} by
N. Tanaka) showed that there is a natural connection in the bundle $\xi
_{1,0}$ adapted to a pseudohermitian structure $(J,\theta )$. Locally, there
exist unique $1$-forms $\theta _{\alpha }{}^{\beta }$ (connection forms)$,\
\tau ^{\beta }$ (torsion forms) satisfying the structure equations%
\begin{eqnarray}
d\theta ^{\beta } &=&\theta ^{\alpha }\wedge \theta _{\alpha }\text{ }%
^{\beta }+\theta \wedge \tau ^{\beta },  \label{SE} \\
0 &=&\theta _{\alpha }\text{ }^{\beta }+\theta _{\bar{\beta}}\text{ }^{\bar{%
\alpha}},\text{ }0=\tau _{\beta }\wedge \theta ^{\beta }  \notag
\end{eqnarray}%
where \{$\theta ^{\beta }\}$ is a unitary coframe (meaning $h_{\alpha \bar{%
\beta}}$ $=$.$\delta _{\alpha \beta })$. Let $\{Z_{\beta }\}$ denote a
unitary frame of $\xi _{1,0}$ dual to \{$\theta ^{\beta }\}$. These forms $%
\theta _{\alpha }{}^{\beta }$ satisfy the transformation law of connection
forms, so we can use them to define a connection. Let $T$ denote the Reeb
vector field associated to $\theta ,$ the unique vector field such that $%
\theta (T)$ $=$ $1$ and $L_{T}\theta $ $=$ $0$ ($L_{T}$ means the Lie
derivative in the direction $T).$ The \textbf{pseudohermitian }(or
Tanaka-Webster)\textbf{\ connection} $\nabla ^{p.h.}$ is defined by 
\begin{equation}
\begin{split}
\nabla ^{p.h.}{Z_{\alpha }}& =\theta _{\alpha }{}^{\beta }\otimes Z_{\beta }
\\
\nabla ^{p.h.}{Z_{\bar{\alpha}}}& =\theta _{\bar{\alpha}}{}^{\bar{\beta}%
}\otimes Z_{\bar{\beta}} \\
\nabla ^{p.h.}T& =0.
\end{split}
\label{Con_c}
\end{equation}

\noindent Differentiate the connection to define the curvature: $d\theta
_{\alpha }{}^{\beta }$ $-$ $\theta _{\alpha }{}^{\gamma }\wedge \theta
_{\gamma }{}^{\beta }$ $=$ $R_{\alpha }{}^{\beta }{}_{\rho \bar{\sigma}%
}\theta ^{\rho }\wedge \theta ^{\bar{\sigma}}$ $+\ $terms including the
torsion. The pseudohermitian-Ricci tensor is the hermitian form on $\xi
_{1,0}$ defined by 
\begin{equation*}
\rho (X,Y)=R_{\alpha \bar{\beta}}X^{\alpha }Y^{\bar{\beta}},
\end{equation*}%
where $X=X^{\alpha }Z_{\alpha },Y=Y^{\beta }Z_{\beta }$ and $R_{\alpha \bar{%
\beta}}=R_{\gamma }{}^{\gamma }{}_{\alpha \bar{\beta}}$. The Tanaka-Webster
scalar curvature is 
\begin{equation}
W:=R_{\beta }{}^{\beta },  \label{Wsc}
\end{equation}%
\noindent which is the contraction of the pseudohermitian-Ricci tensor. We
can also have "real formulation" for the pseudohermitian structure $%
(J,\theta )$. Write $\theta ^{\alpha }$ $=$ $\omega ^{\alpha }+i\omega
^{n+\alpha }$ for real coframe fields $\{\omega ^{1},$ $..,$ $\omega ^{n},$ $%
\omega ^{n+1},$ $..,$ $\omega ^{2n}\}$ and $Z_{\alpha }$ $=$ $\frac{1}{2}%
(e_{\alpha }-ie_{n+\alpha })$ for real frame fields $\{e_{1},$ $..,$ $e_{n},$
$e_{n+1},$ $..,$ $e_{2n}\}$ (orthonormal with respect to the Levi metric $%
L_{\theta }).$ It is easily seen that $\{\omega ^{A}\}_{A=1,..,2n}$ is dual
to $\{e_{A}\}_{A=1,..,2n}.$ Write%
\begin{equation}
\nabla ^{p.h.}e_{A}=\omega _{A}\text{ }^{B}e_{B}  \label{Con_r}
\end{equation}

\noindent for real connection forms $\omega _{A}$ $^{B},$ $1$ $\leq $ $A,$ $%
B $ $\leq $ $2n$. Comparing (\ref{Con_c}) with (\ref{Con_r}) gives%
\begin{eqnarray}
\theta _{\alpha }{}^{\beta } &=&\omega _{\alpha }\text{ }^{\beta }+i\omega
_{\alpha }\text{ }^{n+\beta }\text{ and}  \label{Sym_r} \\
\omega _{\alpha }\text{ }^{n+\beta } &=&-\omega _{n+\alpha }\text{ }^{\beta }%
\text{, }\omega _{\alpha }\text{ }^{\beta }=\omega _{n+\alpha }\text{ }%
^{n+\beta }.  \notag
\end{eqnarray}

\noindent From the condition $0$ $=$ $\theta _{\alpha }$ $^{\beta }+\theta _{%
\bar{\beta}}$ $^{\bar{\alpha}}$ in (\ref{SE}) and (\ref{Sym_r}), it follows
that%
\begin{equation*}
\omega _{A}\text{ }^{B}+\omega _{B}\text{ }^{A}=0,\text{ }1\leq A,B\leq 2n.
\end{equation*}

\noindent Note that if we denote the scalar curvature associated to $\omega
_{A}$ $^{B}$ by $R,$ then we have%
\begin{equation}
W=\frac{1}{4}R.  \label{SC}
\end{equation}

\noindent Let $u_{\alpha \beta }$ denote the second covariant derivative of
a function $u$ in the directions $Z_{\alpha },Z_{\beta }.$ Define the
subgradient $\nabla _{b}$ and the sublaplacian $\Delta _{b}$ (or $\nabla
_{b}^{\theta }$ and $\Delta _{b}^{\theta }$ to indicate the dependence on $%
\theta )$ by%
\begin{equation*}
\nabla _{b}u:=u^{\alpha }Z_{\alpha }+u^{\bar{\alpha}}Z_{\bar{\alpha}},
\end{equation*}%
\begin{equation}
\Delta _{b}u:=-(u_{\alpha }\text{ }^{\alpha }+u_{\bar{\alpha}}\text{ }^{\bar{%
\alpha}})  \label{SubL}
\end{equation}%
\noindent (notice the negative sign) where $u_{\alpha }$ $^{\alpha }$ :$=$ $%
u_{\alpha \bar{\beta}}h^{\alpha \bar{\beta}}$ $=$ $u_{\alpha \bar{\alpha}}$
for a unitary frame ($h^{\alpha \bar{\beta}}$ $=$ ($h_{\alpha \bar{\beta}%
})^{-1}$ $=$ $\delta _{\alpha \beta })$. Define the $CR$ invariant
sublaplacian $L_{b}$ by%
\begin{equation}
L_{b}:=b_{n}\Delta _{b}+W,\text{ }b_{n}=2+\frac{2}{n}.  \label{Lb}
\end{equation}%
\noindent Consider a new contact form $\hat{\theta}=u^{2/n}\theta $ for a
smooth positive function $u.$ $L_{b}$ rules the change of the Tanaka-Webster
scalar curvature:%
\begin{equation}
L_{b}u=\hat{W}u^{1+\frac{2}{n}}  \label{Lb-1}
\end{equation}%
\noindent where $\hat{W}$ is the Tanaka-Webster scalar curvature with
respect to $(J,\hat{\theta}).$ The Green's function $G_{p}$ of $L_{b}$ at $p$
satisfies%
\begin{equation}
L_{b}G_{p}=16\delta _{p}  \label{Lb-2}
\end{equation}
\noindent where $\delta _{p}$ is the delta function w.r.t. the volume form $%
dV_{\theta }$ $:=$ $\theta \wedge (d\theta )^{n}.$ We define the CR Yamabe
constant $\mathcal{Y}(M,J)$ as follows:%
\begin{equation}
\mathcal{Y}(M,J):=\inf_{\hat{\theta}}\frac{\int_{M}\hat{W}dV_{\hat{\theta}}}{%
(\int_{M}dV_{\hat{\theta}})^{\frac{n}{n+1}}}=\inf_{0<u\in C^{\infty }(M)}%
\frac{\int_{M}(b_{n}|\nabla _{b}u|^{2}+Wu^{2})dV_{\theta }}{%
(\int_{M}u^{b_{n}}dV_{\theta })^{\frac{2}{b_{n}}}}  \label{YMJ}
\end{equation}

\noindent where $|\nabla _{b}u|^{2}:=2h^{\alpha \bar{\beta}}u_{\alpha }u_{%
\bar{\beta}}.$ Given a background $W$ with respect to $(J,\theta ),$ we aim
to find a solution $u$ to (\ref{Lb-1}) with $\hat{W}$ $=$ constant, say $1$
\ This is the so called Yamabe problem. The CR Yamabe equation (with
critical Sobolev exponent) for $\hat{W}=1$ reads as follows: 
\begin{equation}
b_{n}\Delta _{b}u+Wu=u^{1+\frac{2}{n}}.  \label{YE}
\end{equation}

The structure equations imply ($h_{\alpha \bar{\beta}}$ $=$ $\delta _{\alpha
\beta }$ for a unitary (co)frame)%
\begin{equation}
\begin{split}
\lbrack Z_{\bar{\beta}},Z_{\alpha }]& =ih_{\alpha \bar{\beta}}T+\theta
_{\alpha }{}^{\gamma }(Z_{\bar{\beta}})Z_{\gamma }-\theta _{\bar{\beta}}{}^{%
\bar{\gamma}}(Z_{\alpha })Z_{\bar{\gamma}}, \\
\lbrack Z_{\beta },Z_{\alpha }]& =\theta _{\alpha }{}^{\gamma }(Z_{\beta
})Z_{\gamma }-\theta _{\beta }{}^{\gamma }(Z_{\alpha })Z_{\gamma }, \\
\lbrack Z_{\alpha },T]& =A^{\bar{\gamma}}{}_{\alpha }Z_{\bar{\gamma}}-\theta
_{\alpha }{}^{\gamma }(T)Z_{\gamma }.
\end{split}
\label{ide1}
\end{equation}

\noindent where we have written the torsion (forms) $\tau ^{\beta }$ $=$ $%
A^{\beta }$ $_{\bar{\alpha}}\theta ^{\bar{\alpha}}$ and $A^{\bar{\gamma}}$ $%
_{\alpha }$ $=$ $\overline{A^{\gamma }\text{ }_{\bar{\alpha}}}.$ Let $L_{T}$
denote the Lie differentiation in the direction $T.$ From (\ref{SE}) and the
third equality in (\ref{ide1}), it follows that%
\begin{equation}
L_{T}J=2iA^{\beta }\text{ }_{\bar{\alpha}}\theta ^{\bar{\alpha}}\otimes
Z_{\beta }-2iA^{\bar{\beta}}\text{ }_{\alpha }\theta ^{\alpha }\otimes Z_{%
\bar{\beta}}.  \label{LTJ}
\end{equation}

As a flat pseudohermitian manifold, the Heisenberg group plays an important
role in pseudohermitian geometry. We refer the reader to \cite{CC} and \cite%
{CL} for the details about the Heisenberg group, and to \cite{DT},\cite{Le1},%
\cite{Le2} and \cite{We} for pseudohermitian geometry. Denote by $H_{n}$ the
Heisenberg group, which is the space $\mathbb{R}^{2n+1}$ with coordinates $%
(x_{\beta },y_{\beta },t)$ as a set. It is a $(2n+1)$-dimensional Lie group
with group structure defined by 
\begin{equation*}
(x,y,t)\circ (x^{\prime },y^{\prime },t^{\prime })=(x+x^{\prime
},y+y^{\prime },t+t^{\prime }+2yx^{\prime }-2xy^{\prime }).
\end{equation*}%
The associated Lie algebra is spanned by the following left invariant vector
fields 
\begin{equation}
\mathring{e}_{\beta }:=\frac{1}{\sqrt{2}}\left( \frac{\partial }{\partial
x^{\beta }}+2y^{\beta }\frac{\partial }{\partial t}\right) ,\ \ \mathring{e}%
_{n+\beta }:=\frac{1}{\sqrt{2}}\left( \frac{\partial }{\partial y^{\beta }}%
-2x^{\beta }\frac{\partial }{\partial t}\right) ,\ \ \mathring{T}:=\frac{%
\partial }{\partial t}.  \label{7-8a}
\end{equation}%
The associated standard CR structure $\mathring{J}$ and contact form $%
\mathring{\theta}$ (or denoted by $\Theta $) are defined respectively by%
\begin{eqnarray}
\mathring{J}\mathring{e}_{\beta } &=&\mathring{e}_{n+\beta },\text{ }%
\mathring{J}\mathring{e}_{n+\beta }=-\mathring{e}_{\beta },  \label{A1} \\
\mathring{\theta} &=&dt+\sum_{\beta =1}^{n}(iz^{\beta }dz^{\bar{\beta}}-iz^{%
\bar{\beta}}dz^{\beta }).  \notag
\end{eqnarray}%
Here $z^{\beta }$ $:=$ $x^{\beta }+iy^{\beta }$. The contact bundle is $%
\mathring{\xi}$ $:=\text{ ker }\mathring{\theta}$. We linearly extend $%
\mathring{J}:\mathring{\xi}\otimes \mathbb{C}\rightarrow \mathring{\xi}%
\otimes \mathbb{C}$. Let $\mathring{Z}_{\beta }:=\frac{1}{2}(\mathring{e}%
_{\beta }-i\mathring{e}_{n+\beta })=\frac{1}{\sqrt{2}}\left( \frac{\partial 
}{\partial z^{\beta }}+iz^{\bar{\beta}}\frac{\partial }{\partial t}\right) .$
Then for all $\beta ,\gamma ,$ $\mathring{J}\mathring{Z}_{\beta }=i\mathring{%
Z}_{\beta },\ \mathring{J}\mathring{Z}_{\bar{\beta}}=-i\mathring{Z}_{\bar{%
\beta}}$ and $[\mathring{Z}_{\beta },\mathring{Z}_{\gamma }]=0\ \left(
\Rightarrow \lbrack \mathring{\xi}_{1,0},\mathring{\xi}_{1,0}]\subset 
\mathring{\xi}_{1,0}\right) $ where $\mathring{\xi}\otimes \mathbb{C}=%
\mathring{\xi}_{1,0}\oplus \mathring{\xi}_{0,1}.$ It is easily seen that the
frame $\{\mathring{T},\mathring{Z}_{\beta },\mathring{Z}_{\bar{\beta}}\}$ is
dual to the coframe $\{\mathring{\theta},\sqrt{2}dz^{\beta },\sqrt{2}dz^{%
\bar{\beta}}\}$. If we regard $\{\mathring{e}_{\beta },\mathring{e}_{n+\beta
}|1\leq \beta \leq n\}$ as an orthonormal basis, then this defines a metric
on $\mathring{\xi}$, which equals the Levi metric $L_{\mathring{\theta}}$
given by $L_{\mathring{\theta}}(X,Y)=\frac{1}{2}d\mathring{\theta}(X,%
\mathring{J}Y)$ for all $X,Y\in \mathring{\xi}$. The standard
pseudohermitian connection on $H_{n}$ is defined by 
\begin{equation*}
\mathring{\nabla}^{p.h.}\mathring{e}_{\beta }=\mathring{\nabla}^{p.h.}%
\mathring{e}_{n+\beta }=\mathring{\nabla}^{p.h.}\mathring{T}=0.
\end{equation*}%
It follows that the pseudohermitian connection forms $\mathring{\theta}%
_{\alpha }$ $^{\gamma }$ vanish:%
\begin{equation}
\mathring{\theta}_{\alpha }\text{ }^{\gamma }=0  \label{A1-1}
\end{equation}%
We define the Heisenberg norm $\rho $ on $H_{n}$ by%
\begin{equation}
\rho ^{4}=(|z|^{4}+t^{2})  \label{A2}
\end{equation}%
where $|z|^{2}=\sum_{\beta =1}^{n}|z^{\beta }|^{2}$.

Define the torsion tensor field of the pseudohermitian connection $\nabla
^{p.h.}$ by 
\begin{equation*}
\mathbb{T}(V,U)=\nabla _{V}^{p.h.}U-\nabla _{U}^{p.h.}V-[V,U],
\end{equation*}%
for all complex vector fields $V,U$. Then (\ref{ide1}) implies 
\begin{equation}
\begin{split}
\mathbb{T}(Z_{\alpha },Z_{\bar{\beta}})& =ih_{\alpha \bar{\beta}}T, \\
\mathbb{T}(Z_{\alpha },Z_{\beta })& =0, \\
\mathbb{T}(Z_{\alpha },T)& =-A^{\bar{\beta}}{}_{\alpha }Z_{\bar{\beta}}.
\end{split}
\label{ide2}
\end{equation}%
The real version of (\ref{ide2}) is 
\begin{equation*}
\begin{split}
\mathbb{T}(e_{\beta },e_{n+\beta })& =2T,\ \ \mathbb{T}(e_{n+\beta
},e_{\beta })=-2T,\ \ \mathbb{T}(e_{a},e_{b})=0,\ \ \ \text{otherwise}; \\
\mathbb{T}(e_{\gamma },T)& =-(ReA^{\bar{\beta}}{}_{\gamma })e_{\beta }+(ImA^{%
\bar{\beta}}{}_{\gamma })e_{n+\beta } \\
\mathbb{T}(e_{n+\gamma },T)& =(ImA^{\bar{\beta}}{}_{\gamma })e_{\beta
}+(ReA^{\bar{\beta}}{}_{\gamma })e_{n+\beta }
\end{split}%
\end{equation*}%
Recall that the curvature operator is defined by 
\begin{equation*}
R_{XY}^{p.h.}=\nabla _{X}^{p.h.}\nabla _{Y}^{p.h.}-\nabla _{Y}^{p.h.}\nabla
_{X}^{p.h.}-\nabla _{\lbrack X,Y]}^{p.h.}.
\end{equation*}%
And the Ric Tensor is defined by 
\begin{equation}
Ric^{p.h.}(X,Y)=-\big<R_{e_{a}X}^{p.h.}(e_{a}),Y\big>
\end{equation}%
We have the following Bianchi identity

\begin{lemma}[\textbf{{Bianchi identity}}]
\begin{equation}  \label{ide3}
\begin{split}
&R^{p.h.}_{XY}(Z)+R^{p.h.}_{YZ}(X)+R^{p.h.}_{ZX}(Y) \\
=&\mathbb{T}(X,[Y,Z])+\mathbb{T}(Y,[Z,X])+\mathbb{T}(Z,[X,Y]) \\
&+\nabla^{p.h.}_{X}(\mathbb{T}(Y,Z))+\nabla^{p.h.}_{Y}(\mathbb{T}%
(Z,X))+\nabla^{p.h.}_{Z}(\mathbb{T}(X,Y)).
\end{split}%
\end{equation}
In particular, if all $X,Y,Z$ are horizontal, then we have 
\begin{equation}  \label{ide4}
\begin{split}
&R^{p.h.}_{XY}(Z)+R^{p.h.}_{YZ}(X)+R^{p.h.}_{ZX}(Y) \\
=&\mathbb{T}(X,[Y,Z])+\mathbb{T}(Y,[Z,X])+\mathbb{T}(Z,[X,Y]),\ \ \text{mod}%
\ T.
\end{split}%
\end{equation}
\end{lemma}

\begin{proof}
The formula (\ref{ide4}) follows from (\ref{ide2}) and (\ref{ide3}). Now we
prove formula (\ref{ide3}). We have 
\begin{equation}  \label{ide5}
\begin{split}
[X,[Y,Z]]&=\nabla^{p.h.}_{X}[Y,Z]-\nabla^{p.h.}_{[Y,Z]}X-\mathbb{T}(X,[Y,Z])
\\
&=\nabla^{p.h.}_{X}\nabla^{p.h.}_{Y}Z-\nabla^{p.h.}_{X}\nabla^{p.h.}_{Z}Y-%
\nabla^{p.h.}_{X}\left(\mathbb{T}(Y,Z)\right) \\
&\ \ \ \ -\nabla^{p.h.}_{[Y,Z]}X-\mathbb{T}(X,[Y,Z]);
\end{split}%
\end{equation}
Similarly, we have 
\begin{equation}  \label{ide6}
\begin{split}
[Y,[Z,X]]&=\nabla^{p.h.}_{Y}\nabla^{p.h.}_{Z}X-\nabla^{p.h.}_{Y}%
\nabla^{p.h.}_{X}Z-\nabla^{p.h.}_{Y}\left(\mathbb{T}(Z,X)\right) \\
&\ \ \ \ -\nabla^{p.h.}_{[Z,X]}Y-\mathbb{T}(Y,[Z,X])
\end{split}%
\end{equation}
and 
\begin{equation}  \label{ide7}
\begin{split}
[Z,[X,Y]]&=\nabla^{p.h.}_{Z}\nabla^{p.h.}_{X}Y-\nabla^{p.h.}_{Z}%
\nabla^{p.h.}_{Y}X-\nabla^{p.h.}_{Z}\left(\mathbb{T}(X,Y)\right) \\
&\ \ \ \ -\nabla^{p.h.}_{[X,Y]}Z-\mathbb{T}(Z,[X,Y]).
\end{split}%
\end{equation}
Taking the sum of (\ref{ide5}), (\ref{ide6}) and (\ref{ide7}), we get 
\begin{equation*}
\begin{split}
0&=[X,[Y,Z]]+[Y,[Z,X]]+[Z,[X,Y]] \\
&=R^{p.h.}_{XY}(Z)+R^{p.h.}_{YZ}(X)+R^{p.h.}_{ZX}(Y) \\
&\ \ \ \ -\mathbb{T}(X,[Y,Z])-\mathbb{T}(Y,[Z,X])-\mathbb{T}(Z,[X,Y]) \\
&\ \ \ \ -\nabla^{p.h.}_{X}(\mathbb{T}(Y,Z))-\nabla^{p.h.}_{Y}(\mathbb{T}%
(Z,X))-\nabla^{p.h.}_{Z}(\mathbb{T}(X,Y)).
\end{split}%
\end{equation*}
This completes the proof of (\ref{ide3}).
\end{proof}


\end{document}